%% file: loopHecke.tex
\documentclass[11pt,reqno]{amsart}

\input{macros-etal/packages-etc}

\input{macros-etal/defs.tex}

\input{macros-etal/custom-biblatex-style.sty}
\let\oldtocsubsection=\tocsubsection

\renewcommand{\tocsubsection}[2]{\hspace{2em}\oldtocsubsection{#1}{#2}}

\title{A basis and Schur--Weyl duality for the loop Hecke algebra}
\author[G. Janssens, A. Lacabanne, L. Schelstraete and P. Vaz]{Geoffrey Janssens, Abel Lacabanne, Léo Schelstraete and Pedro Vaz}

\address{G.J.: Institut de Recherche en Math{\'e}matique et Physique, 
Universit{\'e} Catholique de Louvain, Chemin du Cyclotron 2,  
1348 Louvain-la-Neuve, Belgium \newline 
\& Department of Mathematics and Data Science, Vrije Universiteit Brussel,
Pleinlaan $2$, 1050 Elsene,
\href{https://geoffreyjanssens.github.io/}{geoffreyjanssens.github.io},
\href{https://orcid.org/0000-0001-5540-3171}{ORCID 0000-0001-5540-3171}
}
\email{geoffrey.janssens@uclouvain.be}
\address{A.L.: Laboratoire de Math{\'e}matiques Blaise Pascal (UMR 6620), Universit{\'e} Clermont Auvergne, Campus Universitaire des C{\'e}zeaux, 3 place Vasarely, 63178 Aubi{\`e}re Cedex, France, \newline\href{http://www.normalesup.org/~lacabanne} {www.normalesup.org/$\sim$lacabanne}, \href{https://orcid.org/0000-0001-8691-3270}{ORCID 0000-0001-8691-3270}}
\email{abel.lacabanne@uca.fr}
\address{L.S.:
    Max Planck Institüt für Mathematik,
    Vivatsgasse 7, 53111 Bonn, Germany,
    \newline  \href{https://leo-schelstraete.github.io/}{leo-schelstraete.github.io},
    \href{https://orcid.org/0000-0001-7167-3964}{ORCID 0000-0001-7167-3964}
}
\email{leoschelstraete@gmail.com}
\address{P.V.:
    Institut de Recherche en Math{\'e}matique et Physique, 
    Universit{\'e} Catholique de Louvain, Chemin du Cyclotron 2,  
    1348 Louvain-la-Neuve, Belgium,
    \newline \href{https://perso.uclouvain.be/pedro.vaz}{https://perso.uclouvain.be/pedro.vaz},
    \href{https://orcid.org/0000-0001-9422-4707}{ORCID 0000-0001-9422-4707}
}
\email{pedro.vaz@uclouvain.be}

\begin{document}
%
%
\newdimen\captionwidth\captionwidth=\hsize
%

\begin{abstract}
The loop Hecke algebra is a generalization of the Hecke algebra to the loop braid group, introduced by Damiani, Martin and Rowell.
We give a new presentation of the loop Hecke algebra provided a mild condition on the parameter and give a basis. We use higher linear rewriting theory to show linear independence and the combinatorics of Dyck paths to compute the cardinality of the basis. This yields a conjecture of Damiani--Martin--Rowel.
We also give a representation theoretic interpretation of the loop Hecke algebra in terms of (non-semisimple) Schur--Weyl duality involving the negative half of quantum $\mathfrak{gl}_{1|1}$.
\end{abstract}

\maketitle

\vspace{-1cm}
{\hypersetup{hidelinks}
\tableofcontents 
}
\vspace{-.1cm}



\section{Introduction}
  
The classical braid group $\Br_n$ can be identified with the group of motions of $n$ points in the plane $\bR^2$. In a similar spirit, the \emph{loop braid group} $\LBr_n$ is the group of motions of $n$ unlinked circles in the space $\bR^3$.
This definition was given by Dahm in his unpublished PhD thesis \cite{Dahm_GeneralizationBraidTheory_1962}, then published and extended by Goldsmith \cite{Goldsmith_TheoryMotionGroups_1981}. The terminology is due to Baez, Wise and Crans \cite{BWC_ExoticStatisticsStrings_2007}, inspired by physical motivations.
As its classical counterpart, the loop braid group admits many different definitions, reflecting its connections with various fields: as certain automorphisms of the free group on $n$ generators \cite{McCool_BasisconjugatingAutomorphismsFree_1986,Savushkina_GroupConjugatingAutomorphisms_1996};
as certain braid-like objects called \emph{welded braids} \cite{FRR_BraidpermutationGroup_1997}, in connection to virtual knot theory \cite{Kauffman_VirtualKnotTheory_1999} and knotted surfaces \cite{Satoh_VirtualKnotPresentation_2000};
or as the configuration space of $n$ unlinked circles in $\bR^3$ \cite{BH_ConfigurationSpacesRings_2013}.
We refer the reader to \cite{Damiani_JourneyLoopBraid_2017} for an overview.

The loop braid group admits an Artin-like presentation (see e.g.\ \cite{FRR_BraidpermutationGroup_1997}):
\begin{equation}
\label{eq:defn_LB}
    \LBr_n \coloneqq
    \left\langle
        \begin{array}{c}
             \sigma_1,\ldots,\sigma_{n-1},\\
             \rho_1,\ldots,\rho_{n-1}
        \end{array}
        \;\left\vert\;
        \begin{array}{c}
            \sigma_i\sigma_{i+1}\sigma_i = \sigma_{i+1}\sigma_i\sigma_{i+1},
            \mspace{5mu} 
            \rho_i\rho_{i+1}\rho_i = \rho_{i+1}\rho_i\rho_{i+1},
            \mspace{5mu} 
            \rho_i^2 = 1,
            \\
            \rho_i\sigma_{i+1}\sigma_i = \sigma_{i+1}\sigma_i\rho_{i+1},
            \mspace{5mu} 
            \sigma_{i}\rho_{i+1}\rho_{i} = \rho_{i+1}\rho_{i}\sigma_{i+1},
            \\
            \sigma_i\sigma_j = \sigma_j\sigma_i ,
            \mspace{5mu} 
            \rho_i\rho_j = \rho_j\rho_i ,
            \mspace{5mu}  
            \sigma_i\rho_j = \rho_j\sigma_i   
            \mspace{5mu}
            \text{for } \vert i-j\vert> 1
        \end{array}
        \right.
    \right\rangle .
\end{equation}
The generators $\sigma_i$ correspond to the $i$th circle passing through the $(i+1)$th circle; they generate a copy of the braid group $\Br_n\hookrightarrow\LBr_n$.
The generators $\rho_i$ correspond to permuting the $i$th circle and the $(i+1)$th circle; they generate a copy of the symmetric group $\Sym_n\hookrightarrow \LBr_n$.
The remaining relations are mixed relations, capturing how the copy of the braid group $\Br_n$ and the copy of the symmetric group $\Sym_n$ interact inside $\LBr_n$.

To study a group, one looks for interesting representations. 
As the braid group $\Br_n$ sits inside the loop braid group $\LBr_n$, a natural approach is to try to extend a representation of $\Br_n$ to a representation of $\LBr_n$.
Arguably, the most classical representation of $\Br_n$ is the Burau representation \cite{Burau_UberZopfgruppenUnd_1935}.
It has a long history, with connection to the Alexander polynomial and a still-standing faithfulness conjecture for $n=4$.
It is known to factor through (the complexification of) the Hecke algebra $\Hecke_n$, defined as a quotient of the group algebra $\bZ[t][\Br_n]$ by quadratic relations $\sigma_i^2 = (t+1)\sigma_i+t$.

Recently, Damiani, Martin and Rowell \cite{DMR_GeneralisationsHeckeAlgebras_2023} introduced the \emph{loop Hecke algebra} $\LH_n$ as an analogue for $\LBr_n$ of the Hecke algebra\footnote{To the authors' knowledge, there is no connection with loop algebras as appearing in affine Lie theory.}.
This builds on earlier work by Vershinin \cite{Vershinin_HomologyVirtualBraids_2001} extending the Burau representation to the loop braid group.
Their definition is an analogue of the definition of the Hecke algebra: it is a quotient of the group algebra $\bZ[t][\LBr_n]$ by certain quadratic relations (see \autoref{Definition loop heck}).
Although not clear from the definition, it was shown in \cite[Corollary 3.5]{DMR_GeneralisationsHeckeAlgebras_2023} that $\LH_n$ is finite dimensional. Furthermore, rather surprisingly, the dimension should be independent of $t$ under an intriguing condition on the parameter:  

\begin{conj}[Damiani--Martin--Rowell]\label{conj: dim formula}
Let $\bC_z$ be the complex numbers $\bC$ seen as a left $\bZ [t]$-module by evaluating $t$ at $z \in \bC$. Then for $z \neq \pm 1$:
\[\dim_{\bC} \LH_n \otimes_{\bZ [t]} \bC_z = \frac{1}{2} \binom{2n}{n}.\]
\end{conj}

The first goal of this paper is to confirm \autoref{conj: dim formula}.
The main object in our proof is an integral form $\varLH_n$ (\autoref{defn:var_loop_hecke}) for the loop Hecke algebra $\LH_n$.
When specializing to $\bC_z$ for $z\neq \pm1$, the integral form $\varLH_n$ gives a new presentation of $\LH_n$, independent of the parameter $t=z$.
Using this new presentation, we are able to describe an explicit basis. Linear independence is shown in \autoref{sec:basis_theorem} using higher linear rewriting theory \cite{Schelstraete_RewritingModuloDiagrammatic_2025}, i.e.\ an analogue of the diamond lemma for linear monoidal categories. We count the cardinality of the basis using the combinatorics of Dyck paths in \autoref{sec:counting_basis}. Finally, we show in \autoref{sec:equivalence_presentations} that the two presentations are equivalent when $z\neq \pm1$, leading to a proof of \autoref{conj: dim formula}.
A more detailed introduction is given in \autoref{subsec:intro_new_presentation} and \autoref{subsec:intro_basis_theorem} below.

\medbreak

The second goal of this paper is to give a representation-theoretic interpretation of the loop Hecke algebra.
The braid group is Schur--Weyl dual to the quantum $\uglone$ via its standard representation $V$.
Furthermore, the Hecke algebra fully describes $\uglone$-intertwiners, in the sense that the algebra morphism
\[
\Hecke_n\otimes_{\bZ[t]} \bQ(q)\to \End_{\uglone}(V^{\otimes n})
\]
is surjective, where $\bQ(q)$ is viewed as a $\bZ[t]$-algebra via the map $t\mapsto q^{-2}$.
In this paper, we show that a similar statement holds for the loop Hecke algebra $\LH_n$. Since $\LH_n$ is ``larger'' than $\Hecke_n$, we must ``shrink'' $\uglone$. It turns out that the right answer is to consider its negative-half $\neguglone$.
In \autoref{sec:background_uglone} we recall some background on $\uglone$ and its representations. \autoref{sec:Schur--Weyl} is the core of this second part of the paper, showing a Schur--Weyl duality between $\LH_n$ and $\neguglone$. Finally, we use this result in \autoref{sec:ring_approach_loop_Hecke} to further study the loop Hecke algebra from a ring-theoretic and representation theory of algebras perspective. A more detailed introduction is given in \autoref{subsec:intro_schur_weyl}.

\medbreak

We now give a more detailed account of the main results and objects. We conclude with some further directions of research in \autoref{subsec:intro_further_questions}.

\medbreak
\noindent \textbf{Acknowledgments.} 
G.J.\ is grateful to Fonds Wetenschappelijk Onderzoek Vlaanderen - FWO (grant 88258), and le Fonds de la Recherche Scientifique - FNRS (grant 1.B.239.22) for financial support. L.S.\ is supported by the Max Planck Institute for Mathematics (Bonn, Germany). The authors were supported by a PHC Tournesol Wallonie Bruxelles grant.

G.J.\ warmly thanks Celeste Damiani and Eric Rowell for explaining their work and telling about \autoref{conj: dim formula} during the problem session of the Banff workshop "Skew Braces, Braids and the Yang-Baxter Equation" (24w5201). Furthermore thank Paul Martin for clarifications on \cite[Section 6]{DMR_GeneralisationsHeckeAlgebras_2023}. A.L. thanks Jacques Darn\'e for discussions around this work.

\subsection{A parameter independent presentation of the loop Hecke algebra}
\label{subsec:intro_new_presentation}

We recall the definition of the loop Hecke algebra:

\begin{defn}[{\cite[section~3B]{DMR_GeneralisationsHeckeAlgebras_2023}}]
\label{Definition loop heck}
The \emph{loop Hecke algebra} $\LH_n$ is the associative unital $\bZ [t]$-algebra generated by 
$\sigma_1,\dotsc,\sigma_{n-1}$ and $\rho_1,\dotsc,\rho_{n-1}$, subject to the loop braid relations
\begin{IEEEeqnarray}{rClcrCl}
    \sigma_i\sigma_{i+1}\sigma_i &=& \sigma_{i+1}\sigma_i\sigma_{i+1},
    &\qquad&
    \rho_i\rho_{i+1}\rho_i &=& \rho_{i+1}\rho_i\rho_{i+1},
    \label{eq:sigma_rho_braid_same}
    \\
    \rho_i\sigma_{i+1}\sigma_i &=& \sigma_{i+1}\sigma_i\rho_{i+1},
    &&
    \sigma_{i}\rho_{i+1}\rho_{i}&=&\rho_{i+1}\rho_{i}\sigma_{i+1},
    \label{eq:sigma_rho_braid_mixed}
\end{IEEEeqnarray}
\begin{equation}
    \sigma_i\sigma_j = \sigma_j\sigma_i ,
    \mspace{40mu} 
    \rho_i\rho_j = \rho_j\rho_i ,
    \mspace{40mu}  
    \sigma_i\rho_j = \rho_j\sigma_i ,   
    \mspace{40mu}
    \text{for } \vert i-j\vert> 1,
    \label{eq:sigma_rho_distant_label}
\end{equation}
and the quadratic relations
\begin{IEEEeqnarray}{rCl.c.rCl}
\label{eq:sigma_rho_same_label_same}
    \rho_i^2  &=& 1, &\qquad&  (\sigma_i-1)(\sigma_i+t) &=& 0, 
    \\
    \label{eq:sigma_rho_same_label_mixed}
    (\rho_i-1)(\sigma_i+t)&=&0, && (\sigma_i-1)(\rho_i+1)&=&0.
\end{IEEEeqnarray}
\end{defn}
Note that relations \eqref{eq:sigma_rho_braid_same}, \eqref{eq:sigma_rho_braid_mixed} and \eqref{eq:sigma_rho_distant_label}, together with the quadratic relation $\rho_i^2=1$, are the defining relations of the loop braid group $\LBr_n$ given in \eqref{eq:defn_LB}.

\begin{rem}\label{rem: coefficients}
   In \cite{DMR_GeneralisationsHeckeAlgebras_2023} the loop Hecke algebra is defined as the $\bZ [t,t^{-1}]$-algebra generated by the relations \eqref{eq:sigma_rho_braid_same}-\eqref{eq:sigma_rho_same_label_mixed}. However, as pointed out around \cite[eq. (3-14)]{DMR_GeneralisationsHeckeAlgebras_2023}, it is sensible to define it over $\bZ [t]$. The extension $\LH_n \otimes_{\bZ [t]} \bZ [t^{\pm 1}]$ has the advantage that $t^{-1}(\sigma_i+t-1)$ is an inverse for $\sigma_i$; however, we will not require this fact. 
\end{rem}

When $t-1$ is invertible one could consider the following alternative generating set for $\LH_n$:
\[
D_i=(\sigma_i-\rho_i)/(1-t) \text{ and } U_i=(\sigma_i-t \rho_i)/(1-t),
\]
for $1\leq i\leq n-1$.
It turns out that for these generators $\LH_n$ has a more symmetric presentation, under some further conditions on the parameter.
This and some experiments for small $n$ using \textsc{Magma} \cite{BSM+_Magma_} motivate considering the following $\bZ$-algebra:

\begin{defn}
\label{defn:var_loop_hecke}
    For each $n\in\bN_{>0}$, the \emph{integral form of the loop Hecke algebra $\varLH_n$} is the $\bZ$-algebra with generators $D_i$ and $U_i$ for each $1\leq i \leq n-1$, subject to the following relations\footnote{In fact, the relations $U_iU_{i+1}U_i = U_{i+1}U_i$ and $D_{i+1}D_iD_{i+1} = D_{i+1}D_i$ are consequences of the other relations; see \autoref{rem:some-braid-rels-follow-from-quadratic-rels}.}:
    \begin{gather}
        \label{eq:D_and_U_same_label}
        D_iD_{i}=D_i
        \qquad
        D_iU_{i}=0
        \qquad
        U_iD_{i}=U_i+D_i-1
        \qquad
        U_iU_{i}=U_i
    \end{gather} 
    for $1\leq i\leq n-1$, and
    \begin{gather}
    \label{eq:D, U length 2}
        D_iU_{i+1}=U_{i+1}D_i
        \qquad
        U_iD_{i+1}=0
        \qquad
        D_{i+1}U_{i}=D_{i+1}+U_{i}-1
        \\ \label{eq: braid D to length 2}
        D_iD_{i+1}D_i=D_{i+1}D_i=D_{i+1}D_{i}D_{i+1}
        \\ \label{eq: braid U to length 2}
        U_iU_{i+1}U_i=U_{i+1}U_{i}=U_{i+1}U_{i}U_{i+1}
    \end{gather}
    for $1\leq i\leq n-2$, and
    \begin{gather}
    \label{eq:D_and_U_distant_label}
    U_iD_j=D_jU_i\qquad U_iU_j=U_jU_i\qquad D_iD_j=D_jD_i\qquad |i-j|>1
    \end{gather}
    for $1\leq i,j\leq n-1$.
\end{defn}

Note that the $U_i$'s and $D_i$'s still satisfy the braid relations \eqref{eq: braid D to length 2}-\eqref{eq: braid U to length 2}, but the quadratic relations \eqref{eq:sigma_rho_same_label_same}-\eqref{eq:sigma_rho_same_label_mixed} combine to the nicer relations \eqref{eq:D_and_U_same_label}, which do not involve the parameter $t$. Note also that the generators of $\varLH_n$ are idempotent elements.

An important step towards confirming \autoref{conj: dim formula} is to prove that the presentation in \autoref{defn:var_loop_hecke} is also one of the loop Hecke algebra.

\begin{maintheorem}[\autoref{cor:equivalence_presentation_over_field}]\label{main thm: eq rel}
    Let $\bC_z$ be the complex numbers $\bC$ seen as a left $\bZ [t]$-module by evaluating $t$ at $z \in \bC$. Then for $z \neq \pm 1$, the loop Hecke algebra $\LH_n$ (\autoref{Definition loop heck}) and its integral form $\varLH_n$ (\autoref{defn:var_loop_hecke}) are isomorphic over $\bC$:
    \[
    \LH_n\otimes_{\bZ[t]}\bC_z
    \cong 
    \varLH_n\otimes_\bZ\bC.
    \]
\end{maintheorem}

A more general statement is given in \autoref{thm:equivalence_presentations}.
The main difficulty in the proof is to verify that the image in $\LH_n$ of the relations \eqref{eq:D, U length 2}-\eqref{eq: braid U to length 2} holds. The full \autoref{section proof pres are equivalent} will be dedicated to that. 

\subsection{A basis for the loop Hecke algebra and consequences}
\label{subsec:intro_basis_theorem}
The relations of $\varLH_n$ have the advantage of saying how to swap any two generators. This allows us to obtain a basis, which we now introduce. First, denote $\Sym_m$ the symmetric group on $m-1$ generators, and recall that a permutation $\tau \in \Sym_m$ is called 321-\emph{avoiding} if there is no $i < j < k$ such that $\tau(i) > \tau (j) > \tau (k).$

\begin{defn}
\label{defn:reduced_words}
    A word in the alphabet $\{U_i,D_i\}_{1\leq i< n}$ is said to be \emph{$\varLH_n$-reduced} if it is of the form
    \[
    \omega = \underline{D}\;\underline{U}
    \]
    for $\underline{D}$ (resp.\ $\underline{U}$) a 321-avoiding reduced word in the alphabet $\{D_i\}_{1\leq i< n}$ (resp.\ $\{U_i\}_{1\leq i< n}$) 
    for each $1\leq i< n$,\footnote{That is, $\underline{D}$ and $\underline{U}$ are 321-avoiding reduced words each in their own alphabet.} such that if $D_i$ is a letter of $\underline{D}$, then $U_i$ and $U_{i-1}$ are \emph{not} letters of $\underline{U}$:
    \begin{gather}
    \label{eq:condition_reduced_word}
    D_i\in\underline{D}\Rightarrow U_i,U_{i-1}\not\in\underline{U}.
    \end{gather}
    We write $\Red(\varLH_n)$ for the set of $\varLH_n$-reduced words.
\end{defn}

\begin{maintheorem}[\autoref{thm:reduced_words_basis}] \label{main thm: basis}
    For each $n\in\bN_{>0}$, the set $\Red(\varLH_n)$ of $\varLH_n$-reduced words defines a basis of $\varLH_n$. 
\end{maintheorem}

We provide two proofs of linear independence of $\Red(\varLH_n)$.
The first one, given in \autoref{sec:basis_theorem}, works over $\bZ$ and uses \emph{higher linear rewriting theory} \cite{Schelstraete_RewritingModuloDiagrammatic_2025}, a higher analogue of linear rewriting theory, which encompasses both Gröbner--Shirshov bases theory \cite{Buchberger_AlgorithmFindingBasis_2006,Shirshov_AlgorithmicProblemsLie_2009} and Bergman's diamond lemma \cite{Bergman_DiamondLemmaRing_1978}.
Readers familiar with either of them should be able to follow \autoref{sec:basis_theorem} without prior knowledge of \cite{Schelstraete_RewritingModuloDiagrammatic_2025}.
In fact, our proof provides a \emph{monoidal Gröbner basis} for $\varLH_n$, i.e.\ a solution to the word problem as a linear monoidal category (see \autoref{cor:higher_grobner_basis_LH}).

The second one works over $\bC_z$, and is a by-product of the proof in \autoref{sec:Schur--Weyl} of our Schur--Weyl type theorem, see \autoref{Remark other lin.ind proof}.

\medbreak

Next, we count $\varLH_n$-reduced words:

\begin{maintheorem}[\autoref{thm:cardinality_reduced_words}]
\label{main thm: cardinality}
    The cardinality of $\Red(\varLH_n)$ is $\frac{1}{2}\binom{2n}{n}$.
\end{maintheorem}

The proof is the content of \autoref{sec:counting_basis}. It uses the combinatorics of Dyck paths, which may be of independent interest.

\medbreak

Taken together, \autoref{main thm: eq rel}, \autoref{main thm: basis} and \autoref{main thm: cardinality} imply that \autoref{conj: dim formula} indeed holds.

\begin{maincorollary}\label{main coro: dim formula}
Let $\bC_z$ be the complex numbers $\bC$ seen as a left $\bZ [t]$-algebra by evaluating $t$ at $z \in \bC$. Then for $z \neq \pm 1$:
\[\dim_{\bC} \LH_n \otimes_{\bZ [t]} \bC_z = \frac{1}{2} \binom{2n}{n}.\]
\end{maincorollary}

A more general statement is given in \autoref{cor:dimension_LH}.

\subsection{On a Schur--Weyl duality for the loop Hecke algebra}
\label{subsec:intro_schur_weyl}

In the recent years there has been quite some interest in describing representations of $\LBr_n$ which are extended from representations of the classical braid group $\Br_n$, e.g. \cite{BFM_RepresentationsLoopBraid_2019,Chang_RepresentationsLoopBraid_2020,MRT_ClassificationChargeconservingLoop_2025}. There has been particular attention to so-called local representations which include those representations associated to a braided vector space. However, in contrast to the symmetric loop braid group\footnote{This is $\LBr_n$ modulo the relation $\rho_i \sigma_{i+1}\sigma_i = \sigma_{i+1}\sigma_i\rho_{i+1}$. It is also called the {\it unrestricted virtual braid group} in \cite{KL_VirtualBraidsLmove_2006}.}, only a single $R$-matrix seems to be known that yields a local representation of $\LBr_n$, see \cite[Remark 5.4]{DMR_GeneralisationsHeckeAlgebras_2023}. The second main goal of this paper is to contribute to this and subsequently to use it to provide a Schur--Weyl duality for $\LH_n$.

A well-known representation of the braid group is the Burau representation \cite{Burau_UberZopfgruppenUnd_1935}, which has been extended to the loop braid group by \cite{Vershinin_HomologyVirtualBraids_2001}. However, as in \cite{DMR_GeneralisationsHeckeAlgebras_2023}, we consider a slight variation on the Burau representation, which is defined on a tensor space. Let $V$ be a two-dimensional vector space, and define a map
\begin{equation*}\label{tensor representation}
\left\{
\begin{array}{lll}
\Br_n &\rightarrow &\End (V^{\otimes n})\\
\sigma_i &\mapsto &\id_{V}^{\otimes (i-1)} \otimes M(t) \otimes \id_V^{\otimes(n-i-1)}
\end{array}
\right.
\end{equation*}
with 
\[
M(t) = 
\begin{pmatrix}
    1 & 0 & 0&0 \\
    0& 1-t & t & 0 \\
    0& 1 & 0 &0 \\
    0 &0 &0 & t
\end{pmatrix}.
\]
This is a well-defined representation of $\Br_n$ and it factors through the Hecke algebra $\Hecke_n$. The tensor space $V$ can be turned into a module for the quantum group $U_q(\mathfrak{sl}_2)$ and, under the map $t\mapsto q^{-2}$, the above action of the Hecke algebra is $U_q(\mathfrak{sl}_2)$-equivariant and generates the endomorphism ring $\End_{U_q(\mathfrak{sl}_2)}(V^{\otimes n})$, see \cite{jimbo-q-analogue}. Moreover, the classical Burau representation can be recovered as a weight space of $V^{\otimes n}$, as well as the more general Lawrence--Krammer--Bigelow representations if one changes the representation $V$ by a generic Verma module \cite{jackson-kerler,ltv-verma-howe}.

However, as noticed in \cite{DMR_GeneralisationsHeckeAlgebras_2023}, this tensor space version of the Burau representation does not extend to the loop braid group $\LBr_n$, one needs to tweak the representation by a sign: we replace the above matrix $M(t)$ by the matrix $M'(t)$ obtained by replacing the $t$ in the bottom right entry by $-t$. This still defines a representation of the braid group $\Br_n$ and the above story is then the same, but with the quantum group $U_q(\mathfrak{gl}_{1|1})$ instead of $U_q(\mathfrak{sl}_2)$. This culminates in a version of the Schur--Weyl duality also known as Schur--Sergeev duality \cite{Moon_HighestWeightVectors_2003,Mitsuhashi_SchurWeylReciprocityQuantum_2006}. As noted in \cite[Theorem 5.2]{DMR_GeneralisationsHeckeAlgebras_2023}, this representation of the braid group does extend to the loop braid group $\LBr_n$ via the map
\begin{equation}\label{DMR repr}
\LBurau_n \colon
\left\{
\begin{array}{lll}
\Br_n &\rightarrow &\End (V^{\otimes n}) \\
\sigma_i &\mapsto &\id_{V}^{\otimes (i-1)} \otimes M'(t) \otimes \id_V^{\otimes(n-i-1)}\\
\rho_i & \mapsto & \id_{V}^{\otimes (i-1)} \otimes M'(1) \otimes \id_V^{\otimes(n-i-1)}
\end{array}
\right..
\end{equation}
The authors call $\LBurau_n$ the extended super representation, denoted $SP$ in \cite{DMR_GeneralisationsHeckeAlgebras_2023}, or the Burau--Rittenberg representation, and show that this representation factors through the loop Hecke algebra $\LH_n$. Damiani--Martin--Rowell moreover conjectured that this representation $\LBurau_n$ is faithful if $t\neq \pm 1$.

To complete the picture, since the action of the braid group $\Br_n$ has been extended to the loop braid group $\LBr_n$, one must restrict the action of $U_q(\mathfrak{gl}_{1|1})$ to obtain a Schur--Weyl duality statement. The correct answer turns out to restrict to the negative half $U_q(\mathfrak{gl}_{1|1})^{\leq 0}$. 

\begin{maintheorem}[\autoref{the rep is iso}]
    The Burau--Rittenberg representation $\LBurau_n$ is $U_q(\mathfrak{gl}_{1|1})^{\leq 0}$-equivariant. Moreover, it realizes an isomorphism $\LBurau_n \colon \LH_n\otimes_{\bZ[t]}\bQ(q) \to \End_{U_q(\mathfrak{gl}_{1|1})^{\leq 0}}(V^{\otimes n})$.
\end{maintheorem}

This result is proven in \autoref{sec:Schur--Weyl} and, as an immediate corollary, the Burau--Rittenberg representation of the loop Hecke algebra is faithful, when working over the field $\bQ(q)$. As a final comment on this Schur--Weyl duality statement, we want to emphasize that the action of negative half $U_q(\mathfrak{gl}_{1|1})^{\leq 0}$ is not semisimple anymore, even when working with a field of characteristic $0$ and $q$ not being a root of unity.

\medbreak

In \autoref{sec:ring_approach_loop_Hecke}, we make use of this result to study the loop Hecke algebra over the field $\bQ(q)$. In \cite[Theorem 5.8]{DMR_GeneralisationsHeckeAlgebras_2023}, the authors determine the structure of the algebra $SP_n$, the image of the Burau--Rittenberg representation. Our interpretation of $SP_n$ as the $U_q(\mathfrak{gl}_{1|1})^{\leq 0}$-equivariant endomorphisms of $V^{\otimes n}$ enlights these structural properties. Indeed, we determine explicitly the Wedderburn--Mal'cev decomposition of $\LH_n\otimes_{\bZ[t]}\bQ(t)\simeq SP_n$ in \autoref{W-M decompisition End} and, in particular, the quotient by the Jacobson radical is isomorphic to $\End_{U_q(\mathfrak{gl}_{1|1})}(V^{\otimes n})$. This clarifies the appearance of the hook shaped Young diagrams in the classification of irreducible representations in \cite[Theorem 5.8 (i)]{DMR_GeneralisationsHeckeAlgebras_2023}. We also describe the projective indecomposable representations through the lens of the quantum group action, and recover the Cartan matrix and the Ext-quiver of the algebra $\End_{U_q(\mathfrak{gl}_{1|1})^{\leq 0}}(V^{\otimes n})$ as in \cite[Theorem 5.8 and Corollary 6.1]{DMR_GeneralisationsHeckeAlgebras_2023}.

\newpage
\subsection{Further questions}
\label{subsec:intro_further_questions}

\subsubsection{Extension to other Coxeter groups}

Virtual braids are braid-like objects closely related to loop braids. They have been extended outside type $A$, giving \emph{virtual Artin braid groups} \cite{BPT_VirtualArtinGroups_2023}.
Can one mimick \cite{BPT_VirtualArtinGroups_2023} and define ``loop Artin braid groups'', i.e.\ loop braid groups outside type $A$?
Note that the Burau representation can be defined outside type $A$ (see e.g.\ \cite{BQ_RemarksFaithfulnessBurau_2024}). 

\begin{mainquestion}
    Assuming a notion of ``loop Artin groups'', are there loop Burau representations and loop Hecke algebras outside type $A$?
\end{mainquestion}

\subsubsection{Toward a quantum group approach to representations of the loop braid group}

The classical theory of quantum groups provides a unified approach to produce $R$-matrices, i.e.\ representations of the braid group. In analogy, we may call a ``loop $(R,S)$-matrix pair'' the data of two matrices $R$ and $S$ inducing a representation of the loop braid group.
In this article, we show that the algebraic data of half quantum $\glone$ gives such a loop $(R,S)$-matrix pair, realizing the Burau--Rittenberg $\LBurau$ representation of the loop braid group.

\begin{mainquestion}
    Is there a general theory of (half?)\ quantum groups that produces large families of representations of the loop braid group? Are there other Schur--Weyl dualities involving half quantum groups?
\end{mainquestion}

Let us comment on the specific case of the quantum group $U_q(\mathfrak{sl}_2)$ with the standards generators $K^{\pm 1},E$ and $F$, together with its standard $2$-dimensional representation $W$. We encourage the reader to compare with \autoref{E:2-fold-tensor}. We assume basic knowledge of the representation theory of $U_q(\mathfrak{sl}_2)$ and identify its weights with $\bZ$. As a representation of $U_q(\mathfrak{sl}_2)$, the tensor product $W\otimes W$ decomposes as a sum $W\otimes W\cong L_3\oplus L_1$, where $L_3$ is a $3$-dimensional representation with basis vectors $v_2,v_0$ and $v_{-2}$ of respective weights $2,0$ and $-2$, and $L_1$ is a $1$-dimensional representation with basis vector $w_0$ of weight $0$. Assume that $\varphi \colon L_{3} \to L_{0}$ is a linear map commuting only with the negative part of the quantum group $U_q(\mathfrak{sl}_2)$. For weights reasons, $\varphi(v_{-2}) = \varphi(v_{2}) = 0$. Moreover, since $F\cdot v_{2}$ is a non zero multiple of $v_{0}$, we must have $\varphi(v_{0}) = 0$. This can be pictured as follows, where rightward (resp.\ leftward dashed) arrows denote the action of $F$ (resp.\ $E$), and the vertical arrow is $\varphi$:
\[
    \begin{tikzcd}
        L_{3}= & v_{2} \ar[r,bend left=10pt,"{[2]_q}"] 
        & v_{0} \ar[l,bend left=10pt,"1",dashed] \ar[r,bend left=10pt,"{1}"] \ar[d,"\varphi"] & v_{-2} \ar[l,bend left=10pt,"{[2]_q}",dashed]
        \\
        &
        & w_{0} & 
        & =L_{1}
    \end{tikzcd}
\]
Similarly, the only map $L_{1}\to L_{3}$ commuting with the action of the negative part of the quantum group is the zero map. Therefore, restricting the action of $U_q(\mathfrak{sl}_2)$ on $W\otimes W$ to its negative part does not produce any new intertwiner.

\subsubsection{Categorification}

This article gives two bases of the loop Hecke algebra: the one coming from its integral form in the $U$'s and $D$'s (\autoref{defn:var_loop_hecke}), and the one coming from Schur--Weyl duality.

\begin{mainquestion}
    What is the matrix of change of coefficients between the two bases?
\end{mainquestion}

The Burau representation was categorified by Khovanov and Seidel \cite{KS_QuiversFloerCohomology_2002}. This was followed by extensive work on categorical braid group actions on triangulated categories, in analogy with the classical Teichmüller theory of surfaces.
On a related note, the Hecke algebra is categorified by Soergel bimodules---shadow of this categorification is the Kazhdan--Lusztig basis of the Hecke algebra.

\begin{mainquestion}
    Is there a loop analogue of the Kazhdan--Lusztig basis, Soergel bimodules and the Khovanov--Seidel categorical Burau representation?
\end{mainquestion}

\section{A basis for the integral form of the loop Hecke algebra}
\label{sec:basis_theorem}

In this section, we prove a basis theorem for the integral form of the loop Hecke algebra:

\begin{thm} \label{thm:reduced_words_basis}
    For each $n\in\bN_{>0}$, the set of $\varLH_n$-reduced words (see \autoref{defn:reduced_words}) defines a basis of $\varLH_n$ (see \autoref{defn:var_loop_hecke}).
\end{thm}

It will be convenient to package the $\bZ$-algebras $\varLH_n$ for each $n$ into one $\bZ$-linear monoidal category: we define the \emph{loop Hecke category} in \autoref{subsec:loop_hecke_category}.
\autoref{subsec:pattern_avoiding_reduced_words} describes reduced words using pattern avoidance.
These two sections are preliminaries for \autoref{subsec:proof_basis_theorem}, where we prove \autoref{thm:reduced_words_basis} using higher linear rewriting theory.
This gives an intrinsic proof of \autoref{thm:reduced_words_basis}. Another proof of linear independence over $\bQ(t)$ will be given in \autoref{sec:Schur--Weyl} (see \autoref{Remark other lin.ind proof}), using Schur--Weyl duality.

As a direct consequence of the theorem, we find:

\begin{cor}
\label{cor:varLH-canonical-map-is-inclusion}
    The canonical algebra morphism $\varLH_n\rightarrow\varLH_{n+1}$, sending $U_i$ to $U_i$ and $D_i$ to $D_i$, is an inclusion.
\end{cor}

\begin{proof}
    The canonical map induces an inclusion of bases.
\end{proof}

\subsection{The loop Hecke category}
\label{subsec:loop_hecke_category}

\begin{defn}
\label{defn:loop_hecke_category}
    The \emph{loop Hecke category} $\varLH$ is the $\bZ$-linear monoidal category given by the following presentation:
    \begin{itemize}
        \item one generating object, so that $\mathrm{ob}(\varLH) \cong \bN$;
        \item two generating morphisms
        \[ U\colon 2\to 2\quad\text{ and }\quad D\colon 2\to 2\]
        \item subject to the following relations, where we abuse notation in \eqref{eq:cat D, U length 2}, \eqref{eq:cat braid length 2} and write $U$ for $U\otimes \mathrm{id}_1$, we write ${U_+\coloneqq \mathrm{id}_1\otimes U}$, and similarly for $D$ and $D_+$:
        \begin{gather}
            \label{eq:cat_D_and_U_same_label}
            DD=D
            \qquad
            DU=0
            \qquad
            UD=U+D-1
            \qquad
            UU=U
            \\       
            \label{eq:cat D, U length 2}
            DU_{+}=U_{+}D
            \qquad
            UD_{+}=0
            \qquad
            D_{+}U=D_{+}+U-1
            \\ 
            \label{eq:cat braid length 2}
            DD_{+}D=D_{+}D=D_{+}DD_{+}
            \qquad
            UU_{+}U=U_{+}U=U_{+}UU_{+}
        \end{gather}
    \end{itemize}
\end{defn}
The hom-spaces of $\varLH$ recover the algebras $\varLH_n$:
\[\Hom_{\varLH}(n,n) = \varLH_n,\]
where the identification is given by
\[
U_i = \id_{i-1}\otimes U\otimes \id_{n-i-1}
\qquad\text{ and }\qquad
D_i = \id_{i-1}\otimes D\otimes \id_{n-i-1}.
\]
Note that the relations \eqref{eq:D_and_U_distant_label} are not explicitly part of the above presentation, as they correspond to the interchange law of a monoidal category.
For that reason, we will refer to these relations as \emph{interchange}, even when considered in the algebra $\varLH_n$.

\subsubsection{Symbolics}
\label{subsubsec:abuse_notation_U+}

In what follows, we will often abuse notation and write
\begin{gather*}
U = \id_{i-1}\otimes U\otimes \id_{n-i-1}, \quad 
U_+ \coloneqq \id_{i}\otimes U\otimes \id_{n-i-2},\quad\text{ and }\quad 
U_{++} \coloneqq \id_{i+1}\otimes U\otimes \id_{n-i-3}
\end{gather*}
for some $i$ and $n$ clear from context.

\subsubsection{Diagrammatics}
\label{subsubsec:diagrammatic_definition}

It will be convenient to have a diagrammatic notation for morphisms in the loop Hecke category.
We define:
\begin{gather*}
    U\coloneqq \diagLH{\diagU}
    \quad\text{ and }\quad
    D\coloneqq \diagLH{\diagD}\;.
\end{gather*}
In diagrammatic notation, the definition relations of $\varLH$ become
\begin{gather*}
    \diagLH{\diagD[0][1]\diagD}
    =
    \diagLH{\diagD}
    \mspace{60mu}
    \diagLH{\diagD[0][1]\diagU}
    =0
    \mspace{60mu}
    \diagLH{\diagU[0][1]\diagD}
    =
    \diagLH{\diagU}
    +
    \diagLH{\diagD}
    -
    \diagLH{\diagSTR\diagSTR[1][0]}
    \mspace{60mu}
    \diagLH{\diagU[0][1]\diagU}
    =
    \diagLH{\diagU}
    \\[1ex]
    \diagLH{\diagDm[0][1]\diagUp}
    =
    \diagLH{\diagUp[0][1]\diagDm}
    \mspace{60mu}
    \diagLH{\diagUm[0][1]\diagDp}
    = 0
    \mspace{60mu}
    \diagLH{\diagDp[0][1]\diagUm}
    =
    \diagLH{\diagDp}
    +
    \diagLH{\diagUm}
    -
    \diagLH{\diagSTR\diagSTR[1][0]\diagSTR[2][0]}
    \\[1ex]
    \diagLH{\diagDm[0][2]\diagDp[0][1]\diagDm}
    =
    \diagLH{\diagDp[0][1]\diagDm}
    =
    \diagLH{\diagDp[0][2]\diagDm[0][1]\diagDp}
    \mspace{60mu}
    \diagLH{\diagUm[0][2]\diagUp[0][1]\diagUm}
    =
    \diagLH{\diagUp[0][1]\diagUm}
    =
    \diagLH{\diagUp[0][2]\diagUm[0][1]\diagUp}
\end{gather*}

\subsection{A pattern-avoiding description of reduced words}
\label{subsec:pattern_avoiding_reduced_words}

Recall from \autoref{defn:reduced_words} the notion of word and $\varLH_n$-reduced word.
Say that a word is \emph{$\varLH$-reduced} if it is $\varLH_n$-reduced for some $n\in\bN$.

Recall that we call ``interchange'' the relations \eqref{eq:D_and_U_distant_label}.

\begin{lem}
\label{lem:HF-reduced_are_S-reduced}
    A word is $\varLH$-reduced if and only if it avoids the following patterns, up to interchange:
    \begin{gather}
        \label{eq:reduced_word_avoided_pattern_A}
        UD
        \qquad
        U_{+}D
        \qquad
        UD_{+}
        \\
        \label{eq:reduced_word_avoided_pattern_B}
        DD
        \qquad
        UU
        \qquad
        DD_{+}D\qquad D_{+}DD_{+}
        \qquad
        UU_{+}U\qquad U_{+}UU_{+}
        \\
        \label{eq:reduced_word_avoided_pattern_C}
        DU
        \qquad
        D_{+}U
        \qquad
        DU_{+}U
        \qquad
        D_+DU_+
        \\
        \label{eq:reduced_word_avoided_pattern_D}
        D_+DU_{++}U_+
    \end{gather}
\end{lem}
Here we use the abuse of notation from \autoref{subsubsec:abuse_notation_U+}.
For instance, fixing $n\in\bN$, the pattern $D_+U$ covers all patterns of the form $D_{i+1}U_i$ for $1\leq i\leq n-2$.

\begin{proof}
    For the purpose of the proof, we call \emph{$\sR$-reduced}\footnote{The terminology ``$\sR$-reduced'' (or ``$\sR$-normal forms'') is a rewriting terminology, used explicitly in the proof of the basis theorem below. For the purpose of the current proof, this can be taken as an ad-hoc terminology.} any word that avoids patterns \eqref{eq:reduced_word_avoided_pattern_A}, \eqref{eq:reduced_word_avoided_pattern_B}, \eqref{eq:reduced_word_avoided_pattern_C} and \eqref{eq:reduced_word_avoided_pattern_D}, up to interchange. It is clear that if a word is $\varLH_n$-reduced, then is it $\sR$-reduced. Let then $\omega$ be an $\sR$-reduced word.
    Since $\omega$ avoids the patterns \eqref{eq:reduced_word_avoided_pattern_A}, it is of the form $
    \omega = \underline{D}\;\underline{U}
    $
    for $\underline{D}$ and \ $\underline{U}$ words in the alphabets $\{D_i\}_{1\leq i< n}$ and $\{U_i\}_{1\leq i< n}$, respectively.
    Since $\omega$ avoids the patterns \eqref{eq:reduced_word_avoided_pattern_B}, the words $\underline{D}$ and $\underline{U}$ are 321-avoiding and reduced.

    It remains to check that $\omega$ verifies the condition \eqref{eq:condition_reduced_word}.
    Let then $1\leq i\leq n-1$ such that $D_i$ appears as a letter in $\underline{D}$, and denote $d$ the rightmost such letter in $\underline{D}$. Using interchange, move $d$ to the right within $\underline{D}$, as much as possible. (In the process, other letters may move as well.) Since $\underline{D}$ is 321-avoiding and reduced, this expresses $\underline{D}$ as
    \[\underline{D} = \underline{D}'\;D_iD_{i+1}\ldots D_{i+m_+}D_{i-1} D_{i-2}\ldots D_{i-m_-},\]
    where $D_i$ above is the chosen letter $d$, and $m_+,m_-$ are non-negative integers.
    (If $m_- = 0$, then $D_{i-1} D_{i-2}\ldots D_{i-m_-}$ is the empty word.)
    Note that if $U_l$ is the leftmost letter of $\underline{U}$, we must have either $l < i-m_--1$ or $i+m_+ < l$, since $\omega =\underline{D}\;\underline{U}$ avoids the patterns \eqref{eq:reduced_word_avoided_pattern_C}.

    Pick $1\leq j\leq n-1$ such that $U_j$ appears in $\underline{U}$, and let $u$ be the leftmost such letter in $\underline{U}$. Similarly as above, we can express $\underline{U}$ as
    \[\underline{U}=U_{j-n_-}\ldots U_{j-2}U_{j-1}
    U_{j+n_+}\ldots U_{j+1}U_j\;\underline{U}'\]
    where $U_j$ above is the chosen letter $u$, and $n_-,n_+$ are non-negative integers. If $D_{k}$ is the rightmost letter of $\underline{D}$ then we must have either $k < j-n_-$ or $j+n_++1 < k$.

    Let
    \begin{gather*}
        M_\pm \coloneqq
        \begin{cases}
            m_\pm & \text{if } m_\pm \geq 1,\\
            - m_\mp & \text{otherwise}
        \end{cases}
        \quad\text{ and }\quad
        N_\pm \coloneqq
        \begin{cases}
            n_\pm & \text{if } n_\pm \geq 1,\\
            - n_\mp & \text{otherwise}
        \end{cases}.
    \end{gather*}
    The letters $D_{i-M_-}$ and $D_{i+M_+}$ (resp.\ $U_{i-N_-}$ and $U_{i+N_+}$) are, up to interchange, rightmost letters in $D$ (resp.\ leftmost letters in $U$).
    The conditions above give:
    \begin{IEEEeqnarray*}{RrClcrCl}
        &(i-M_- &<& j-n_-\quad&\text{ or }&\quad j+n_++1 &<& i-M_-) \\
        \text{and }\quad&
        (i+M_+ &<& j-n_-\quad&\text{ or }&\quad j+n_++1 &<& i+M_+) \\
        \text{and }\quad&
        (j-N_- &<& i-m_--1\quad&\text{ or }&\quad i+m_+ &<& j-N_-) \\
        \text{and }\quad&
        (j+N_+ &<& i-m_--1\quad&\text{ or }&\quad i+m_+ &<& j+N_+).
    \end{IEEEeqnarray*}
    Using that $M_\pm\leq m_\pm$ and $N_\pm\leq n_\pm$, we see that $i-M_- < j-n_-$ and $j-N_- < i-m_--1$ cannot hold at the same time; similarly, $j+n_++1 < i+M_+$ and $i+m_+ < j+N_+$ cannot hold at the same time.
    It follows that:
    \begin{IEEEeqnarray*}{RrClcrCl}
        &(j+n_++1 &<& i-M_-\quad&\text{ or }&\quad i+m_+ &<& j-N_-) \\
        \text{and }\quad&
        (i+M_+ &<& j-n_-\quad&\text{ or }&\quad j+N_+ &<& i-m_--1).
    \end{IEEEeqnarray*}
    If the first equation holds and $M_-\geq 0$, then $j+1<i$ and the pair $(i,j)$ verifies condition \eqref{eq:condition_reduced_word}; similar statements hold for the three other inequalities.
    Moreover, at least one element of $\{M_-,M_+\}$ is non-negative; and similarly for $\{N_-,N_+\}$.
    Hence, it only remains to check the following two cases:
    \begin{itemize}
        \item $M_-,N_+<0$:
        this implies that $m_- = n_+ = 0$;
        \item $M_+,N_-<0$: 
        this implies that $m_+ = n_- = 0$.
    \end{itemize}
    In both cases, we can use the avoidance of pattern \eqref{eq:reduced_word_avoided_pattern_D} to conclude that we have either $i<j$ or $j+1<i$.
\end{proof}

\subsection{Proof of the basis theorem via rewriting theory}
\label{subsec:proof_basis_theorem}

We wish to show \autoref{thm:reduced_words_basis}.
Recall that given a linear monoidal category, a hom-basis is a basis for each hom-space.
\autoref{thm:reduced_words_basis} equivalently states that $\varLH$-reduced words define a hom-basis of the loop Hecke category $\varLH$ (\autoref{defn:loop_hecke_category}).

With the pattern-avoidance description of reduced words given in \autoref{lem:HF-reduced_are_S-reduced}, it is not difficult to show the following:

\begin{lem}
\label{lem:reduced_words_generate}
    For each $n\in\bN$, the set of $\varLH_n$-reduced words generates $\varLH_n$ as a $\bZ$-module.
\end{lem}

\begin{proof}
    Let $\omega$ be any word in the alphabet $\{U_i,D_i\}_{1\leq i< n}$.
    Each pattern described in \autoref{lem:HF-reduced_are_S-reduced} can be rewritten using a relation in $\varLH_n$.
    This follows from the defining relations and the relations
    \[
    D_+ DU_+= 0,
    \quad
    DU_+U= 0
    \quad\text{ and }\quad
    D_+D U_{++}U_+= 0,
    \]
which are easy consequences of the defining relations.
As long as $\omega$ contains one of the patterns in \autoref{lem:HF-reduced_are_S-reduced}, we continue rewriting it. This process will eventually terminate, as rewriting a pattern strictly decreases the number of letters.
This shows that $\omega$ can be expressed as a linear combination of reduced words, using relations in $\varLH_n$.
\end{proof}

To show linear independence, we use \emph{higher linear rewriting theory}, as introduced in \cite{Schelstraete_RewritingModuloDiagrammatic_2025}. In fact, the case of the loop Hecke category is rather simple compared to other monoidal categories, and combining linear rewriting and higher rewriting is relatively straightforward, at least when starting from their modern formulation (see e.g.\ \cite{GHM_ConvergentPresentationsPolygraphic_2019} and e.g.\ \cite{GM_HigherdimensionalCategoriesFinite_2009}, respectively) and once injectivity of contexts is recognized; see \autoref{rem:comparison_rewriting_literature} for a discussion.
In particular, it allows us to phrase our discussion and review of \cite{Schelstraete_RewritingModuloDiagrammatic_2025} in terms that resemble the classical theory of Gröbner--Shirshov bases \cite{Buchberger_AlgorithmFindingBasis_2006,Shirshov_AlgorithmicProblemsLie_2009}, or Bergman's diamond lemma \cite{Bergman_DiamondLemmaRing_1978}, for associative algebras.
We begin with an informal discussion of the main ideas; these ideas are then (semi-)formalized in \autoref{subsubsec:review_rewriting}, which gives a minimal review of the relevant theory from \cite{Schelstraete_RewritingModuloDiagrammatic_2025}.
\autoref{subsubsec:LH_counting_critical_branchings} explains how the theory is applied to the loop Hecke category, the full details being postponed to \autoref{app:confluence_critical_branchings}.

\medbreak

The proof of \autoref{lem:reduced_words_generate} defined a process that rewrites every word as a linear combination of reduced words; it is encapsulated in \autoref{fig:higher-Grobner-basis-LH} (symbolics) and \autoref{fig:higher-Grobner-basis-LH-diagrammatic} (diagrammatics). The idea of rewriting theory is to formalize this process as an algorithm, where each step is called a \emph{rewriting step}; by studying the properties of this algorithm, we will deduce linear independence.
More precisely, we wish to show that not only can we rewrite a word as a linear combination of reduced words, but moreover this linear combination is \emph{unique}.
The latter is highly non-obvious, since a word can have two forbidden patterns at the same time, and our process does not choose which one should be rewritten first (in other words, the algorithm it defines is not deterministic).
For instance the word $DU_+D$ can be rewritten in two different ways:
\begin{gather}
\label{eq:LH-example-branching}
\begin{tikzcd}[ampersand replacement=\&,column sep=small,row sep=.1em]
    \& D_+U+D_+D-D_+ 
    \\
    D_+UD \ar[ur,bend left=15pt] \ar[dr,bend right=15pt]
    \\
    \& D_+D+UD-D
\end{tikzcd}
\end{gather}
Such a pair of rewriting steps is called a \emph{branching}.
While the two rewriting steps have distinct target, one can check that they \emph{confluate}, it the sense that one can use further rewriting steps to reach a common target:
\begin{gather*}
\begin{tikzcd}[ampersand replacement=\&,column sep=small,row sep=.1em]
    \& D_+U+D_+D-D_+ \rar\& D_++U-1+D_+D-D_+ \ar[dr,equals,bend left=10pt]
    \\
    D_+UD \ar[ur,bend left=15pt] \ar[dr,bend right=15pt]
    \&\&\&
    D_+D+U-1
    \\
    \& D_+D+UD-D \rar\& D_+D+U+D-1-D \ar[ur,equals,bend right=10pt]
\end{tikzcd}
\end{gather*}
We say that the algorithm \emph{confluates} if every branching confluates; in this case, a word always rewrites as a linear combination of reduced words \emph{in a unique way}.
In fact, to show confluence it suffices to show confluence of branchings that ``overlap''; they are called \emph{critical branchings}. For instance, the branching in \eqref{eq:LH-example-branching} is a critical branching.
If all critical branchings confluate, we say that the algorithm \emph{critically confluates}.

Linear combinations on which the algorithm terminates are called \emph{normal forms} (or \emph{reduced}); if a word is a normal form, it is called a \emph{monomial normal form}. In our setting, monomial normal forms are precisely $\varLH$-reduced words; this is the content of \autoref{lem:HF-reduced_are_S-reduced}. To sum up:
\begin{quote}
    \textsc{Slogan} (see \autoref{thm:basis-from-rewriting}): If the algorithm in \autoref{fig:higher-Grobner-basis-LH} terminates and critically confluates, then $\varLH$-reduced words define a hom-basis of the loop Hecke category $\varLH$, showing \autoref{thm:reduced_words_basis}.
\end{quote}
As a byproduct, we get a solution to the word problem; that is, an algorithm that decides whether two (linear combination of) words are equal in $\varLH$.
This is the perspective of Gröbner bases. For that reason, the oriented relations underpinning the algorithm may be called a \emph{monoidal Gröbner basis}; that is, a Gröbner basis for a linear monoidal category.

In the terminology defined in \autoref{subsubsec:review_rewriting}:

\begin{cor}
\label{cor:higher_grobner_basis_LH}
    The higher linear rewriting system given in \autoref{fig:higher-Grobner-basis-LH} is a monoidal Gröbner basis for the loop Hecke category.
\end{cor}

\begin{figure}
    \def\sin{.2ex}
    \def\sout{1ex}
    \centering
    \begin{gather*}
    DD\to D 
    \qquad
    DU\to 0
    \qquad
    UU\to U 
    \qquad
    UD\to U +D -1
    \\[\sin]
    \text{\small same-label rewriting steps}
    \\[\sout]
    U_+ D \to D U_+ 
    \qquad
    U D_+ \to 0
    \qquad
    D_+ U \to D_+ +U -1
    \\
    D D_+ D \to D_+ D \qquad D_+ D D_+ \to D_+ D
    \\
    U U_+ U \to U_+ U \qquad U_+ U U_+ \to U_+ U
    \\[\sin]
    \text{\small distinct-label rewriting steps}
    \\[\sout]
    D_+ D U_+\to 0
    \qquad
    D U_+ U\to 0
    \qquad
    D_+ D U_{++} U_+ \to 0
    \\[\sin]
    \text{\small additional rewriting steps}
\end{gather*}
    \caption{A monoidal Gröbner basis for the loop Hecke category.}
    \label{fig:higher-Grobner-basis-LH}
\end{figure}

\begin{figure}
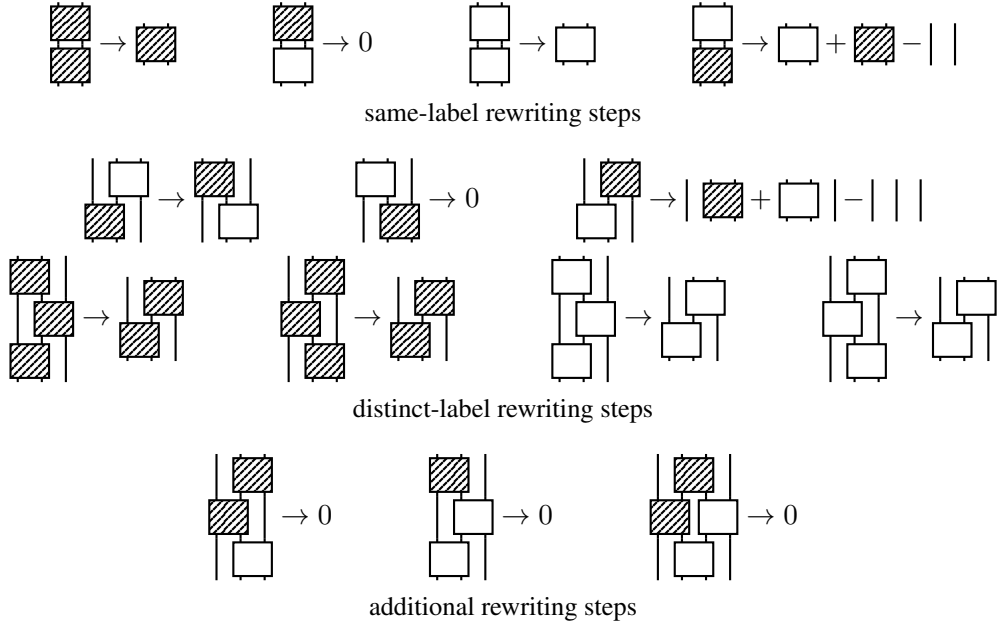

    \def\sin{.2ex}
    \def\sout{1ex}
    \centering
    \begin{gather*}
    \diagLH{\diagD[0][1]\diagD}
    \to
    \diagLH{\diagD}
    \mspace{60mu}
    \diagLH{\diagD[0][1]\diagU}
    \to
    0
    \mspace{60mu}
    \diagLH{\diagU[0][1]\diagU}
    \to
    \diagLH{\diagU}
    \mspace{60mu}
    \diagLH{\diagU[0][1]\diagD}
    \to
    \diagLH{\diagU}
    +
    \diagLH{\diagD}
    -
    \diagLH{\diagSTR\diagSTR[1][0]}
    \\[\sin]
    \text{\small same-label rewriting steps}
    \\[\sout]
    \diagLH{\diagUp[0][1]\diagDm}
    \to
    \diagLH{\diagDm[0][1]\diagUp}
    \mspace{60mu}
    \diagLH{\diagUm[0][1]\diagDp}
    \to 
    0
    \mspace{60mu}
    \diagLH{\diagDp[0][1]\diagUm}
    \to
    \diagLH{\diagDp}
    +
    \diagLH{\diagUm}
    -
    \diagLH{\diagSTR\diagSTR[1][0]\diagSTR[2][0]}
    \\
    \diagLH{\diagDm[0][2]\diagDp[0][1]\diagDm}
    \to
    \diagLH{\diagDp[0][1]\diagDm}
    \mspace{60mu}
    \diagLH{\diagDp[0][2]\diagDm[0][1]\diagDp}
    \to
    \diagLH{\diagDp[0][1]\diagDm}
    \mspace{60mu}
    \diagLH{\diagUm[0][2]\diagUp[0][1]\diagUm}
    \to
    \diagLH{\diagUp[0][1]\diagUm}
    \mspace{60mu}
    \diagLH{\diagUp[0][2]\diagUm[0][1]\diagUp}
    \to
    \diagLH{\diagUp[0][1]\diagUm}
    \\[\sin]
    \text{\small distinct-label rewriting steps}
    \\[\sout]
    \diagLH{\diagDp[0][2]\diagDm[0][1]\diagUp}
    \to 0
    \mspace{60mu}
    \diagLH{\diagDm[0][2]\diagUp[0][1]\diagUm}
    \to 0
    \mspace{60mu}
    \diagLH{
        \diagDp[0][2]\diagSTR[3][2]
        \diagD[0][1]\diagU[2][1]
        \diagUp\diagSTR[3][0]
    }
    \to 0
    \\[\sin]
    \text{\small additional rewriting steps}
\end{gather*}
    \caption{The monoidal Gröbner basis from \autoref{fig:higher-Grobner-basis-LH} in diagrammatic notation.}
    \label{fig:higher-Grobner-basis-LH-diagrammatic}
\end{figure}

\subsubsection{A minimal review of higher linear rewriting theory}
\label{subsubsec:review_rewriting}

We give a minimal review of \cite{Schelstraete_RewritingModuloDiagrammatic_2025} suitable for our purpose.
In the interest of length, we sometimes stay at a semi-formal level of explanation.
Experts are referred to \autoref{rem:comparison_rewriting_literature} for comparison with \cite{Schelstraete_RewritingModuloDiagrammatic_2025}.

Fix $k$ a commutative ring.
Let
\[\cC = \langle \sX_0\mid\sX_1\mid\sR\rangle_{\otimes,k}\]
be a presented linear monoidal category.
This means that:
\begin{itemize}
    \item $\sX_0$ is the set of generating objects. We write $\sX_0^* = \mathrm{ob}(\cC)$ the set of objects generated by $\sX_0$ under the tensor product $\otimes$.
    \item $\sX_1$ is the set of generating morphisms, or \emph{generators}, equipped with source and target maps $s,t\colon \sX_1\to\sX_0^*$.
    A generator $f\in \sX_1$ can be ``extended'' by identities of objects, giving a \emph{whiskered generator} $\id_a\otimes f\otimes\id_b$ for each $a,b\in\sX_0^*$.
    The source and target maps extend to whiskered generators in the natural way, and whiskered generators with matching source and target can be composed.
    
    For each pair of objects $(a,b)$, we write $\sX^*(a,b)$ the set of compositions of whiskered generators with source $a$ and target $b$, regarded up to the interchange law:
    \[(f\otimes \id_{t(g)})\circ (\id_{s(f)}\otimes g) = (\id_{t(f)}\otimes g)\circ (f\otimes\id_{s(g)}).\]
    An element of $\sX^*(a,b)$ is called a \emph{monomial}.

    For each pair of objects $(a,b)$, we write $\sX^l(a,b)\coloneqq\langle\sX^*(a,b)\rangle_k$, the free $k$-module generated by the set $\sX^*(a,b)$. An element of $\sX^l(a,b)$ is called a \emph{vector}.
    We write $\sX^*$ (resp.\ $\sX^l$) the union of all the $\sX^*(a,b)$'s (resp.\ $\sX^l(a,b)$'s).

    \item $\sR$ is a subset $\sR\subset \sX^l(a,b)$.
\end{itemize}
In the case of the loop Hecke Category, we have $\sX_0 = \{1\}$, $\sX_0^* \cong \bN$ and $\sX_1 = \{D,U\}$.
Monomials are the same as words (up to interchange) in symbolic notation, or diagrams (up to interchange) in diagrammatic notation.

\begin{defn}
    In the context of linear monoidal categories,
    a \emph{higher linear rewriting system} is a presentation of a linear monoidal category, such that each relation comes equipped with an orientation, and the source of each relation is a monomial.
\end{defn}

In our notation, this means that $\sR$ is an abstract set equipped with source map $s\colon\sR\to \sX^*(a,b)$ and target map $s\colon\sR\to \sX^l(a,b)$.
We often abuse notation by writing $\sR$ for the data $(\sX_0\mid\sX_1\mid\sR)$, and abuse terminology by calling $\sR$ a higher linear rewriting system.

Recall that a vector is a linear combination of monomials. For a vector $v\in\sX^l$, we write $\supp(v)$ its \emph{support}, that is, the set of monomials appearing in the linear decomposition of $v$.

Let $(\sX_0\mid\sX_1\mid\sR)$ be a higher linear rewriting system.
We define:
\begin{itemize}
    \item A \emph{context} is a ``monomial with a hole''.
    In diagrammatic terms, a context $\Gamma$ is:
    \[
    \Gamma \;=\; 
    \begin{tikzpicture}[scale=.7,
    baseline={(current bounding box.center)},
    ]
        \def\hshift{.2}
        \def\vshift{.1}
        \draw (0,0) to (0,3);
        \draw (1,.5) to (1,2.5);
        \draw (2,.5) to (2,2.5);
        \draw (3,.5) to (3,2.5);
        \draw (4,.5) to (4,2.5);
        \draw (5,0) to (5,3);
        \draw[fill=white,rounded corners=1pt] (0-\hshift,0+\vshift) rectangle (5+\hshift,1-\vshift);
        \draw[fill=white,rounded corners=1pt] (2-\hshift,1+\vshift) rectangle (3+\hshift,2-\vshift);
        \draw[fill=white,rounded corners=1pt] (0-\hshift,2+\vshift) rectangle (5+\hshift,3-\vshift);
        \node at (2.5,.5) {\footnotesize $f$};
        \node at (2.5,2.5) {\footnotesize $g$};
        \node at (.5,1.5) {\footnotesize $a$};
        \node at (4.5,1.5) {\footnotesize $b$};
    \end{tikzpicture}
    \]
    where $a,b\in\mathrm{ob}(\cC)$ are objects and $f,g\in\sX^*$ are monomials, suitably composable.
    Given a context $\Gamma$, we can \emph{contextualize} a generating rewriting step $r\colon s(r)\to t(r)$ as \[{\Gamma[r]\colon \Gamma[s(r)] \to \Gamma[t(r)]},\] provided source and target are compatible.

    \item A \emph{rewriting step} is a rule of the form
    \[\lambda\Gamma[r]+v \colon\lambda\Gamma[s(r)] + v \to \lambda\Gamma[t(r)] + v,\]
    where $\lambda\in k\setminus\{0\}$ is a non-zero scalar, $\Gamma$ is a context, $r\in\sR$ is a generating rewriting step and $v\in\sX^l$ is a vector such that $\Gamma[s(r)]\notin\supp(v)$. Again, it is implicit that $\Gamma[r]$ and $v$ have the same source and target.
    
    We say that $\lambda\Gamma[r]+v$ is a rewriting step \emph{of type $r$}.
    We denote by $\sR^+$ the set of rewriting steps. Source and target maps naturally extend to $\sR^+$.
\end{itemize}
For instance, in our example $r = DD\to D$ is a generating rewriting step, $\Gamma = D_+[-]U$ is a context and $v=UD$, $v'=D_+DDU$ are vectors (in fact, monomials).
We have that $\Gamma[r] = D_+ DD U \to D_+ D U$ is a rewriting step, viewed a contextualization of $r$. The rule
\[\Gamma[r]+v = D_+ DD U + UD \to D_+ D U + UD\]
is a rewriting step, but not the rule $\Gamma[r]+v' = D_+ DD U + D_+ DD U \to D_+ D U + D_+ DD U$, as it does not verify the condition on the support.
This condition is known as the positivity condition, which explains the notation $\sR^+$.

Having defined rewriting steps, we have the following notions:
\begin{itemize}
    \item A \emph{rewriting sequence} is a finite sequence of rewriting steps $(\alpha_i)_{1\leq i\leq N}$ such that we have ${s(\alpha_{i+1}) = t(\alpha_i)}$;
    \item A \emph{branching}\footnote{More precisely, this is the definition of a \emph{local} branching; we abuse terminology in this review.} is a pair of rewriting steps $(\alpha,\beta)$ with the same source;
    \item A \emph{confluence} is a pair of rewriting sequences $(\alpha',\beta')$ with the same target;
    \item A branching $(\alpha,\beta)$ is \emph{confluent} (we say that it \emph{confluates}) if it admits a confluence $(\alpha',\beta')$ such that $t(\alpha) = s(\alpha')$ and $t(\beta) = s(\beta')$.
\end{itemize}

There is an intuitive notion of the multiset of generators in a given monomial; for instance, the monomial $(D\otimes \id_1)\circ (\id_1\otimes D)$ has generators $\{D,D\}$.
Given a rewriting step $\alpha=\Gamma[r]$ for $r\in\sR$ a generating rewriting step, we call the multiset of generators in $s(r)$ the ``generators associated to $\alpha$''; they constitute the pattern on which we apply the rewriting rule.
A branching $(\alpha,\beta)$ is \emph{monomial} if its source $s(\alpha)=s(\beta)$ is a monomial. If further $\alpha$ and $\beta$ have generators in common, we say that $(\alpha,\beta)$ is an \emph{overlapping branching}.
For instance, the branching given in \eqref{eq:LH-example-branching} is an overlapping branching, as the two rewriting steps have the generator ``$U$'' in the middle of the monomial $D_+UD$ in common.

Just like rewriting steps, a branching $(\alpha,\beta)$ can be contextualized as $(\Gamma[\alpha],\Gamma[\beta])$; we say that a branching is \emph{minimal} if it is not the (non-trivial) contextualization of another branching.

\begin{defn}
    In the context of linear monoidal categories, a \emph{critical branching} is a minimal overlapping branching.
\end{defn}

\begin{defn}
    A higher linear rewriting system is said to \emph{terminate} if there is no infinite sequence of rewriting steps, and to \emph{critically confluate} if all critical branchings confluate.
\end{defn}

Now we pause the review to emphasize a special feature of our setting.
Note that for a generic higher linear rewriting system, \emph{contextualization needs not be injective}. That is, if $f,g\in\sX^*$ are monomials and $\Gamma$ is a context, having $f\neq g$ does \emph{not} imply that $\Gamma[f]\neq \Gamma[g]$\footnote{Here the inequality is an inequality as monomials in the free monoidal category $\sX^*$, not an inequality as elements of the monoidal category presented by $\sR$.}. This is because we consider monomials up to the interchange law, and as a context may connect two regions of a diagram, it may allow ``floating morphisms'' to move from one region to another. This fact is in contrast with the classical settings of linear rewriting in associative algebras or commutative algebras, where contexts \emph{are} indeed injective.
This makes the general theory of higher linear rewriting subtler than its classical counterpart; see \cite{Schelstraete_RewritingModuloDiagrammatic_2025}.

However, in the case of the loop Hecke category, contexts \emph{are} injective: if $f,g\in\sX^*$ are monomials such that $f\neq g$ and $\Gamma$ is a context, then $\Gamma[f]\neq\Gamma[g]$. This makes the theory similar to the classical setting, and one finds a statement analogous to (say) Bergman's diamond lemma.

\begin{defn}
    In the context of higher linear rewriting system with injective contexts,
    A \emph{monoidal Gröbner basis} is a higher linear rewriting system that terminates and critically confluates.
\end{defn}

If $v\in\sX^l$ is a vector such that no rewriting step has source $v$, we say that $v$ is a \emph{normal form}; if further $v\in\sX^*$ is a monomial, then $v$ is a \emph{monomial normal form}.

\begin{thm}
\label{thm:basis-from-rewriting}
    In the context of higher linear rewriting system with injective contexts,
    If $\mathsf{R}$ is a monoidal Gröbner basis presenting a linear monoidal category $\mathcal{C}$,
    monomial normal form defines a hom-basis of $\mathcal{C}$.
\end{thm}

\begin{rem}[{comparison with \cite{Schelstraete_RewritingModuloDiagrammatic_2025}}]
    \label{rem:comparison_rewriting_literature}
    In \cite{Schelstraete_RewritingModuloDiagrammatic_2025}, the interchange law is made explicit as a modulo rule; in particular, $\sX^*$ denotes the set of formal compositions of whiskered generators, \emph{not} regarded up to the interchange law.
    Also, the notion of monoidal Gröbner bases is only implicit in \cite{Schelstraete_RewritingModuloDiagrammatic_2025}, and equivalent to the notion of a convergent higher linear rewriting system modulo interchangers.

    To view \autoref{thm:basis-from-rewriting} as a corollary of the results of \cite{Schelstraete_RewritingModuloDiagrammatic_2025}, one requires a strongly compatible terminating order invariant under interchangers \cite[Definition~3.58]{Schelstraete_RewritingModuloDiagrammatic_2025}.
    Since contexts are injective (see the discussion above), an order is strongly compatible if and only if it is compatible, and we can choose our compatible order to be the smallest compatible order $\succ_\sR$ for each linear rewriting system $\sR(a,b)$ \cite[Definition~3.48]{Schelstraete_RewritingModuloDiagrammatic_2025}.
    Thanks to \cite[Lemma~3.41]{Schelstraete_RewritingModuloDiagrammatic_2025}, termination implies that $\succ_\sR$ is terminating.
    In other words, having injective contexts put us in essentially the same setting as associative algebras; see \cite{GHM_ConvergentPresentationsPolygraphic_2019}.
\end{rem}

\subsubsection{A rewriting approach to the loop Hecke category}
\label{subsubsec:LH_counting_critical_branchings}

Let $\sR$ be the higher linear rewriting system defined in \autoref{fig:higher-Grobner-basis-LH}.
We have already argued that $\sR$ terminates (see the proof of  \autoref{lem:reduced_words_generate}) and that monomial normal forms are precisely $\varLH$-reduced words (\autoref{lem:HF-reduced_are_S-reduced}).
To show \autoref{thm:reduced_words_basis} using \autoref{thm:basis-from-rewriting}, it only remains to show that $\sR$ is critically confluent.
This splits in two steps: enumerating all the critical branchings, and showing that they confluate.
Both tasks are cumbersome, but apart from the subtlety of indexed branchings explained below, essentially straightforward.

Enumeration is best done diagrammatically: one tries to match patterns up to rectilinear isotopies. See \autoref{subsubsec:diagrammatic_definition} and \autoref{fig:higher-Grobner-basis-LH-diagrammatic} for the diagrammatics.
We illustrate the process with branchings $(\alpha,\beta)$ where $\alpha$ is of type $U_+D \to DU_+$ and $\beta$ is of type one of the distinct-label rewriting steps. The full analysis is given in \autoref{app:confluence_critical_branchings}.

\begin{lem}
\label{lem:critical_branching_enumeration_example}
The following is a complete list of critical branchings $(\alpha,\beta)$ where $\alpha$ is of type $U_+D \to DU_+$ and $\beta$ is of type one of the distinct-label rewriting steps:
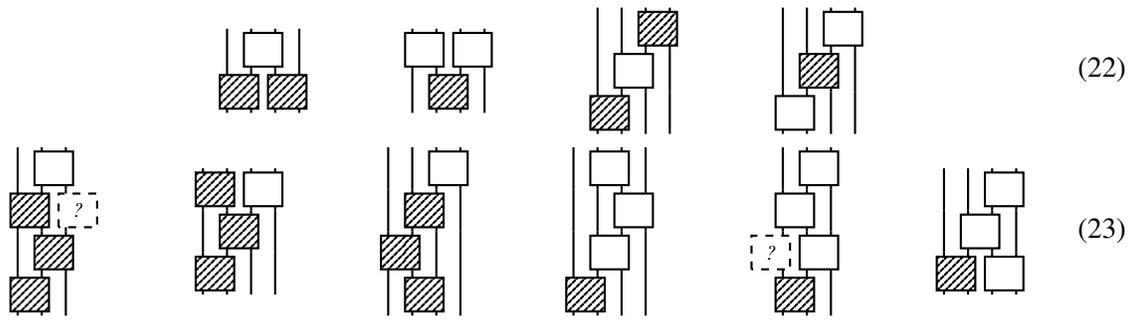
\begin{gather}
    \label{eq:branching_DUp_first}
    \diagLH{\diagIUI[0][1]\diagD\diagD[2][0]}
    \mspace{60mu}
    \diagLH{\diagU[0][1]\diagU[2][1]\diagIDI}
    \mspace{60mu}
    \diagLH{\diagIID[0][2]\diagIUI[0][1]\diagDII}
    \mspace{60mu}
    \diagLH{\diagIIU[0][2]\diagIDI[0][1]\diagUII}
    \\
    \label{eq:branching_DUp_second}
    \diagLH{
        \diagIU[0][3]
        \diagD[0][2]\diagLBOX[2][2]
        \diagID[0][1]\diagDI
    }
    \mspace{60mu}
    \diagLH{
        \diagD[0][2]\diagU[2][2]
        \diagIDI[0][1]\diagDII
    }
    \mspace{60mu}
    \diagLH{
        \diagIIU[0][3]\diagIDI[0][2]
        \diagDII[0][1]\diagIDI
    }
    \mspace{60mu}
    \diagLH{
        \diagIUI[0][3]
        \diagIIU[0][2]
        \diagIUI[0][1]
        \diagDII
    }
    \mspace{60mu}
    \diagLH{
        \diagIU[0][3]
        \diagUI[0][2]
        \diagRBOX[-1][1]\diagU[1][1]
        \diagDI
    }
    \mspace{60mu}
    \diagLH{
        \diagIIU[0][2]
        \diagIUI[0][1]
        \diagD[0][0]\diagU[2][0]
    }
\end{gather}
\end{lem}

Here we describe a branching by its source, leaving its branches implicit. (Boxes with ``?'' will be explained below.)
For instance, the first diagram encodes the following branching:
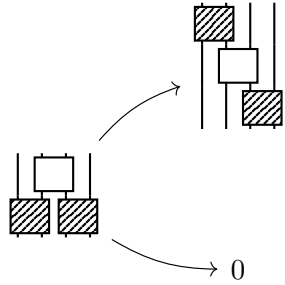
\[
\begin{tikzcd}[row sep=0em]
    & \diagLH{\diagDII[0][2]\diagIUI[0][1]\diagIID}
    \\
    \diagLH{\diagIUI[0][1]\diagD\diagD[2][0]}
    \ar[ur,bend left=15pt] \ar[dr,bend right=15pt]
    \\
    & 0
\end{tikzcd}
\]
Here we remind the reader that we view diagrams modulo the interchange law, so that we can slide the $D$s past each other to be able to apply each of the two rewriting steps. This branching is easily seen to confluate, as the top branch rewrites to zero using the rewriting step $UD_+\to 0$.

Because we work modulo the interchange law, it may happen that an arbitrary diagram is ``stuck'' in between two rewriting rules.
This happens for instance in the first branching of \eqref{eq:branching_DUp_second}, where
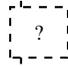
\[\diagLH{\diagLBOX}[scale=1.5]\]
denotes an arbitrary diagram. This is known as an \emph{indexed branching} \cite{GM_HigherdimensionalCategoriesFinite_2009}, a phenomenon typical of higher rewriting.
A priori, this leads to an infinite family of branchings.
However, one can always rewrite this diagram into a normal form.
As we already know that normal forms are precisely the reduced words, it is not hard to check the following:

\begin{lem}
    Denote an arbitrary diagram with a dashed box marked with ``?'', different boxes indicating (a priori) different diagrams.
    We have that:
    \begin{gather*}
        \diagLH{\diagSTR[0][2]\diagLBOX[0][1]\diagSTR}
        \quad
        \text{rewrites as a linear combination of}
        \quad
        \diagLH{
            \diagSTR[0][2]\diagSTR[1][2]
            \diagSTR[0][1]\diagLBOX[1][1]
            \diagSTR\diagSTR[1][0]
        }\;,
        \quad
        \diagLH{
            \diagSTR[0][2]\diagLBOX[1][2]
            \diagU[0][1]
            \diagSTR\diagLBOX[1][0]
        }
        \quad\text{ and }\quad
        \diagLH{
            \diagSTR[0][2]\diagLBOX[1][2]
            \diagD[0][1]
            \diagSTR\diagLBOX[1][0]
        }\;,
    \end{gather*}
    We have an analogous statement when flipping all diagrams along the vertical axis.
\end{lem}

This reduces the analysis of indexed branchings to three cases:
\begin{gather}
\label{eq:indexed_branching_three_cases}
\diagLH{\diagLBOX} = \diagLH{\diagI},\quad
\diagLH{\diagLBOX} = \diagLH{\diagD}\quad\text{ and }\quad
\diagLH{\diagLBOX} = \diagLH{\diagU}.
\end{gather}

\begin{lem}
\label{lem:critical_branching_confluence_example}
The critical branchings of \autoref{lem:critical_branching_enumeration_example} are confluent.
\end{lem}

\begin{proof}
    We have already seen that the first branching of \eqref{eq:branching_DUp_first} is confluent. The three other branchings of \eqref{eq:branching_DUp_first} are similar.

    Consider the first branching of \eqref{eq:branching_DUp_second}. As we argued above, it suffices to consider the three cases in \eqref{eq:indexed_branching_three_cases}.
    In fact, if $\diagLH{\diagLBOX} = \diagLH{\diagD}$ then both branches rewrite to zero using the rewriting step $UD_+\to 0$.
    Moreover, if $\diagLH{\diagLBOX} = \diagLH{\diagU}$ then we can use the rewriting step $U_+D \to DU_+$ on both branches to slide this $U$ downward, therefore reducing to the case $\diagLH{\diagLBOX} = \diagLH{\diagI}$.
    In this latter case, we have:
    
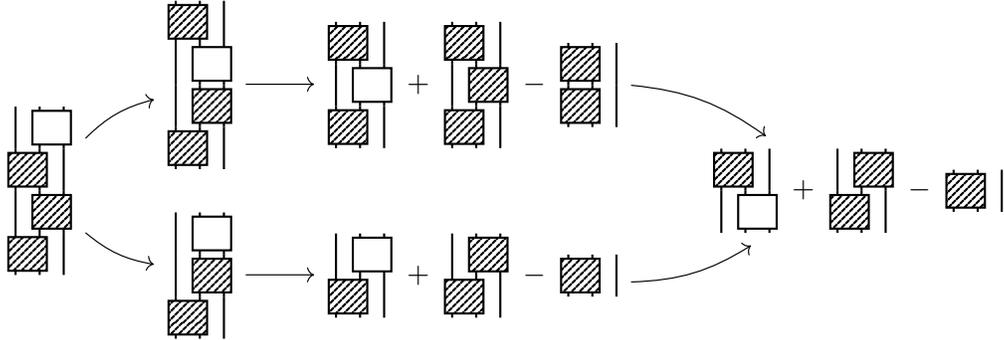
\[
    \begin{tikzcd}[row sep=-3em]
        & \diagLH{\diagDI[0][3]\diagIU[0][2]\diagID[0][1]\diagDI}
        \rar
        & 
        \diagLH{\diagDI[0][2]\diagIU[0][1]\diagDI}
        \;+\;
        \diagLH{\diagDI[0][2]\diagID[0][1]\diagDI}
        \;-\;
        \diagLH{\diagDI[0][1]\diagDI}
        \ar[dr,bend left=15pt]
        \\
        \diagLH{\diagIU[0][3]\diagDI[0][2]\diagID[0][1]\diagDI}
        \ar[ur,bend left=15pt] \ar[dr,bend right=15pt] &&&
        \diagLH{\diagDI[0][1]\diagIU}
        \;+\;
        \diagLH{\diagID[0][1]\diagDI}
        \;-\;
        \diagLH{\diagDI}
        \\
        & \diagLH{\diagIU[0][2]\diagID[0][1]\diagDI}
        \rar
        & 
        \diagLH{\diagIU[0][1]\diagDI}
        \;+\;
        \diagLH{\diagID[0][1]\diagDI}
        \;-\;
        \diagLH{\diagDI}
        \ar[ur,bend right=15pt]
    \end{tikzcd}
    \]
    The other indexed branching of \eqref{eq:branching_DUp_second} works similarly.
    The confluence of the remaining branchings is straightforward to check.
\end{proof}

\section{Cardinality of reduced words and Dyck paths}
\label{sec:counting_basis}

In this section, we prove that the set of reduced words $\Red(\varLH_n)$ (see \autoref{defn:reduced_words}), i.e.\ the basis obtained in \autoref{thm:reduced_words_basis}, has the conjectured cardinality.

\begin{thm} \label{thm:cardinality_reduced_words}
    For each $n\in\bN_{>0}$, the cardinality of $\Red(\varLH_n)$ is $\frac{1}{2}\binom{2n}{n}$.
\end{thm}

The proof of \autoref{thm:cardinality_reduced_words} will consist of two steps. 

First, we rephrase the problem in terms of Dyck paths (\autoref{lem:MDD_bijection_property}). A \emph{Dyck path of semilength $n$} is a path in the lattice $\bZ^2$ between $(0,0)$ and $(n,n)$, consisting only of steps $u\coloneqq(0,1)$ (\emph{$u$-step}) and $r\coloneqq(1,0)$ (\emph{$r$-step}), and such that the path always lies above\footnote{One could equivalently define them to be paths below the diagonal $d$.} the diagonal $d=\{(x,y)\mid x=y\}$. A Dyck path can be encoded as words in $u$ and $r$, reading from left to right as the path goes from $(0,0)$ to $(n,n)$; see \autoref{fig:MDD-bijection} for an example.

We denote
\[\Dyck_n := \{ \text{Dyck paths of semilength } n \}\] 
and for a path $P\in \Dyck_n$ and $(a,b)\in\bZ^2$ we write $(a,b)\in P$ whenever $(a,b)$ lies on $P$.

In the first step we show that the $\varLH_n$-reduced words correspond to the following set.

\begin{defn}
For $n\in\bN$:
    \[\dbDyck_n\coloneqq\big\{(P,Q)\in \Dyck_n\times \Dyck_n\mid (i,i)\notin P \Rightarrow (i,i)\text{ and }(i-1,i-1)\in Q\big\}.\]
\end{defn}

The second step of the proof of \autoref{thm:cardinality_reduced_words} will consist of relating $\dbDyck_n$ to another set whose cardinality is easily seen to be the desired one, see \autoref{lem:bijection_double_Dyck} below.\vspace{0,2cm}

Recall that $\Sym_n(321)$ denotes the set of 321-avoiding permutations in the symmetric group $\Sym_n$ and recall $\Red(\varLH_n)\subset \Sym_n(321)\times \Sym_n(321)$ from \autoref{defn:reduced_words}.
Both 321-avoiding permutations and Dyck paths count Catalan numbers. There are several bijections realizing that fact; for our purpose, we are interested in the one given by Mansour, Deng and Du \cite{MDD_DyckPathsRestricted_2006}.

Although we do not use them explicitly, the references \cite{Stanley_CatalanNumbers_2015,Callan_BijectionsDyckPaths_2007} have been helpful guides to the literature.

\begin{lem}
    \label{lem:MDD_bijection_property}
    The Mansour--Deng--Du bijection $\mathrm{MDD}\colon \Dyck_n\to \Sym_n(321)$ between Dyck paths and 321-avoiding permutations is such that for $P\in \Dyck_n$, we have $(i,i)\notin P$ if and only if $s_i\in\mathrm{MDD}(P)$.
    In particular, it induces a bijection:
    \[\Red(\varLH_n)\cong\dbDyck_n.\]
\end{lem}

\begin{proof}
    We describe the bijection $\mathrm{MDD}\colon \Dyck_n\to  \Sym_n(321)$, following \cite[section~2.1]{MDD_DyckPathsRestricted_2006}.
    An example of the procedure is given in \autoref{fig:MDD-bijection}, following \cite[Fig.~1]{MDD_DyckPathsRestricted_2006}.

    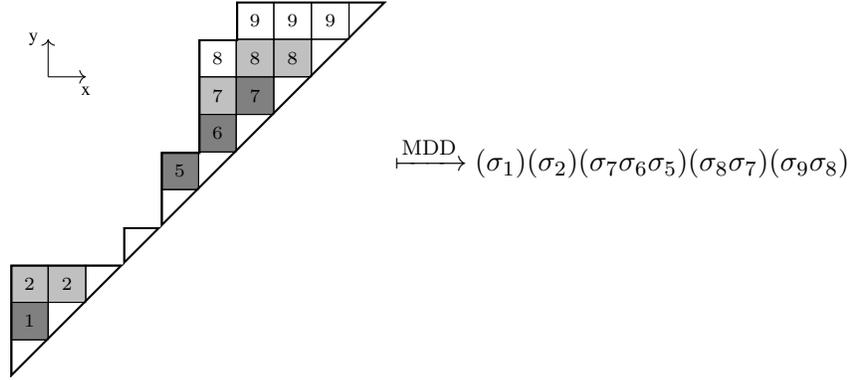
\begin{figure}
        \centering

    \begin{gather*}
        \begin{tikzpicture}[scale=.5,baseline={(base)}]
            \draw[->] (1,8) to (1,9);
            \draw[->] (1,8) to (2,8);
            \node[left] at (1,9) {\tiny y};
            \node[below] at (2,8) {\tiny x};
            \coordinate (base) at (5.5,5.5);
            \clip (0,0) to (0,3) to (3,3) to (3,4) to (4,4) to (4,6) to (5,6) to (5,9) to (6,9) to (6,10) to (10,10) to (0,0);
            \draw (0,0) grid (14,14);
            \draw[line width=1.5pt] (0,0) to (0,3) to (3,3) to (3,4) to (4,4) to (4,6) to (5,6) to (5,9) to (6,9) to (6,10) to (10,10) to (0,0);
            \draw (0,0) grid (10,10);
            \draw (0,0) to (10,10);
            \node at (0+.5,1+.5) {\tiny $1$};
            \node at (0+.5,2+.5) {\tiny $2$};
            \node at (1+.5,2+.5) {\tiny $2$};
            \node at (4+.5,5+.5) {\tiny $5$};
            \node at (5+.5,6+.5) {\tiny $6$};
            \node at (5+.5,7+.5) {\tiny $7$};
            \node at (6+.5,7+.5) {\tiny $7$};
            \node at (5+.5,8+.5) {\tiny $8$};
            \node at (6+.5,8+.5) {\tiny $8$};
            \node at (7+.5,8+.5) {\tiny $8$};
            \node at (6+.5,9+.5) {\tiny $9$};
            \node at (7+.5,9+.5) {\tiny $9$};
            \node at (8+.5,9+.5) {\tiny $9$};
            %
            \fill[opacity=.25] (5,7) rectangle (6,8);
            \fill[opacity=.25] (6,8) rectangle (8,9);
            \fill[opacity=.5] (6,7) rectangle (7,8);
            \fill[opacity=.5] (5,6) rectangle (6,7);
            \fill[opacity=.5] (4,5) rectangle (5,6);
            \fill[opacity=.25] (0,2) rectangle (2,3);
            \fill[opacity=.5] (0,1) rectangle (1,2);
        \end{tikzpicture}
        \xmapsto{\mathrm{MDD}}
        (\sigma_1)(\sigma_2)(\sigma_7\sigma_6\sigma_5)(\sigma_8\sigma_7)(\sigma_9\sigma_8)
    \end{gather*}

        \caption{Example of the Mansour--Deng--Du's bijection, following \cite[Fig.~1]{MDD_DyckPathsRestricted_2006}. The Dyck path is $uuurrruruuruuururrrr$. Each cell is equipped with its label, that is, its y-coordinate. Shadings emphasize the five zigzag strips.}
        \label{fig:MDD-bijection}
    \end{figure}
    
    Let $P\in \Dyck_n$ be a Dyck path. A \emph{cell} is a size-one square in $\bZ^2$ that lies between $P$ and the diagonal $d=\{(x,y)\mid x=y\}$. We identify a cell with the coordinate $(i,j)$ of its bottom-left corner. We label each cell with its y-coordinate $j$. A cell is \emph{essential} if $(i,j-1)\to (i,j)$ and $(i,j)\to (i,j+1)$ are steps in $P$; in other words, if the point $(i,j)$ lies between two ``up'' steps.
    Given an essential step $(i,j)$, its \emph{zigzag strip} is the set of cells adjacent to $P$ between $(i,j)$ and $(n,n)$.

    Let $\Sym_n$ be the symmetric group on $n-1$ generators $\sigma_1,\ldots,\sigma_{n-1}$. We associate an element $\mathrm{MDD}(P)\in \Sym_n$ to the Dyck path $P$, by induction on the number of essential cells:
    \begin{itemize}
        \item If $P$ has no essential cell, then $P=(ru)^n$, and we associate the identity $1\in \Sym_n$;
        \item Assume instead that $P$ has essential cells. Let $(i,j)$ be the rightmost essential cell in $P$ and let $Z$ be its zigzag strip. Let $\{i,i+1,\ldots,i+k\}$ be the set of labels of cells in $Z$ and let $P'\coloneqq P\setminus Z$ be the Dyck path obtained from $P$ by removing the zigzag strip $Z$. Note that $\mathrm{MDD}(P')\in \Sym_n$ is defined by induction. We set $\mathrm{MDD}(P)=\mathrm{MDD}(P')(\sigma_{i+k}\ldots\sigma_{i+1}\sigma_{i}$).
    \end{itemize}
    It was shown in \cite[theorem~3]{MDD_DyckPathsRestricted_2006} that this procedure defines a bijection between Dyck paths and 321-avoiding permutations.
    If $P$ is a Dyck path, then $(i,i)\in P$ if and only if no cell of $P$ has $i$ as its y-coordinate, that is, if and only if $\sigma_i\notin\mathrm{MDD}(P)$. The lemma follows.
\end{proof}

The second and last step of the proof of \autoref{thm:cardinality_reduced_words} consists in relating $\dbDyck_n$ to another set with the expected cardinality:

\begin{lem}
    \label{lem:bijection_double_Dyck}
    Denote by $\mathrm{Path}(n,n-1)$ the set of paths from $(0,1)$ to $(n,n)$ consisting in steps $r=(1,0)$ and $u=(0,1)$.
    There is a bijection:
    \[\dbDyck_n\cong \mathrm{Path}(n,n-1).\]
\end{lem}

\begin{proof}
For each pair $(P,Q)\in\dbDyck_n$, we think of $P$ as sitting above the diagonal $d=\{(x,y)\mid x=y\}$, and $Q$ as sitting below the diagonal.
We begin with a few definitions:
\begin{itemize}
    \item If $P$ (resp.\ $Q$) is incident to the diagonal at points $(i,i)$ and $(j,j)$ for $0\leq i,j\leq n$, we write $P_{i,j}$ (resp.\ $Q_{i,j}$) the sub-Dyck path of $P$ (resp.\ $Q$) from $(i,i)$ to $(j,j)$;
    \item A \emph{squiggly $P$-line} is a sub-Dyck path of $P$ of the form $P_{i,i+k}=(ur)^k$ for $i\geq 1$;
    \item A \emph{maximal squiggly $P$-line} is a squiggly $P$-line $P_{i,i+k}$ maximal with respect to $k$, that is, such that neither $P_{i-1,i-1+k}$ nor $P_{i,i+k+1}$ is a squiggly $P$-line;
\end{itemize}
We stress the condition $i\geq 1$ for a squiggly $P$-line: a squiggly $P$-line never contains the first $u$-step of $P$.

Given a maximal squiggly $P$-line $P_{i,j}$, the definition of $\dbDyck_n$ implies that $Q_{i-1,j}$ is a sub-Dyck path of $Q$, from $(i-1,i-1)$ to $(j,j)$. Indeed, on one hand $P_{i,j+1}$ is not a squiggly line, so either $(j,j)=(n,n)$ or $(j+1,j+1)\notin P$, and both situations imply that $(j,j)\in Q$; and on the other hand, $P_{i-1,j}$ is not a squiggly line, so either $(i,i)=(1,1)$ or $(i-1,i-1)\notin P$, and both situations imply that $(i-1,i-1)\in Q$.
Let $\mathrm{tr}(Q_{i-1,j})$ be the \emph{truncation} of $Q_{i-1,j}$ obtained by removing the first $r$-step and last $u$-step of $Q_{i-1,j}$.
Let $\varphi(Q_{i-1,j})$ be the upward shift of $\mathrm{tr}(Q_{i-1,j})$, obtained by shifting each step of $\mathrm{tr}(Q_{i-1,j})$ by one $u$-step.

For $(P,Q)\in\dbDyck_n$, define $\varphi(P,Q)\in\mathrm{Path}(n,n-1)$ by replacing each maximal squiggly $P$-line $P_{i,j}$ by $\varphi(Q_{i-1,j})$ and removing the first $u$ step of $P$.
This defines a function:
\[\varphi\colon\dbDyck_n\to \mathrm{Path}(n,n-1).\]
\begin{figure}
    \centering
    \begin{equation*}
    \def\sh{.2}
    \varphi\colon\quad
    \begin{tikzpicture}[scale=.6,baseline={(0,2)}]
        \draw[line width=.3pt,gray] (0,0) grid (7,7);
        \draw[line width=1.5pt] (0,0) to (0,1) to (1,1) to (1,3) to (3,3) to (3,4) to (4,4) to (4,6) to (6,6) to (6,7) to (7,7);
        \node[left=5pt,inner sep=0pt,fill=white] at (3,3.5) {\footnotesize $P_{3,4}$};
        \draw[rounded corners,fill,opacity=.2] (3-\sh,3.5-\sh) to (3-\sh,4+\sh) to (4+\sh,4+\sh) to (4+\sh,4-\sh) to (3+\sh,4-\sh) to (3+\sh,3-\sh) to (3-\sh,3-\sh)to (3-\sh,3.5+\sh);
        \node[left=5pt,inner sep=0pt,fill=white] at (6,6.5) {\footnotesize $P_{6,7}$};
        \draw[rounded corners,fill,opacity=.2] (6-\sh,6.5-\sh) to (6-\sh,7+\sh) to (7+\sh,7+\sh) to (7+\sh,7-\sh) to (6+\sh,7-\sh) to (6+\sh,6-\sh) to (6-\sh,6-\sh)to (6-\sh,6.5+\sh);
        \draw[line width=1.5pt,blue] (0,0) to (1,0) to (1,1) to (2,1) to (2,2) to (4,2) to (4,4) to (5,4) to (5,5) to (6,5) to (6,6) to (7,6) to (7,7);
        \draw[rounded corners,fill=blue,opacity=.2] (6-\sh,5.5-\sh) to (6-\sh,6+\sh) to (7+\sh,6+\sh) to (7+\sh,6-\sh) to (6+\sh,6-\sh) to (6+\sh,5-\sh) to (6-\sh,5-\sh)to (6-\sh,5.5+\sh);
        \node[right=5pt,inner sep=0pt,fill=white] at (4,2.3) {\footnotesize $\mathrm{tr}(Q_{2,4})$};
        \draw[rounded corners,fill=blue,opacity=.2] (3.5,2-\sh) to (4+\sh,2-\sh) to (4+\sh,3+\sh) to (4-\sh,3+\sh) to (4-\sh,2+\sh) to (3-\sh,2+\sh) to (3-\sh,2-\sh) to (3.5,2-\sh);
        \node[right=5pt,inner sep=0pt,fill=white] at (6,5.3) {\footnotesize $\mathrm{tr}(Q_{5,7})$};
        \draw (-.5,-.5) to (7.5,7.5);
        \node[left] at (-.5,-.5) {\footnotesize $d$};
    \end{tikzpicture}
    \qquad\mapsto\qquad
    \begin{tikzpicture}[scale=.6,baseline={(0,2)}]
        \draw[line width=.3pt,gray] (0,0) grid (7,7);
        \draw[line width=1.5pt] (0,1) to (1,1) to (1,3) to (3,3);
        \draw[line width=1.5pt] (4,4) to (4,6) to (6,6);
        \draw[line width=1.5pt,blue] (3,3) to (4,3) to (4,4);
        \draw[line width=1.5pt,blue] (6,6) to (6,7) to (7,7);
        \begin{scope}[yshift=1cm]  
        \draw[rounded corners,fill=blue,opacity=.2] (6-\sh,5.5-\sh) to (6-\sh,6+\sh) to (7+\sh,6+\sh) to (7+\sh,6-\sh) to (6+\sh,6-\sh) to (6+\sh,5-\sh) to (6-\sh,5-\sh)to (6-\sh,5.5+\sh);
        \node[right=5pt,inner sep=0pt,fill=white] at (4,2.3) {\footnotesize $\varphi(Q_{2,4})$};
        \draw[rounded corners,fill=blue,opacity=.2] (3.5,2-\sh) to (4+\sh,2-\sh) to (4+\sh,3+\sh) to (4-\sh,3+\sh) to (4-\sh,2+\sh) to (3-\sh,2+\sh) to (3-\sh,2-\sh) to (3.5,2-\sh);
        \node[right=5pt,inner sep=0pt,fill=white] at (6,5.3) {\footnotesize $\varphi(Q_{5,7})$};
        \end{scope}
        \draw[red] (0-.5,1-.5) to (6+.5,7+.5);
        \node[left] at (-.5,1-.5) {\footnotesize \textcolor{red}{$l$}};
    \end{tikzpicture}
\end{equation*}
    \caption{%
        Definition of $\varphi(P,Q)$ for some pair of Dyck paths $(P,Q)\in\dbDyck_7$, with $P=uruurruruurrur$ depicted in black above the diagonal $d$ and $Q=rururruurururu$ depicted in blue below the diagonal.
        The two maximal squiggly $P$-lines of $P$ are $P_{3,4}$ and $P_{6,7}$, shaded in black; the associated sub-Dyck paths of $Q$ are $Q_{2,4}$ and $Q_{5,7}$, respectively. Their truncations $\mathrm{tr}(Q_{2,4})$ and $\mathrm{tr}(Q_{5,7})$ are shaded in blue. To obtain $\varphi(P,Q)$, remove the first step of $P$, remove $P_{3,4}$ and $P_{6,7}$ and add $\varphi(Q_{2,4})$ and $\varphi(Q_{5,7})$, obtained by shifting $\mathrm{tr}(Q_{2,4})$ and $\mathrm{tr}(Q_{5,7})$ up by one step.%
    }
    \label{fig:bijection-dyckpaths}
\end{figure}
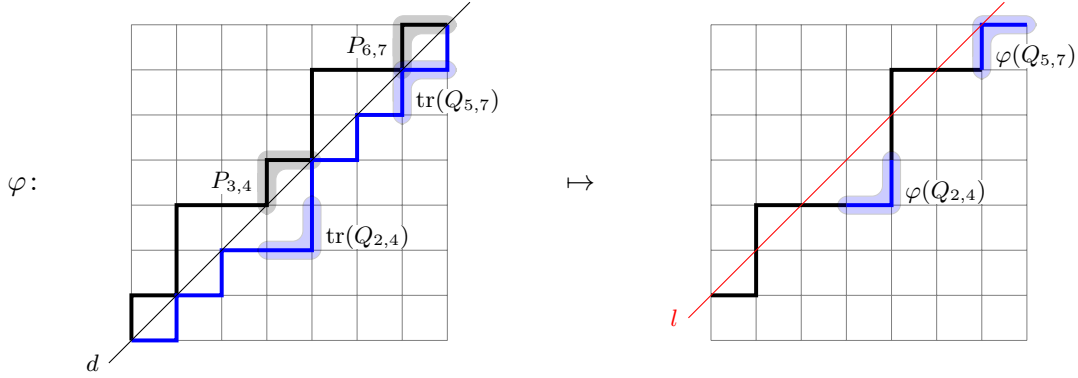
An example is given in \autoref{fig:bijection-dyckpaths}.

To recover $(P,Q)$ from $\varphi(P,Q)$, it suffices to recover which parts of $\varphi(P,Q)$ come from $P$. Let $l=\{(x,y)\mid y=x+1\}$ be the upward-shifted diagonal (in red in \autoref{fig:bijection-dyckpaths}).
One checks that a step in $\varphi(P,Q)$ is a part of $P$ if and only if it belongs to a sub-path of the following form:
\begin{itemize}
    \item the first step of $\varphi(P,Q)$, if this first step is an $r$-step;
    \item a sub-path of the form $uTr$, where $T$ is a sub-path that lies entirely above $l$.
\end{itemize}
This shows that $\varphi$ admits an inverse and concludes.
\end{proof}

\begin{proof}[Proof of \autoref{thm:cardinality_reduced_words}]
    It is clear that the cardinality of $\mathrm{Path}(n,n-1)$ is $\binom{2n-1}{n}$. Combining \autoref{lem:MDD_bijection_property} with \autoref{lem:bijection_double_Dyck} concludes.
\end{proof}

\section{Background on quantum \texorpdfstring{$\glone$}{gl(11)} and its representations}
\label{sec:background_uglone}

For a comprehensive exposition of the Hopf superalgebra $\uglone$ and its representations, see Sections 2 and 3 of~\cite{Sartori_AlexanderPolynomialQuantum_2015}.

\subsection{Quantum \texorpdfstring{$\glone$}{gl(1|1)}}
Fix a basis $\{\varepsilon_1,\varepsilon_2\}$ of the weight lattice $P=\bZ^2$ with a pairing given by
\begin{equation*}
\brak{\varepsilon_i,\varepsilon_j} =
\begin{cases}
0 & \text{ if } i \neq j,
\\
1 & \text{ if } i=j=1,
\\
-1 & \text{ if } i=j=2 .
\end{cases}
\end{equation*}
This allows us to define by duality the coweight lattice $P^*=\bZ^2$ with basis  $\{h_1,h_2\}$. We will work over the field $\bQ(q)$, but one can also work over any field of characteristic zero with an element $q\neq 0$ which is not a root of unity. 

\begin{defn}
The \emph{quantum superalgebra} $\Uq$ is the associative unital $\bQ(q)$-superalgebra with odd generators
$E,F$ and even generators $K_1^{\pm 1},K_2^{\pm 1}$, subject to the relations
\begin{align}
K_1K_2 &= K_2K_1, & K_i^{\pm 1}K_i^{\mp 1} &= 1,
\\
K_1E &= qEK_1, & K_2E &= q^{-1}EK_2,
\\
K_1F &= q^{-1}FK_1, & K_2F &= qFK_2,
\\
EF+FE &= \frac{K-K^{-1}}{q-q^{-1}}, & E^2 =& F^2 = 0,
\end{align}
where $K=K_1K_2$.
\end{defn}

For $h=n_1 h_1 + n_2 h_2\in P^*$ we set $K_h=K_1^{n_1}K_2^{n_2}$, so that $K_1=K_{h_1}$, $K_2=K_{h_2}$ and $K=K_{h_1+h_2}$ (note that $K$ is central). Note also that $K_h^{-1}=K_{-h}$. The superalgebra $\Uq$ is actually a Hopf superalgebra, with comultiplication $\Delta$, counit $\varepsilon$ and antipode $S$ given below.
\begin{align}
    \Delta(K_h) &= K_h\otimes K_h, &  \Delta(E) &= E\otimes K + 1\otimes E, & \Delta(F) &= F\otimes 1 + K^{-1}\otimes F ,
    \\
    \epsilon(K_h) &= 1 , & \epsilon(E) &= 0 , & \epsilon(F) &= 0 ,
    \\
    S(K_h) &= K_h^{-1} , & S(E) &= -EK^{-1} , & S(F) &= -KF .
\end{align}
Here, the conventions for the coproduct differ slightly from~\cite[(2.9)]{Sartori_AlexanderPolynomialQuantum_2015} and are rather in line with \cite{Zhang_StructureRepresentationsQuantum_1998}. 
\smallskip

Define a bar involution on $\Uq$ by
$\overline{E}=E$, $\overline{F}=F$, $\overline{K}_h=K_h^{-1}$  and $\overline{q}=q^{-1}$. 
Then $\overline{\Delta} := \overline{(\cdot)}\otimes \overline{(\cdot)} \circ\Delta\circ \overline{(\cdot)}$ is also a coproduct with
\begin{align*}
 \overline{\Delta}(E) &= E\otimes K^{-1} + 1\otimes E ,
\\
 \overline{\Delta}(F) &= F\otimes 1 + K\otimes F ,
\\
 \overline{\Delta}(K) &= K\otimes K .
\end{align*}

Denote $\vert x \vert$ the parity of $x\in\Uq$ and set $\Delta^{\mathrm{op}}(x)=\sum (-1)^{\vert x_1 \vert . \vert x_2 \vert}x_{(2)}\otimes x_{(1)}$ in the Sweedler notation. 

\subsection{Representations}

Set $\vert\varepsilon_1\vert=0$ and $\vert\varepsilon_2\vert=1$ and extend linearly to a map $\vert\bullet\vert: P \to \bZ/2\bZ$.
We only consider finite dimensional weight representations: these consist of finite dimensional $\Uq$-supermodules $M=\bigoplus\limits_{\lambda\in P}M_\lambda$ with 
\[M_\lambda=\{v\in M \colon K_h \cdot v=q^{\brak{h,\lambda}}v\}\]
and $M_\lambda$ in parity $\vert\lambda\vert$.

All simple representations of $\Uq$ are one- or two-dimensional and are indexed by their highest weight $\lambda\in P$.

\begin{itemize}

\item If $\brak{h_1+h_2, \lambda}=0$ then the simple representation with highest weight $\lambda$ is one-dimensional: $\bQ(q)_\lambda=\bQ(q)v_\lambda$, with $\vert v_\lambda\vert =\vert\lambda\vert$, and 
\begin{equation*}
E.v_\lambda=F.v_\lambda=0, \mspace{60mu} K_h.v_\lambda=q^{\brak{h,\lambda}}v_\lambda .
\end{equation*}

\item  If $\brak{h_1+h_2, \lambda}\neq 0$ then the simple representation with highest weight $\lambda$ is two-dimensional: 
$L(\lambda) = \bQ(q)v^0_\lambda \oplus \bQ(q)v^1_\lambda$ with $\vert v^0_\lambda\vert=\vert\lambda\vert$, $\vert v^1_\lambda\vert=\vert\lambda\vert+1$ and 
\begin{align}\label{action on 1st basis vector}
    E.v^0_\lambda &=0,   &   F.v^0_\lambda &= [\brak{h_1+h_2,\lambda}] v^1_\lambda , & K_h.v^0_\lambda &= q^{\brak{h,\lambda}} v^0_\lambda ,
    \\\label{action on 2nd basis vector}
    E.v^1_\lambda &= v^0_\lambda,   &   F.v^1_\lambda &= 0 , & K_h.v^1_\lambda &= q^{\brak{h,\lambda-\varepsilon_1+\varepsilon_2}} v^1_\lambda .
\end{align}
\end{itemize}
where $[k] := \frac{q^k-q^{-k}}{q-q^{-1}}$ is the quantum number.

\begin{exe}
    The representation $V:=L(\varepsilon_1)$ is called the {\it vector representation.} In this case $\langle h_1 + h_2, \varepsilon_1 \rangle = 1$ and hence also $[\brak{h_1+h_2,\varepsilon_1}] = 1$. Furthermore $q^{\langle h , \lambda - \varepsilon_1 + \varepsilon_2 \rangle} = q^{\langle h, \varepsilon_2 \rangle}$. To ease notations, we will drop the subscript $\varepsilon_1$ in the vectors $v^0_{\varepsilon_1}$ and $v^1_{\varepsilon_1}$ of the vector representation $V$.
\end{exe}

The representations $L((n-\ell)\varepsilon_1 + \ell \varepsilon_2))$, for $n\in\bN$ and $0\leq \ell \leq n$ will often appear in the sequel. For ease of notation, with $n$ clear within the context, we will denote this representation by $L(\ell)$.

The tensor product of two-dimensional simple representations follows an easy rule.
If $\lambda,\mu\in P$ are such that $\brak{h_1+h_2,\lambda}$, $\brak{h_1+h_2,\mu}$ and $\brak{h_1+h_2,\lambda+\mu}$ are nonzero 
then
\begin{equation}\label{otimes of highest weight}
    L(\lambda)\otimes L(\mu) \cong L(\lambda+\mu) \oplus L(\lambda+\mu-\varepsilon_1+\varepsilon_2).
\end{equation}
If $\brak{h_1+h_2,\lambda+\mu}=0$, then the representation $L(\lambda)\otimes L(\mu)$ is indecomposable.

\subsection{Braiding and Schur--Weyl duality}

\subsubsection{The quasi-$R$-matrix and braiding} The monoidal category of finite dimensional weight representations can be endowed with a braided structure. For this, we introduce the quasi-$R$-matrix $\Theta$ defined by
\begin{equation*}
\label{eq:quasi-R-matrix}
\Theta = 1\otimes 1 - (q-q^{-1}) E\otimes F.
\end{equation*}
Since our coproduct is not the one of \cite{Sartori_AlexanderPolynomialQuantum_2015}, we also have a different quasi-$R$-matrix. We also define a morphism of superalgebras $\Psi \colon \Uq \otimes \Uq \to \Uq \otimes \Uq$ by
\begin{align*}
    \Psi(K_h \otimes 1) &= K_h\otimes 1, &\Psi(E\otimes 1) & = E\otimes K^{-1}, & \Psi(F\otimes 1) &= F\otimes K,\\
    \Psi(1\otimes K_h) &= 1\otimes K_h, & \Psi(1\otimes E) &= K^{-1}\otimes E, & \Psi(1\otimes F) &= K\otimes F.
\end{align*}

From here and onward, given an element $x = \sum_i a_i\otimes b_i \in \Uq^{\otimes 2}$, we denote by $x_{12}$ the element $\sum_i a_i\otimes b_i \otimes 1$, by $x_{13}$ the element $\sum_i a_i\otimes 1 \otimes b_i$ and by $x_{23}$ the element $\sum_i 1\otimes a_i \otimes b_i$. We use a similar notation for morphisms. The following are straightforward calculations. 

\begin{prop}
\label{prop:quasi-R-matrix}
    Let $x\in \Uq$. We have
    \begin{enumerate}
        \item \label{it:intertwining1}$\Theta \Delta(x) = \overline{\Delta}(x)\Theta$,
        \item \label{it:intertwining2} $\overline{\Delta}(x) = \Psi(\Delta^{\mathrm{op}}(x))$,
        \item \label{it:hexagon1} $\Delta\otimes\id(\Theta)=\Psi_{23}(\Theta_{13})\Theta_{23}$,
        \item \label{it:hexagon2} $\id\otimes\Delta(\Theta) = \Psi_{12}(\Theta_{13})\Theta_{12}$.
    \end{enumerate}
\end{prop}

\subsubsection{Braided structure on the category of representations}

The above quasi-$R$-matrix $\Theta$ is used to endow the category of representations of $\Uq$ with a braided structure. For $M$ and $N$ two weight modules we define $\Theta_{M,N}\colon M\otimes N\to M\otimes N$ and $f_{M,N}\colon M\otimes N\to M\otimes N$ by
\[
\Theta_{M,N}(m\otimes n) = m\otimes n - (-1)^{\vert m\vert}(q-q^{-1})E.m\otimes F.n ,
\]
for $m\in M$ and $n\in N$, and 
\[
f_{M,N}(m\otimes n)= q^{(\mu,\nu)}m\otimes n 
\]
if in addition $m\in M_\mu$ and $n\in N_\nu$.

\begin{prop}
    \label{prop:intertwining}
    We have that $f_{M,N}\circ \Theta_{M,N}$ intertwines $\Delta$ and $\Delta^{\mathrm{op}}$:
    \[
    \forall x\in\Uq, f_{M,N}\circ \Theta_{M,N}\circ\Delta(x)=\Delta^{\mathrm{op}}(x)\circ f_{M,N}\circ \Theta_{M,N}.
    \]
\end{prop}

\begin{proof}
This follows from \autoref{prop:quasi-R-matrix}.\eqref{it:intertwining1} and \autoref{prop:quasi-R-matrix}.\eqref{it:intertwining2}.
\end{proof}

Now set $\check{R}_{M,N}=\tau\circ f_{M,N}\circ\Theta_{M,N}$, where $\tau$ is the super-twist, which is defined by $\tau(m\otimes n)=(-1)^{\vert m\vert .\vert n\vert}n\otimes m$. 

\begin{thm}\label{thm: R is braiding}
    The map $\check{R}$ is a braiding in the category of $\Uq$ weight representations.
\end{thm}

\begin{proof}
    The map $\check{R}_{M,N}$ is $\Uq$-linear thanks to \autoref{prop:intertwining}. The hexagon axioms follow from \autoref{prop:quasi-R-matrix}.\eqref{it:hexagon1} and \autoref{prop:quasi-R-matrix}.\eqref{it:hexagon2}.
\end{proof}

\subsubsection{The vector representation and a Schur--Weyl duality}

Recall the vector representation $V=L(\varepsilon_1)$. An iterated use of \eqref{otimes of highest weight} yields
\begin{equation}\label{eq: decomp tensor power V}
    V^{\otimes n} \cong \bigoplus\limits_{\ell=0}^{n-1} L(\ell)^{\oplus\binom{n-1}{\ell}},
\end{equation}
where we recall the notation $L(\ell)\coloneqq L((n-\ell)\varepsilon_1 + \ell \varepsilon_2))$.

Thanks to the braiding, the representation $V^{\otimes n}$ is acted upon by the braid group on $n$ strands, and this action factors through the Hecke algebra $\Hecke_n\otimes_{\bZ[t]}\bQ(q)$, where $\bQ(q)$ is seen as a $\bZ[t]$-algebra via $t\mapsto q^{-2}$. To be more precise, the map
\[
\sigma_i\mapsto q^{-1}\id^{\otimes i}\otimes \check{R}_{V,V} \otimes \id^{\otimes(n-i-2)}
\]
defines an action of the braid group which factors through the quadratic relation of $\sigma_i^2 = (1-q^{-2})\sigma_i+q^{-2}\id$, which is the quadratic relation for the Hecke algebra with $q^{-2}$ instead of $t$.

A Schur--Weyl duality between $\Uq$ and the Hecke algebra $\Hecke_n$ has been shown independently by Moon \cite{Moon_HighestWeightVectors_2003} and Mitsuhashi \cite{Mitsuhashi_SchurWeylReciprocityQuantum_2006}.

\begin{thm}\label{Thm Schur--Weyl hecke algebra}
The algebra $\Uq$ and $\Hecke_n\otimes_{\bZ[t]}\bQ(q)$ centralize each other in $\End(V^{\otimes n})$; that is, the action of $\Hecke_n\otimes_{\bZ[t]}\bQ(q)$ generates $\End_{\Uq}(V^{\otimes n})$ and the action of $\Uq$ generates $\End_{\Hecke_n\otimes_{\bZ[t]}\bQ(q)}(V^{\otimes n})$.

More precisely, there is a decomposition
\begin{equation}
    V^{\otimes n}\cong\bigoplus\limits_{\ell=1}^{n-1}  L(l)\otimes \cS_{hk(n-\ell,\ell+1)}
\end{equation}
as $\Uq\otimes \Hecke_n$-representation, where $\cS_{hk(n-\ell,\ell+1)}$ is the Specht module for the hook partition with length $n-\ell$ and height $\ell+1$. 
\end{thm}

To complete the picture, we exhibit the matrix of $q^{-1}\check{R}_{V,V}$ on the basis 
\[\{ v^0\otimes v^0, {v^1\otimes v^0} , {q^{-1}v^0\otimes v^1} , q^{-1}v^1\otimes v^1\}.\]
We recover the matrix for the Burau representation of the braid group:
\begin{equation}\label{matrix of R in basis}
\begin{pmatrix}
    1 & 0 & 0 & 0
    \\
    0 & (1-q^{-2}) & t & 0 
    \\
    0 & 1 & 0 & 0
    \\
    0 & 0 & 0 & -q^{-2}
\end{pmatrix}.
\end{equation}
Note that the chosen basis is the tensor product of two different bases of $V$: $\{v^{0},v^{1}\}$ and $\{v^{0},q^{-1}v^{1}\}$.

The above surjective map $\Hecke_n\otimes_{\bZ[t]}\bQ(q)\to\End_{\Uq}(V^{\otimes n})$ is not injective for $n\geq 4$, since 
\[
\dim(\Hecke_n\otimes_{\bZ[t]}\bQ(q)) = n! > \binom{2(n-1)}{n-1} = \sum_{l=0}^{n-1} \binom{n-1}{l}^2 = \dim(\End_{\Uq}(V^{\otimes n})).
\]
The quotient of $\Hecke_n\otimes_{\bZ[t]}\bQ(q)$ by the kernel of this map, which is the isomorphic to $\End_{\Uq}(V^{\otimes n})$ is the \emph{super Temperley--Lieb algebra}. In \cite[Theorem III.5.4]{samson_phd}, it is shown that this kernel is generated by the element
\[
(\sigma_1-q)(\sigma_3-q)\sigma_2(\sigma_1+q^{-1})(\sigma_3+q^{-1}),
\]
or equivalently by
\[
(\sigma_1-1)(\sigma_3-1)\sigma_2(\sigma_1+t)(\sigma_3+t)
\]
since $t$ and $q^{-2}$ are identified. We also refer to \cite[Definition 4.9]{Sartori_CategorificationTensorPowers_2016} for definition of the super Temperley--Lieb algebra in the Kazhdan--Lusztig generators of the Hecke algebra.

\section{A Schur--Weyl duality with \texorpdfstring{$\neguglone$}{half gl(11)}}
\label{sec:Schur--Weyl}

The goal is now to enhance the Schur--Weyl duality between the Hecke algebra and $\Uq$ to a Schur--Weyl duality involving the loop Hecke algebra. We denote by $\Uqle$ the subalgebra of $\Uq$ generated by $K_1$, $K_2$ and $F$. We still work over the field $\bQ(q)$ which is a $\bZ[t]$-algebra via $t\mapsto q^{-2}$.

\subsection{An \texorpdfstring{$\LBr_n$}{LBrn}-representation via an \texorpdfstring{$R$}{R}-matrix for \texorpdfstring{$\Uqle$}{Uq<}}
Inspired by the notion of a twist in a quasi-triangular Hopf algebra, as introduced in \cite{Reshetikhin_MultiparameterQuantumGroups_1990}, we define, for any two $\Uq$ weight modules $M$ and $N$, a map $S_{M,N}\colon M\otimes N \to M\otimes N$ by
\[
S_{M,N}(m\otimes n) = q^{\mu_1\nu_2-\mu_2\nu_1}m\otimes n,
\]
where $m\in M_{\mu}$ and $N\in N_{\nu}$.

\begin{prop}\label{prop:intertwining-neg}
    We have that $S_{M,N}$ intertwines $\Delta$ and $\Delta^{\mathrm{op}}$ on $\Uqle$:
    \[
    \forall x\in \Uqle, S_{M,N}\circ\Delta(x) = \Delta^{\mathrm{op}}(x)\circ S_{M,N}.
    \]
\end{prop}

\begin{proof}
    The calculations for $x=K_1$ and $x=K_2$ are trivial. For $x=F$, $m\in M_\mu$ and $n\in N_\nu$ , we have:
    \begin{align*}
        S_{M,N}(\Delta(F)\cdot m\otimes n) &= S_{M,N}(Fm\otimes n + (-1)^{\lvert m \rvert}q^{-\mu_1-\mu_2}m\otimes Fn)\\
        &= q^{(\mu_1-1)\nu_2-(\mu_2+1)\nu_1}Fm\otimes n + (-1)^{\lvert m \rvert}q^{\mu_1(\nu_2+1)-\mu_2(\nu_1-1)-\mu_1-\mu_2}m\otimes Fn\\
        &= q^{\mu_1\nu_2-\mu_2\nu_1}(q^{-\nu_1-\nu_2}Fm\otimes n + (-1)^{\lvert m \rvert}m\otimes Fn)\\
        &= \Delta^{\mathrm{op}}(F)(S_{M,N}(m\otimes n)),
    \end{align*}
    the equalities being obtained from the definition of $S_{M,N}$ and of $\Delta$.
\end{proof}

\begin{rem}
    The map $S_{M,N}$ does not intertwine $\Delta$ and $\Delta^{\mathrm{op}}$ on $\Uq$. One may check that $S_{M,N}\circ\Delta(E) \neq \Delta^{\mathrm{op}}(E)\circ S_{M,N}$ but $S_{M,N}\circ\Delta^{\mathrm{op}}(E) = \Delta(E)\circ S_{M,N}$.
\end{rem}

We now set $\check{S}_{M,N} = \tau\circ S_{M,N}$, where $\tau$ is the super-twist, as usual.

\begin{prop}\label{prop: symmetric braiding}
    The map $\check{S}$ is a symmetric braiding on the category of $\Uqle$ weight representations.
\end{prop}

\begin{proof}
    It remains to check the hexagon axioms and that $\check{S}_{N,M}\circ \check{S}_{M,N}$ is the identity. Given three weight modules $M,N$ and $L$ and $m\in M_{\mu},n\in N_{\nu}$ and $\ell\in L_{\lambda}$, we have
    \[
        \check{S}_{M\otimes N,L}(m\otimes n \otimes \ell) = q^{(\mu_1+\nu_1)\lambda_{2}-(\mu_2+\nu_2)\lambda_1}(-1)^{(\lvert m \rvert + \lvert n \rvert)\lvert \ell \rvert}\ell\otimes m \otimes n,
    \]
    and
    \begin{align*}
        (\check{S}_{M,L}\otimes \id_{N})\circ (\id_{M}\otimes \check{S}_{N,L})(m\otimes n \otimes \ell) &= q^{\nu_{1}\lambda_{2}-\nu_{2}\lambda_1}(-1)^{\lvert n\rvert\lvert \ell\rvert}\check{S}_{M,L}(m\otimes \ell)\otimes n\\
        & = q^{\nu_1\lambda_2-\nu_2\lambda_1+\mu_1\lambda_2-\mu_2\lambda_1}(-1)^{\lvert n\rvert \lvert \ell \rvert + \lvert m \rvert \lvert \ell \rvert}\ell \otimes m \otimes n,
    \end{align*}
    which shows that the hexagon axiom $\check{S}_{M\otimes N,L}=(\check{S}_{M,L}\otimes \id_{N})\circ (\id_{M}\otimes \check{S}_{N,L})$ holds. The proof for the second hexagon axiom is similar. The axiom of symmetry is easy and omitted.
\end{proof}

The goal is now to show some mixed relations satisfied by $\check{R}$ and $\check{S}$. 

\begin{prop}\label{prop: mixed relations R and S}
    Given three $\Uqle$ weight modules $M,N$ and $L$, we have
    \begin{enumerate}
        \item $(\check{S}_{N,L}\otimes\id_M)\circ(\id_{N}\otimes \check{R}_{M,L})\circ(\check{R}_{M,N}\otimes \id_L) = (\id_{L}\otimes \check{R}_{M,N})\circ(\check{R}_{M,L}\otimes \id_{N})\circ(\id_{M}\otimes\check{S}_{N,L})$, 
        \item $(\check{R}_{N,L}\otimes\id_M)\circ(\id_{N}\otimes \check{S}_{M,L})\circ(\check{S}_{M,N}\otimes \id_L) = (\id_{L}\otimes \check{S}_{M,N})\circ(\check{S}_{M,L}\otimes \id_{N})\circ(\id_{M}\otimes\check{R}_{N,L})$.
    \end{enumerate}
\end{prop}

\begin{proof}
    We start by noticing that
    \begin{equation}
    \label{eq:checkS-theta}
    (\id_M\otimes \check{S}_{N,L}) \circ \Theta_{M,N\otimes L} = \Theta_{M,L\otimes N}\circ (\id_{M}\otimes\check{S}_{N,L}).
    \end{equation}
    Indeed, the maps $\Theta_{M,N\otimes L}$ and $\Theta_{M,L\otimes N}$ are induced by the action of $1\otimes \Delta(1) - (q-q^{-1})E\otimes \Delta(F)$, and \autoref{prop:intertwining-neg} implies that, for any $m\in M, n\in N$ and $\ell \in L$,
    \begin{multline*}
    (\id_{M}\otimes S_{N,L})\circ \Theta_{M,N\otimes L}(m\otimes n \otimes \ell) \\= m\otimes S_{N,L}(n\otimes \ell)-(q-q^{-1})E\otimes \Delta^{\mathrm{op}}(F)\cdot (m \otimes S_{N,L}(n\otimes \ell)).
    \end{multline*}
    It remains to apply the super-twist $\tau_{23}$ in order to obtain \eqref{eq:checkS-theta}. 

    Now, we have
    \begin{align*}
        (\check{S}_{N,L}\otimes\id_M)\circ \check{R}_{M,N\otimes L} &= \tau_{23}\circ \tau_{12} \circ (\id_M\otimes \check{S}_{N,L})\circ f_{M,N\otimes L} \circ \Theta_{M,N\otimes L} \\
        &= \tau_{23}\circ \tau_{12}\circ f_{M,N\otimes L} \circ (\id_M\otimes \check{S}_{N,L}) \circ \Theta_{M,N\otimes L} \\
        &= \tau_{23}\circ \tau_{12}\circ f_{M,L\otimes N} \circ \Theta_{M,L\otimes N} \circ (\id_M\otimes \check{S}_{N,L}) \\
        &= \check{R}_{M,N\otimes L} \circ (\id_M\otimes \check{S}_{N,L}),
    \end{align*}
    the second equality is due to the fact that $\id_M\otimes S_{N\otimes L}$ and $f_{M,N\otimes L}$ commutes and the third equality is a consequence of \eqref{eq:checkS-theta}. We finally obtain (1) using the hexagon axiom for the braiding $\check{R}$. The proof of (2) is similar and ommited.
\end{proof}

Now consider the assignment 
\begin{equation}\label{Rep from braiding}
    \Psi_n: \LBr_n \rightarrow \End_{\Uqle}(V^{\otimes n}): \left\lbrace \begin{array}{l}
        \sigma_{i}\mapsto q^{-1}\id_{V}^{\otimes i-1}\otimes \check{R}_{V,V} \otimes \id_{V}^{\otimes(n-i-1)} \\
        \rho_{i}\mapsto \id_{V}^{\otimes i-1}\otimes \check{S}_{V,V}\otimes\id_{V}^{\otimes(n-i-1)}
    \end{array} \right.
\end{equation}

\begin{thm}\label{DMR rep is local}
The triple $(V, \check{R}, \check{S})$ is a loop braided vector space, that is, the map $\Psi_n$ constructed in \eqref{Rep from braiding} endows the vector space $V$ with a well-defined action of $\LBr_n$. Moreover, $\Psi_n$ factors through $\LH_n\otimes_{\bZ[t]}\bQ(q)$.
\end{thm}

\begin{proof}
We first show that $\Psi_n$ is a well-defined map on $\LBr_n$, by checking the relations of the loop braid group. \autoref{thm: R is braiding} and \autoref{prop: symmetric braiding}, applied to the vector representation $V$, show that the non-mixed braid relations \eqref{eq:sigma_rho_braid_same} are satisfied. \autoref{prop: mixed relations R and S}, still applied to the vector representation $V$, shows that the mixed braid relations \eqref{eq:sigma_rho_braid_mixed} are satisfied. 

To see that $\Psi_n$ factors through $\LH_n\otimes_{\bZ[t]}\bQ(q)$, we need to check the quadratic relations \eqref{eq:sigma_rho_same_label_same} and \eqref{eq:sigma_rho_same_label_mixed}, which we compute directly on $V\otimes V$. In the basis
\[
\{v^0\otimes v^0, v^1\otimes v^0+q^{-1} v^0 \otimes v^1, v^1\otimes v^0 -q v^0 \otimes v^1, v^1\otimes v^1 \}
\]
of $V\otimes V$, which is a basis realizing the decomposition $V\otimes V \simeq L(2\varepsilon_1)\oplus L(\varepsilon_1+\varepsilon_2)$, we find that the matrices of $q^{-1}\check{R}_{V,V}$ and of $\check{S}_{V,V}$ are respectively
\[
\begin{pmatrix}
    1 & 0 & 0 & 0\\
    0 & 1 & 0 & 0\\
    0 & 0 & -q^{-2} & 0\\
    0 & 0 & 0 & -q^{-2}
\end{pmatrix}
\quad\text{ and }\quad
\begin{pmatrix}
    1 & 0 & 0 & 0\\
    0 & 1 & -q(q-q^{-1}) & 0\\
    0 & 0 & -1 & 0\\
    0 & 0 & 0 & -1
\end{pmatrix}
\]
The quadratic relations then follow immediately from a matrix calculation.
\end{proof}

\begin{rem}
\label{rem:LBurau_is_DMR}
    The representation $\LBurau_n$ defined in \eqref{DMR repr} (by \cite{DMR_GeneralisationsHeckeAlgebras_2023}) coincides with $\Psi_n$. Indeed, the basis $\{ v_{\varepsilon_1}^0\otimes v_{\varepsilon_1}^0, v_{\varepsilon_1}^1\otimes v_{\varepsilon_1}^0 , q^{-1}v_{\varepsilon_1}^0\otimes v_{\varepsilon_1}^1, q^{-1}v_{\varepsilon_1}^1\otimes v_{\varepsilon_1}^1\}$ of $V\otimes V$, the respective matrices of $q^{-1}\check{R}_{V,V}$ and $\check{S}_{V,V}$ are given by
    \[
    \begin{pmatrix}
    1 & 0 & 0 & 0
    \\
    0 & (1-q^{-2}) & q^{-2} & 0 
    \\
    0 & 1 & 0 & 0
    \\
    0 & 0 & 0 & -q^{-2}
    \end{pmatrix}
    \quad\text{ and }\quad
    \begin{pmatrix}
        1 & 0 & 0 & 0\\
        0 & 0 & 1 & 0\\
        0 & 1 & 0 & 0\\
        0 & 0 & 0 & -1
    \end{pmatrix}.
    \]
    These matrices correspond to the matrices in \cite{DMR_GeneralisationsHeckeAlgebras_2023} for $q^{-2}$ playing the role of $t$. 
\end{rem}

\subsection{Schur--Weyl duality for the loop Hecke algebra}
Recall that $V$ denotes the vector representation $L(\varepsilon_1)$  of $\Uq$. The aim of the remainder of this section is to obtain the following Schur--Weyl type statement.
\begin{thm}\label{the rep is iso}
The morphism  $\Psi_n$ defined in \eqref{Rep from braiding} yields a $\bQ(q)$-algebra isomorphism from $\LH_n\otimes_{\bZ[t]}\bQ(q)$ to $\End_{\Uqle}(V^{\otimes n}).$
\end{thm}

To obtain \autoref{the rep is iso} we need to compute $\dim \End_{\Uqle}(V^{\otimes n})$. To do so, we will decompose $\End_{\Uqle}(V^{\otimes n})$ into smaller pieces which will be proven in \autoref{subsec: W-M decomp} to be meaningful pieces of the Wedderburn--Mal'cev decomposition of the endomorphism ring. 

First we note that the decomposition in \eqref{eq: decomp tensor power V} is not one of simple $\Uqle$-modules as each $L(\ell) = L((n- \ell)\varepsilon_1+\ell \varepsilon_2)$ has a $\Uqle$-submodule. Since $F^2=0$, it is easily seen that the simple weight $\Uqle$-representations are one-dimensional. We denote by $L'(\lambda)$ the one-dimensional representation of weight $\lambda$. The following lemma describes the structure of the restriction to $\Uqle$ of the two-dimensional representations of $\Uq$. 

\begin{lem}\label{lem:restriction_neg}
    Let $\lambda\in P$ such that $\langle h_1+h_2,\lambda \rangle \neq 0$. The restriction to $\Uqle$ of the $\Uq$-representation $L(\lambda)$ is indecomposable but non-semisimple and has a composition series
    \[
    \{0 \} \subsetneq \bQ (q) v_{\lambda}^1 \subsetneq L(\lambda),
    \]
    where $\bQ (q) v_{\lambda}^1 \simeq L'(\lambda - \varepsilon_1+\varepsilon_2)$ and $L(\lambda)/\bQ (q) v_{\lambda}^1 \simeq L'(\lambda)$ as $\Uqle$-representations.
\end{lem}

\begin{proof}
The weights of $v_{\lambda}^0$ and $v_{\lambda}^1$ are recorded in \eqref{action on 1st basis vector} and \eqref{action on 2nd basis vector}. From these equations we also see that $\bQ (q) v_{\lambda}^1$ is indeed a $\Uqle$-submodule. As it is one-dimensional it is simple and isomorphic to $L'(\lambda - \varepsilon_1+\varepsilon_2)$ since $v_{\lambda}^{1}$ is of weight $\lambda-\varepsilon_1+\varepsilon_2$. The quotient $L(\lambda)/\bQ (q) v_{\lambda}^1$ is also one-dimensional and isomorphic to $L'(\lambda)$ since $v_{\lambda}^{0}$ is of weight $\lambda$.

A direct computation shows that $\bQ (q) v_{\lambda}^1$ is the only non-trivial $\Uqle$ subrepresentation of $L(\lambda)$. Since $L(\lambda)$ is of dimension $2$, $L(\lambda)$ is therefore indecomposable as an $\Uqle$ representation.
\end{proof}

Using \eqref{eq: decomp tensor power V} we obtain that 
\begin{equation}\label{First decomp End}
\End_{\Uqle}(V^{\otimes n})
=\bigoplus_{0 \leq \ell,k \leq n-1} \Hom_{\Uqle} \left( L(\ell), L(k) \right)^{\oplus \binom{n-1}{\ell} \binom{n-1}{k}}.
\end{equation}

We will now determine the dimension of each summand in \eqref{First decomp End}.

\begin{lem}\label{L:intertwiner_negative}
    Let $\lambda,\mu\in P$ such that $\langle h_1+h_1,\lambda\rangle \neq 0$ and $\langle h_1+h_1,\mu\rangle \neq 0$. Then
    \[
    \dim(\Hom_{\Uqle}(L(\lambda),L(\mu))) = 
    \begin{cases}
        1 & \text{if }\lambda=\mu\text{ or } \mu-\lambda=\varepsilon_{1}-\varepsilon_{2}\\
        0 & \text{otherwise}.
    \end{cases}
    \]
\end{lem}

\begin{proof}
    Let $\varphi\in\Hom_{\Uqle}(L(\lambda),L(\mu))$. The representation $L(\lambda)$ has basis vectors $v_{\lambda}^{0}$ and $v_{\lambda_{1}}$ of respective weights $\lambda$ and $\lambda-\varepsilon_{1}+\varepsilon_{2}$. Similarly, the representation $L(\mu)$ has basis vectors $v_{\mu}^{0}$ and $v_{\mu}^{1}$ of respective weights $\mu$ and $\mu-\varepsilon_{1}+\varepsilon_{2}$.

    First, if $\lambda\neq \mu$ and $\lambda\neq\mu \pm (\varepsilon_{1}-\varepsilon_{2})$, then $L(\lambda)$ and $L(\mu)$ have no weights in common and $\varphi$ must be $0$.

    Now, suppose that $\mu = \lambda - \varepsilon_{1} +\varepsilon_{2}$. Therefore, for weight reasons, there exists $z\in \bQ(q)$ such that $\varphi(v_{\lambda}^{1})=z v_{\mu}^{0}$ and $\varphi(v_{\lambda}^{0})=0$. But $[2]v_{\lambda}^{1}=F\cdot v_{\lambda}^{1}$. Since $\varphi$ commutes with the action of $F$, we have 
    \[
    [2]zv_{\mu}^{0} = \varphi([2]v_{\lambda}^{1}) = \varphi(F\cdot v_{\lambda}^{0}) = F\cdot \varphi(v_{\lambda}^{0}) = 0.
    \]
    Since $[2]$ is invertible in $\bQ(q)$, we obtain that $z=0$ and therefore $\varphi=0$.

    Finally, suppose that $\mu = \lambda + \varepsilon_{1} -\varepsilon_{2}$. As before, for weight reasons, there exists $z\in \bQ(q)$ such that $\varphi(v_{\lambda}^{1})=z v_{\mu}^{0}$ and $\varphi(v_{\lambda}^{0})=0$. We check that such a map is well defined for any $z\in\bQ(q)$ and commutes with the action of $\Uqle$.
\end{proof}

As a corollary, we obtain the following proposition:

\begin{prop}\label{decomp End over positive part}
   For any $n \in \bN_{\geq 1}$ one has that 
   \[
    \End_{\Uqle}(V^{\otimes n}) =  \bigoplus_{\ell =0}^{n-1} \End_{\Uqle} (L (\ell))^{\oplus \binom{n-1}{\ell}^{2}} \oplus
  \bigoplus_{\ell=1}^{n-1} \Hom_{\Uqle} \left(L (\ell),L (\ell-1) \right)^{\oplus \binom{n-1}{\ell}\binom{n-1}{\ell-1}}
   \]
with every summand a $1$-dimensional $\bQ(q)$-vector space. Furthermore:
\[\dim \End_{\Uqle}(V^{\otimes n}) = \binom{2n-1}{n}.\]
\end{prop}

\begin{proof}
Everything but the value of the dimension follows from \autoref{L:intertwiner_negative}. By taking dimensions, we get:
\[
\dim \End_{\Uqle}(V^{\otimes n})  = \sum\limits_{\ell=0}^{n-1} \binom{n-1}{\ell}^2 + \sum\limits_{\ell = 1}^{n-1} \binom{n-1}{\ell}\binom{n-1}{\ell -1}
= \binom{2(n-1)}{n-1} + \binom{2(n-1)}{n-1}
\]
where in the second equation we have used Vandermonde's identity. Pascal's rule shows that the last expression equals $\binom{2n-1}{n}$, finishing the proof.
\end{proof}

\begin{exe}\label{E:2-fold-tensor}
    Take $n=2$. Then $L(0) = L(2\varepsilon_1)$, $L(1) = L(\varepsilon_1+\varepsilon_2)$ and
    \[
    V\otimes V = L(2\varepsilon_1)\oplus L(\varepsilon_1+\varepsilon_2),
    \]
    where the direct sum is only a direct sum of $\Uq$-modules.
    For every element $z\in\bQ(q)$, the map $\varphi_z\colon v_{\varepsilon_1+\varepsilon_2}^0\mapsto z v_{2\varepsilon_1}^1$ and zero elsewhere defines a map $\varphi_z \colon L(\varepsilon_1+\varepsilon_2)\to L(2\varepsilon_1)$. It permutes with the $\Uqle$-action, but not with the $\Uq$-action.
    This can be pictured as, where rightward (resp.\ leftward dashed) arrows denote the action of $F$ (resp.\ $E$), and the vertical arrow is $\varphi_z$:
    \[
    \begin{tikzcd}
        L(0)= & v_{2\varepsilon_1}^0 \ar[r,bend left=10pt,"{[2]_q}"] 
        & v_{2\varepsilon_1}^1 \ar[l,bend left=10pt,"1",dashed] &
        \\
        &
        & v_{\varepsilon_1+\varepsilon_2}^0 \ar[r,bend left=10pt,"{[2]_q}"] \ar[u,"z"]  & v_{\varepsilon_1+\varepsilon_2}^1
        \ar[l,bend left=10pt,"1",dashed] 
        & =L(1)
    \end{tikzcd}
    \]
    Write $\pi_{L(l)}$ the projection onto $L(l)$.
    Direct computation shows that $\Psi(D) = \pi_{L(1)}+\varphi_{q^2}$ and $\Psi(U) = \pi_{L(0)}+\varphi_{1}$, where we use the presentation of the loop Hecke algebra given in \autoref{defn:var_loop_hecke}, following \autoref{main thm: eq rel}. In particular, we have $\Psi(DU) = \pi_{L(1)}\varphi_{1} + \varphi_{q^2}\pi_{L(0)} = 0 $ and $\Psi(UD)=\pi_{L(0)}\varphi_{q^2}+\varphi_{1}\pi_{L(1)} = \varphi_{q^2+1} = \Psi(U+D-1)$, as expected.

    Note that if we were to work with $U_q^{\geq 0}(\glone)$ instead, then the map $\varphi_z$ would have to go in the opposite direction, and we would have the opposite relations between $U$ and $D$.
    This suggests that in order to get the analogue story for $U_q^{\geq 0}(\glone)$, one needs to switch the roles of $U$ and $D$.
\end{exe}

We have now the necessary tools to prove \autoref{the rep is iso}.

\begin{proof}[Proof of \autoref{the rep is iso}]
    By \autoref{rem:LBurau_is_DMR} the representation $\Psi_n$ coincides with $\LBurau_n$ defined in \eqref{DMR repr}. It was proven in \cite[Theorem 5.8]{DMR_GeneralisationsHeckeAlgebras_2023} that $\dim \im(\LBurau_n)  = \binom{2n-1}{n}$ which equals $\dim \End_{\Uqle}(V^{\otimes n})$ by \autoref{decomp End over positive part}. Hence $\Psi_n = \LBurau_n$ is surjective. Now, by \autoref{thm:equivalence_presentations}, $\LH_n\otimes_{\bZ[t]}\bQ(q)$ is isomorphic to $\varLH_n\otimes_{\bZ}\bQ(q)$, where $\bQ(q)$ is a $\bZ[t]$-algebra via $t\mapsto q^{-2}$. Solely using that the $\varLH_n$-reduced words $\Red(\varLH_n)$ generate $\varLH_n$ (cf.\ \autoref{lem:reduced_words_generate}), we obtain that $\dim(\LH_n\otimes_{\bZ[t]}\bQ(q)) \leq |\Red(\varLH_n) |$ with the upper bound equal to ${2n-1 \choose n}$ by \autoref{thm:cardinality_reduced_words}. Thus, comparing dimensions with the co-domain, we see that $\Psi_n$ must be an isomorphism.
\end{proof}

\begin{rem}\label{Remark other lin.ind proof}
Note that the proof \autoref{the rep is iso} yields an alternative way to obtain that the $\varLH_n$-reduced words are linearly independent over $\mathbb{Q}(q).$
\end{rem}

\section{A ring and representation theoretical perspective on loop Hecke algebra}
\label{sec:ring_approach_loop_Hecke}

In this section we make use of \autoref{the rep is iso} to study the loop Hecke algebra over the field $\bQ(q)$ in more detail through the isomorphic algebra $\End_{\Uqle}(V^{\otimes n})$. We still equip $\bQ(q)$ with the $\bZ[t]$-algebra structure obtained via $t\mapsto q^{-2}$.

\subsection{Description of the semisimple part and Jacobson radical}\label{subsec: W-M decomp}

If $A$ is a finite-dimensional algebra over a perfect field $\Bbbk$, then the theorem of Wedderburn--Mal'cev yields that there exists a maximal semisimple subalgebra $B_{ss}$ in $A$ such that 
\begin{equation}\label{M-W decomposition}
A \cong B_{ss} \oplus Jac(A)
\end{equation}
as a $\Bbbk$-vector space. In \eqref{M-W decomposition} $Jac(A)$ denotes the Jacobson radical of $A$ and $B_{ss}$ is unique up to an algebra automorphism of $A$. The aim of this section is to describe the constituents of \eqref{M-W decomposition} for the loop Hecke algebra. More precisely, we will use \autoref{the rep is iso} to instead describe the Wedderburn--Mal'cev decomposition of the isomorphic algebra $\End_{\Uqle}(V^{\otimes n})$. 

\subsubsection{The Wedderburn--Mal'cev decomposition and application}

We consider the following constituents of the decomposition obtained in \autoref{decomp End over positive part}:
\[N_{\ell} := \End_{\Uqle}(L(\ell))^{\oplus\binom{n-1}{\ell}^2} \text{ and } R_{\ell} := \Hom_{\Uqle}(L(\ell),L(\ell-1))^{\oplus \binom{n-1}{\ell}\binom{n-1}{\ell-1}}.\]

\begin{thm}\label{W-M decompisition End}
For any $n \in \bN_{\geq 1}$ and with notations as above, we have that:
\begin{enumerate}
    \item $\bigoplus_{\ell=0}^{n-1} N_{\ell}$ is a maximal semisimple subalgebra with $N_{\ell}$ a simple factor,
    \item $\bigoplus_{\ell=0}^{n-1} N_{\ell}= \End_{\Uq}(V^{\otimes n})$ is isomorphic to the super Temperley--Lieb algebra, 
    \item $J\big(\End_{\Uqle}(V^{\otimes n})\big) = \bigoplus_{\ell=1}^{n-1}R_{\ell}$ and has square-zero. 
\end{enumerate}
Moreover, for $0 \leq i ,j \leq n-1$, denoting by $e_{i}$ the identity of $N_{i}$, we have:
\[ e_{j} \End_{\Uqle}(V^{\otimes n}) e_i = \left\lbrace \begin{array}{ll}
N_i& \text{ if } j =  i,\\
R_i&  \text{ if } j =  i-1,\\
0 & \text{ else.}\\
\end{array}\right.\]
\end{thm}

Before proving \autoref{W-M decompisition End} we discuss some consequences.

\begin{rem}
A consequence of \autoref{W-M decompisition End} and \autoref{Thm Schur--Weyl hecke algebra} is that the Hecke algebra $\Hecke_n\otimes_{\bZ[t]}\bQ(q)$ surjects onto $\End_{\Uqle}(V^{\otimes n})/J\big( \End_{\Uqle}(V^{\otimes n})\big)$. Now recalling that the Jacobson radical is preserved under morphisms and using \autoref{the rep is iso} we obtain that
\[
\Hecke_n\otimes_{\bZ[t]}\bQ(q) \twoheadrightarrow \LH_n\otimes_{\bZ[t]}\bQ(q)/ J(\LH_n\otimes_{\bZ[t]}\bQ(q)).
\]
It would be interesting to have an explicit description of a basis for the pieces $\Psi^{-1}(N_i)$ and $\Psi^{-1}(R_i)$.
Note that it follows from \autoref{thm: R is braiding} that 
\[\Psi(\sigma_{i}) \in \bigoplus_{\ell=0}^{n-1}N_{\ell} = \End_{\Uq}(V^{\otimes n})\]
for every $0 \leq i \leq n-1$.
\end{rem}

The elements $\{e_0, \ldots, e_{n-1}\}$ form a set of orthogonal idempotents that add up to the identity: $\id_{V^{\otimes n}} = e_0 + \cdots +e_{n-1}.$ Thanks to this one can consider the associated Peirce decomposition 
\[\End_{\Uqle}(V^{\otimes n}) = \bigoplus_{0\leq i,j \leq n-1}e_j \End_{\Uqle}(V^{\otimes n}) e_i\]
which allows one to consider $\End_{\Uqle}(V^{\otimes n})$ as a $n\times n$ matrix algebra with $e_j \End_{\Uqle}(V^{\otimes n}) e_i$ as entry $(i,j)$. Associated with this one can consider the $n\times n$-matrix $M=(M_{i,j})_{0\leq i,j \leq n-1}$ with
\[M_{i,j} := \dim e_j \End_{\Uqle}(V^{\otimes n}) e_i.\]
In \cite{DMR_GeneralisationsHeckeAlgebras_2023} the matrix $M$ is referred to as encoding the ``structure of the algebra''. In \cite[Section 6 \& Conjecture 6.4]{DMR_GeneralisationsHeckeAlgebras_2023} a conjecture was formulated for the form of $M$. As a corollary of the above we now confirm their conjecture.

\begin{cor}\label{value peirce matrix}
Denote, for $0 \leq  \ell \leq n-1$, by $w_{\ell} := \binom{n-1}{\ell}$ the $\ell^{th}$ entry of the $n^{th}$ row of Pascal's triangle. With notations as above, we have that 
\[M = \begin{pmatrix}
   w_0^2 & 0 &  0 &\cdots & 0 & 0\\
   w_0 w_1& w_1^2 & 0 & \cdots & 0 & 0 \\
   0 & w_1w_2 & w_2^2 &  \ddots & \vdots & 0 \\
   0 & 0 & \ddots & \ddots & 0 & \vdots \\
   \vdots &  &\ddots &\ddots & w_{n-2}^2 & 0\\
   0 & \cdots & & 0&w_{n-2}w_{n-1} & w_{n-1}^2
\end{pmatrix} = D. \begin{pmatrix}
    1 & &&&& \\
    1 & 1 &&&&\\
    & 1 & 1 & &&\\
    & & 1 & 1 &&\\
    & & & \ddots & \ddots&\\
    & & & & 1 & 1
\end{pmatrix} .D\]
where $D := \text{diag}(w_0, \cdots, w_{n-1})$.
\end{cor}
The second equality in \autoref{value peirce matrix} was added as it is in this form that Damiani--Martin--Rowell formulated their conjecture, see matrix $\mathcal{M}^p_n$ in \cite[Section 6]{DMR_GeneralisationsHeckeAlgebras_2023}.
\begin{proof}[Proof of \autoref{value peirce matrix}]
    In \autoref{W-M decompisition End} the form of the corners $e_{j} \End_{\Uqle}(V^{\otimes n}) e_i$ has been described. Now recall that by \autoref{decomp End over positive part} 
    \[\dim  \End_{\Uqle}(L(\ell)) = \dim \Hom_{\Uqle}(L(\ell^{\prime}),L(\ell^{\prime}-1)) = 1 \]
for all $0 \leq \ell \leq n-1$ and $ 1 \leq \ell^{\prime} \leq n-1.$ Hence $\dim N_{\ell} = \binom{n-1}{\ell}^2$ and $\dim R_{\ell^{\prime}} =\binom{n-1}{\ell^{\prime}}\binom{n-1}{\ell^{\prime}-1}$, which proves the first equality. The second equality is a direct matrix multiplication.
\end{proof}

\begin{rem}
In \cite[Conjecture 6.4]{DMR_GeneralisationsHeckeAlgebras_2023} the form of the matrix $M$ was only implicitly formulated for the loop Hecke algebra $\LH_n$ itself. Instead, they conjectured that the structure of certain quotients of $\LH_n$ was given by some truncations of the matrix $M$. In \autoref{section: special quot} we come back to that conjecture.
\end{rem}

\subsubsection{Towards the proof of \autoref{W-M decompisition End}}\label{subsection proof W-M}

To start, recall \eqref{eq: decomp tensor power V} saying that 
\[
 V^{\otimes n} \cong \bigoplus\limits_{\ell=0}^{n-1} L(\ell)^{\oplus\binom{n-1}{\ell}} .
\]
For the remaining of the section we order the copies of $V^{\otimes n}$ from $1$ till $\binom{n-1}{\ell}$. Accordingly denote 
\[L_{\ell,k} := k^{th} \text{ copy of } L(\ell) \text{ in } V^{\otimes n}\]
with $1 \leq k \leq \binom{n-1}{\ell}$. Note that \autoref{decomp End over positive part} now in fact tells that 
\begin{equation}\label{End decomp refined not}
\End_{\Uqle}(V^{\otimes n}) = 
\bigoplus_{\ell=0}^{n-1} \bigoplus\limits_{1 \leq k,k^{\prime}\leq \binom{n-1}{\ell}} \Hom_{\Uqle}(L_{\ell, k},L_{\ell, k^{\prime}}) 
\oplus 
\bigoplus_{\ell^{\prime}=1}^{n-1} \bigoplus_{r=1}^{\binom{n-1}{\ell^{\prime}}} \bigoplus_{r^{\prime}=1}^{\binom{n-1}{\ell^{\prime}-1}} \Hom_{\Uqle}(L_{\ell^{\prime}, r},L_{\ell^{\prime}-1, r^{\prime}})
\end{equation}

Now for $\varphi \in \Hom_{\Uqle}(L_{\ell, k},L_{\ell, k^{\prime}}) $ and $\psi \in \Hom_{\Uqle}(L_{\ell^{\prime}, r},L_{\ell^{\prime}-1, r^{\prime}})$ we directly see from \eqref{End decomp refined not} that 
\begin{equation}\label{left multiplication}
    \varphi \circ \psi \neq 0 \Leftrightarrow \ell = \ell^{\prime}-1 \text{ and } k = r^{\prime}.
\end{equation}
In that case $\varphi \circ \psi \in  \Hom_{\Uqle}(L_{\ell^{\prime}, r},L_{\ell^{\prime}-1, k^{\prime}}).$
Similarly,
\begin{equation}\label{right multiplication}
    \psi \circ  \varphi \neq 0 \Leftrightarrow \ell^{\prime} = \ell \text{ and } r = k^{\prime}
\end{equation}
in which case $\psi \circ \varphi \in  \Hom_{\Uqle}(L_{\ell, k},L_{\ell-1, r^{\prime}}).$

\begin{proof}[Proof of \autoref{W-M decompisition End}]
In terms of the notations above we have by definition that
\[N_{\ell} = \bigoplus\limits_{1 \leq k,k^{\prime}\leq \binom{n-1}{\ell}} \Hom_{\Uqle}(L_{\ell, k},L_{\ell, k^{\prime}}) \text{ and } R_{\ell^{\prime}}= \bigoplus_{r=1}^{\binom{n-1}{\ell^{\prime}}} \bigoplus_{r^{\prime}=1}^{\binom{n-1}{\ell^{\prime}-1}} \Hom_{\Uqle}(L_{\ell^{\prime}, r},L_{\ell^{\prime}-1, r^{\prime}}).\]
From \eqref{left multiplication} and \eqref{right multiplication} it directly follows that $\bigoplus_{\ell^{\prime}=1}^{n-1}R_{\ell^{\prime}}$ is a two-sided ideal of $\End_{\Uqle}(V^{\otimes n})$ whose square is zero (i.e.\ the product of any two elements is zero). Those equations also imply that $N_{\ell_1}.N_{\ell_2}=0$ for $0 \leq \ell_1 \neq \ell_2 \leq n-1.$ Hence the vector space decomposition $\bigoplus_{\ell=0}^{n-1} N_{\ell}$ is also one of rings and in particular $N_{\ell}$ is a two-sided ideal of $\bigoplus_{\ell=0}^{n-1} N_{\ell}$. Since all spaces $\Hom_{\Uqle}(L_{\ell, k},L_{\ell, k^{\prime}})$ are $1$-dimensional, by \autoref{decomp End over positive part}, \eqref{left multiplication} and \eqref{right multiplication} also imply that the two-sided ideal (inside $\bigoplus_{\ell=0}^{n-1} N_{\ell}$) generated by any $0 \neq \varphi \in \Hom_{\Uqle}(L_{\ell, k},L_{\ell, k^{\prime}})$ is the full of $N_{\ell}.$ From this we infer that $N_{\ell}$ is a simple ring, finishing the proof of part (1) of the statement, except the maximality. The latter will follow if we prove that the vector space complement $\bigoplus_{\ell^{\prime}=1}^{n-1}R_{\ell^{\prime}}$, see \eqref{End decomp refined not}, equals the Jacobson radical. 

Next note that since $\bigoplus_{\ell^{\prime}=1}^{n-1}R_{\ell^{\prime}}$ is a nil-ideal, being of square-zero, it is contained in the Jacobson radical. Recall that for Artinian algebras the Jacobson radical is also characterized as the smallest (two-sided) ideal such that the corresponding quotient is semisimple. Now as
\[\End_{\Uqle}(V^{\otimes n}) / \bigoplus_{\ell^{\prime}=1}^{n-1}R_{\ell^{\prime}} \cong  \bigoplus_{\ell=0}^{n-1} N_{\ell}\]
is semisimple, by the statement obtained earlier, we obtain the other inclusion and finishing both statements (3) and (1).

Concerning the moreover part, note that 
\[
e_i = \sum\limits_{1 \leq k,k^{\prime}\leq \binom{n-1}{\ell}} \id_{k,k^{\prime}}
\]
where $\id_{k,k^{\prime}}$ is the canonical identification of $L_{\ell,k}$ and $L_{\ell, k^{\prime}}.$ With this description and using \eqref{left multiplication} and \eqref{right multiplication} the moreover part follows directly.

Finally, for statement (2) note that \autoref{L:intertwiner_negative} shows that $\End_{\Uqle}(L(\ell)) = \End_{\Uq}(L(\ell))$. Since the modules $L(\ell)$ are simple $\Uq$-modules, the latter and \eqref{eq: decomp tensor power V} implies that
\[
\bigoplus_{\ell=0}^{n-1} N_{\ell} = \bigoplus_{\ell=0}^{n-1} \End_{\Uq}(L(\ell)^{\oplus\binom{n-1}{\ell}}) = \End_{\Uq}(V^{\otimes n}).
\]
Now it was proven in \cite[Section 4]{Sartori_CategorificationTensorPowers_2016} that $\End_{\Uq}(V^{\otimes n})$ is isomorphic to the super Temperley--Lieb algebra.
\end{proof}

\subsection{A description of the Ext-quiver and Cartan matrix}

Consider the primitive central idempotents $\{ e_0, \ldots, e_{n-1} \}$ described in \autoref{W-M decompisition End}. As the idempotents are orthogonal and $1 = \sum_{i=0}^{n-1} e_i$ one has the associated decomposition in blocks:
\[
\End_{\Uqle}(V^{\otimes n}) = \bigoplus_{i=0}^{n-1} \End_{\Uqle}(V^{\otimes n}) . e_i
\]
A block $\End_{\Uqle}(V^{\otimes n}) . e_i$ can be further decomposed into
\[\End_{\Uqle}(V^{\otimes n}) . e_i = P_i^{\oplus m_i}\]
where $\{P_0, \ldots, P_{n-1} \}$ are representatives of the isomorphism classses of indecomposable projective (left) modules. Their head $P_i/J(P_i)$ is simple which we denote by $S_i$. Below, in \autoref{ind proj and simples}, we will describe these modules explicitly and also the multiplicities $m_i$.

For a finite dimensional algebra $A$, the {\it Cartan matrix $C(A)$} is an $n\times n$-matrix whose $i^{th}$ row encode the multiplicity of each simple module in the composition series of $P_i$:
\[C(A)_{i,j} := \dim \Hom_A(P_j,P_i).\]
The Ext-quiver $Q_A$ encode a complementary piece of information: its vertices are given by the simple modules $S_i$ and 
\[\# S_i \rightarrow S_j :=  \dim \Ext^1_A(S_i,S_j).\]
In \cite[Theorem 5.7 \& Corollary 6.1]{DMR_GeneralisationsHeckeAlgebras_2023} both were described for the algebra $SP_n := \LBurau_n(\Bbbk[\LBr_n])$ with $\LBurau_n$ defined in \eqref{DMR repr}. Over the field $\Bbbk=\bQ(q)$, it follows from \autoref{DMR rep is local} and \autoref{the rep is iso} that $SP_n \cong \End_{\Uqle}(V^{\otimes n})$ and hence their description also hold for the latter. We now give a direct proof for the Cartan matrix and Ext-quiver of $\End_{\Uqle}(V^{\otimes n})$.

\begin{thm}\label{Cartan and Ext quiver}
Let $A := \End_{\Uqle}(V^{\otimes n})$ for $n \in \bN_{\geq 1}$. Then the Cartan matrix is
\[C(A) = \begin{pmatrix}
    1 & && & \\
    1 & 1 &&&\\
    & 1 & 1  &&\\
    &  & \ddots & \ddots&\\
    & & & 1 & 1
\end{pmatrix}\]
and the Ext-quiver $Q_A$ with relations is the $A_n$-quiver with the composition of two arrows zero:
\begin{center}
\begin{tikzpicture}
  \node (0) at (0,0) {$n-1$};
  \node (1) at (2,0) {$n-2$};
  \node (2) at (4,0) {$n-3$};
  \node (3) at (5,0) {$\cdots$};
  \node (4) at (6,0) {$2$};
  \node (5) at (8,0) {$1$};
  \node (6) at (10,0) {$0$};
  
  \draw[->] (0) -- (1);
  \draw[->] (1) -- (2);
   \draw[->] (4) -- (5);
     \draw[->] (5) -- (6);

  \draw[dashed] (0) to[bend left=20] (2);
  \draw[dashed] (4) to[bend left=20] (6);
\end{tikzpicture}
\end{center}
\end{thm}

\begin{rem}
Recall that any finite-dimensional algebra $A$ is Morita-equivalent to a path algebra $kQ_A/I$ with relations where $Q_A$ is the Ext-quiver and $I$ some (admissible) ideal $I$. Thus \autoref{Cartan and Ext quiver} describes the Morita equivalence class of $ \End_{\Uqle}(V^{\otimes n})$. It is well-known that the path algebra with relations given in the statement is of finite representation type. Consequently, although the Loop Hecke algebra $\LH_n\otimes_{\bZ[t]}\bQ(q)$ is not semisimple, in contrast to the Hecke algebra $\Hecke_n\otimes_{\bZ[t]}\bQ(q)$, we obtained via \autoref{the rep is iso} and \autoref{Cartan and Ext quiver} that it is still of finite representation type.
\end{rem}

The proof of \autoref{Cartan and Ext quiver} will quickly follow once we construct explicitly the simple and indecomposable projective modules. To do so, we use the notations from \autoref{subsection proof W-M} and the statements obtained there.

Define for $1 \leq \ell \leq n-1$ and $1 \leq k \leq \binom{n-1}{\ell}$:
\begin{equation}\label{Description left simple and ind proj}
S_{\ell, k} := \bigoplus_{r =1}^{\binom{n-1}{\ell-1}} \Hom_{\Uqle}(L_{\ell, k}, L_{\ell-1,r}) \text{ and } P_{\ell, k} := S_{\ell,k} \oplus \bigoplus_{k^{\prime}=1}^{\binom{n-1}{\ell}} \Hom_{\Uqle}(L_{\ell, k}, L_{\ell, k^{\prime}})
\end{equation}
If $\ell = 0$, then $\binom{n-1}{\ell} = 1$ and we define $P_{0,1}  := \Hom_{\Uqle}(L_{0, 1}, L_{0, 1})$ and $S_{0,1} = \{ 0 \}$.
Note that, in terms of \autoref{W-M decompisition End}, $S_{\ell,k}$ corresponds to all endomorphisms in the radical which are zero outside a fixed $L_{\ell, k}$, and $P_{\ell,k}$ to all morphisms in $\End_{\Uqle}(V^{\otimes n})$ outside a fixed $L_{\ell, k}$.

\begin{prop}\label{ind proj and simples}
With notations as above, we have:
\begin{enumerate}
\item $P_{\ell,k}$ is an indecomposable projective module and it is cyclic,
\item $\{0 \} \subseteq S_{\ell,k} \subsetneq P_{\ell, k}$ is a composition series of $P_{\ell ,k}$,
\item $Jac(P_{\ell, k}) = Soc(P_{\ell, k}) =S_{\ell,k}$,
\item $P_{\ell_1, k_1} \cong P_{\ell_2,k_2}$ if and only if $ \ell_1 = \ell_2$,
\item $\End_{\Uqle}(V^{\otimes n}) . e_{\ell} \cong P_{\ell,1}^{\oplus m_{\ell}}$ with $ m_{\ell} = \binom{n-1}{\ell}.$
\end{enumerate}
\end{prop}
\begin{proof}
Using \eqref{left multiplication} and \eqref{right multiplication} it is easily seen that $S_{\ell,k}$ and $P_{\ell,k}$ are left modules. Furthermore by construction, see \eqref{End decomp refined not}, 
\begin{equation}\label{decomp in indec proj}
\End_{\Uqle}(V^{\otimes n}) = \bigoplus_{\ell = 1}^{n-1} \bigoplus_{k=1}^{\binom{n-1}{\ell}} P_{\ell,k}
\end{equation}

Thus the complement of $P_{\ell, k}$ in $\End_{\Uqle}(V^{\otimes n})$ is also a left module and hence $P_{\ell, k}$ is projective. Next, again using \eqref{left multiplication} and \eqref{right multiplication} is readily verified that $S_{\ell, k}$ is simple and the only submodule of $P_{\ell, k}$. In particular $P_{\ell, k}$ is indecomposable. Furthermore, via analogue computations one sees that  $P_{\ell, k}/S_{\ell,k}$ is simple and that $P_{\ell,k}$ is generated as left module by any $0 \neq \varphi \in \Hom_{\Uqle}(L_{\ell,k}, L_{\ell, k^{\prime}})$ (for any $k^{\prime}$), finishing the proof of statements (1) and (2). Statement (3) follows from the second. Next, assertion (4) holds by definition and the explicit description of the hom-spaces obtained in \autoref{L:intertwiner_negative}. Finally, statement (5) follows from the fourth and that by construction $\End_{\Uqle}(V^{\otimes n}) e_{\ell} = \bigoplus_{k=1}^{\binom{n-1}{\ell}} P_{\ell,k}$.
\end{proof}

We now have the necessary tools to prove the main result of this section.

\begin{proof}[Proof of \autoref{Cartan and Ext quiver}]
As representatives for the indecomposable projective modules we take the modules $P_{\ell,1}$ which we denote by $P_{\ell}$ for ease of notation.  

Consider $0 \leq i,j \leq n-1$. We want to compute $\Hom_{\Uqle}(P_{i},P_{j})$. Since the modules $P_{\ell}$ are cyclically generated by any $0 \neq \varphi_{\ell} \in \Hom_{\Uqle}(L_{\ell,1}, L_{\ell, k^{\prime}})$, it is enough to determine the image of $\varphi_{i}$. From \eqref{left multiplication}, \eqref{right multiplication} and the composition series in \autoref{ind proj and simples} it now follows that $\Hom_{\Uqle}(P_{i},P_{j})= 0$ if $i \neq j$ or $j-1$. It also implies that if $i = j$, then $\End(P_{i})$ consists of scalar endomorphisms. Finally if $i = j-1$, then any morphism is of the form $\psi^* = (-) \circ \psi$ with $\psi \in \Hom_{\Uqle}(L_{\ell,1}, L_{\ell-1,1})$. Note that 
\begin{equation}\label{image pullback}
\psi^*(Soc(P_{j-1})) =0 \text{ and } \psi^*(P_{j-1}) = Jac(P_i).
\end{equation}
Altogether we obtain the stated form of the Cartan matrix.

For the Ext-quiver we need to compute the values $\dim \Ext^1(S_{i,1}, S_{j,1})$. Now recall that 
\[\dim \Ext^1(S_{i,1}, S_{j,1}) = \dim \Hom_{\Uqle}(P_j, Jac(P_i)/Jac(P_i)^2).\]
From \autoref{ind proj and simples} we know that $Jac(P_i)= S_{i,1}$ and consequently $Jac(P_i)^2= \{ 0 \}$. In particular we need to count morphisms from $P_j$ to $P_i$. From the computations earlier, switching the role of $i$ and $j$, we know that this forces $j = i$ or $i-1$. If $j = i$, the endomorphisms are scalar multiples of the identity and hence the image is not in $Jac(P_i)$. If $j = i-1$, then \eqref{image pullback} yields that $ \dim \Hom_{\Uqle}(P_j, Jac(P_i)) =1$, finishing the proof that the Ext-quiver is of the claimed $A_n$-type.

Over $\overline{\Q(q)}$, the algebra $\End_{\Uqle}(V^{\otimes n})$ is Morita equivalent to an algebra $kQ_A/I$ where $Q_A$ is the Ext-quiver and $I$ an ideal in the radical. It is a general fact that Morita equivalent algebras have isomorphic Jacobson radical, hence \autoref{W-M decompisition End} yields the stated statement on the relations.
\end{proof}

\begin{rem}
One could also have worked with right modules. In that case one has the analogue of \autoref{ind proj and simples} but for the following right modules:
\begin{equation}\label{Description right simple and ind proj}
S_{\ell-1, k^{\prime}}^{(r)} := \bigoplus_{r =1}^{\binom{n-1}{\ell}} \Hom_{\Uqle}(L_{\ell, r}, L_{\ell-1, k^{\prime}}) \text{ and } P_{\ell-1, k^{\prime}}^{(r)} := S_{\ell-1,k^{\prime}}^{(r)} \oplus \bigoplus_{k=1}^{\binom{n-1}{\ell-1}} \Hom_{\Uqle}(L_{\ell-1, k}, L_{\ell-1, k^{\prime}})
\end{equation}
with $0 \leq \ell-1 \leq n-2$ and $1 \leq k^{\prime} \leq \binom{n-1}{\ell-1}$. For $\ell-1 =n-1$ one defines $P_{n-1,1} := \Hom_{\Uqle}(L_{n-1, 1}, L_{n-1, 1})$.
\end{rem}

\section{Equivalence between the presentations and consequences}
\label{sec:equivalence_presentations}

The main aim of this section is to prove the following crucial result.

\begin{thm}
\label{thm:equivalence_presentations}
    The loop Hecke algebra $\LH_n$ (\autoref{Definition loop heck}) and the algebra $\varLH_n$ (\autoref{defn:var_loop_hecke}) are isomorphic under the following change of coefficients:
    \begin{enumerate}[(1)]
        \item when we localize at $t,t\pm 1$:
        \[\LH_n \otimes_{\bZ[t]} \bZ[t^{\pm 1}]\left[\frac{1}{t\pm 1}\right]
        \cong 
        \varLH_n \otimes_{\bZ} \bZ[t^{\pm 1}]\left[\frac{1}{t\pm 1}\right],\]
        \item when we specialize at $t=0$:
        \[\LH_n \otimes_{\bZ[t]} \bZ
        \cong 
        \varLH_n,\]
        where $\bZ$ is viewed as a $\bZ[t]$-algebra with $t$ acting as $0$.
    \end{enumerate}
    In both cases, the isomorphism is explicitly given by the following mutually inverse maps:
    \begin{gather*}
        \begin{cases}
            D_i &\mapsto (\sigma_i-\rho_i)/(1-t)\\
            U_i &\mapsto (\sigma_i-t\rho_i)/(1-t)
        \end{cases}
        \quad\leftrightarrows\quad
        \phi\colon
        \begin{cases}
            \sigma_i &\mapsto U_i-tD_i\\
            \rho_i &\mapsto U_i-D_i
        \end{cases}
        \;.
    \end{gather*}
\end{thm}

\begin{cor}
\label{cor:equivalence_presentation_over_field}
    Let $\Bbbk$ be a field and $z\in \Bbbk$ with $z\neq \pm 1$. We view $\Bbbk$ as a $\bZ[t]$-algebra via the map $t\mapsto z$. Then the algebras $\LH_n\otimes_{\bZ[t]}\Bbbk$ and $\varLH_n\otimes_{\bZ}\Bbbk$ are isomorphic. 
\end{cor}

In what follows, whenever we work over a field $\Bbbk$ as in the corollary above, we write $t$ instead of $z$.

Combining with \autoref{thm:reduced_words_basis}, \autoref{thm:cardinality_reduced_words} and \autoref{cor:varLH-canonical-map-is-inclusion}, we get:

\begin{cor}
\label{cor:dimension_LH}
    Under the change of coefficients (1) and (2) as in \autoref{thm:equivalence_presentations}, the loop Hecke algebra $\LH_n$ is free of rank $\frac{1}{2}\binom{2n}{n}$ and the natural map $\LH_{n}\rightarrow \LH_{n+1}$ is injective.
\end{cor}

Thereafter we consider two applications of the presentation yielded by \Cref{thm:equivalence_presentations}. Firstly, we consider in \autoref{section: special quot}  the quotient of $\LH_n$ by the two-sided ideal generated by 
\[\chi^{(j+1)} := (\sigma_1 - \rho_1) \cdots (\sigma_{j}- \rho_{j}) \]
for $1 \leq j \leq n.$ Concretely we obtain that if $t \neq \pm 1$, then $(\LH_n\otimes_{\bZ[t]}\Bbbk) / (\chi^{(j)}) \cong \Bbbk$. This yields an interesting counit, i.e.\ structure of augmented algebra. It was also proposed in \cite[Section 6]{DMR_GeneralisationsHeckeAlgebras_2023} to consider this quotient, see also \Cref{discussion on conj 6.4}.

Finally, the presentation of the loop Hecke algebra in the generators $\{ \rho_i ,D_i\}$ has quite some similarities with the Hecke-Hopf algebra ${\bf H}(\Sym_n)$ introduced by Berenstein--Kazhdan \cite{BK_HeckeHopfAlgebras_2019}. In \autoref{Hopf-Hecke section}, as a second application, we show that ${\bf H}(\Sym_n)$ canonically surjects as an augmented $\bZ$-algebra onto $\varLH_n$, but the kernel is not a Hopf ideal. In fact, we check that for small $n$ the loop Hecke algebra is {\it not} a Hopf algebra, raising the question of whether there exists a variant of the Hecke-Hopf algebra for the loop Hecke algebra.

\subsection{Proof that both presentations are equivalent}\label{section proof pres are equivalent}

In this section we prove \autoref{thm:equivalence_presentations}.
It essentially consists in finding the right computations; in that process, we heavily used \textsc{Magma} \cite{BSM+_Magma_}. However, the final proof does not require computer assistance.

It is readily verified that the maps in \autoref{thm:equivalence_presentations} are each other's inverses. Hence, it remains to prove that the maps are well-defined.  
Consider the application $\phi$ from (changes of coefficients) from $\LH_n$ to $\varLH_n$, be it in case (1) or (2) in \autoref{thm:equivalence_presentations}. We must show that the image under $\phi$ of the defining relations in \autoref{Definition loop heck} imply the defining relations in \autoref{defn:var_loop_hecke}, and vice-versa.
It is easy to check that the image of the distant-label quadratic relations \eqref{eq:sigma_rho_distant_label} imply the distant-label quadratic relations \eqref{eq:D_and_U_distant_label}, and vice-versa.

As in \autoref{subsubsec:abuse_notation_U+}, we abuse notation and write 
\[U\coloneqq U_i \text{ and } U_+\coloneqq U_{i+1},\] 
and similarly for $D$.

\begin{rem}
\label{rem:some-braid-rels-follow-from-quadratic-rels}
    The relations $D_+DD_+ = D_+D$ and $UU_+U = U_+U$ in the presentation of $\LH_n$ are consequences of the same-label relations and the relations $UD_+=0$ and $U_+D=DU_+$.
    One can see that by simplifying $D_+U_+DD_+$ (resp.\ $UU_+DU$) in two different ways.
\end{rem}

\subsubsection{A first look at the relations}

It is readily seen that the quadratic relations in the $D$'s and $U$'s \eqref{eq:D_and_U_same_label} imply the quadratic relations in the $\sigma$'s and $\rho$'s \eqref{eq:sigma_rho_same_label_same}-\eqref{eq:sigma_rho_same_label_mixed}.
Conversely, the image under $\phi$ of the relations \eqref{eq:sigma_rho_same_label_same}-\eqref{eq:sigma_rho_same_label_mixed} is:
\begin{align*}
    (U-tD-1)&(U-tD+t) = 0, & (U-D)^2 -1  &= 0,\\
    (U-D-1)&(U-tD+t)=0, & (U-tD-1)(U-D+1)&=0.
\end{align*}
When $t - 1$ is invertible (which applies to both case (1) and (2)), these relations are equivalent to the same-label relations in the $U$'s and $D$'s \eqref{eq:D_and_U_same_label}.
Indeed, if we write $q_1$, $q_2$, $q_3$ and $q_4$ the left-hand side of these relations, we have that:
\begin{gather*}
    \frac{1}{(t-1)^2}(q_2 - q_4 - q_3 + q_1)
    =
    D^2 - D,
    \\
    \frac{1}{t-1}(q_3 - q_1)
    =
    DU - tD^2 + tD
    \quad\text{ and }\quad
    \frac{1}{t-1}(q_2 - q_3)
    =
    UD - D^2 - U + 1.
\end{gather*}
This gives the relations $D^2=D$, $DU=0$ and $UD=D+U-1$. Looking at $q_2$ gives the last relation $U^2=U$.




\medbreak

On the other hand, the braid relations in the $\rho$'s and $\sigma$'s \eqref{eq:sigma_rho_braid_same} and \eqref{eq:sigma_rho_braid_mixed} give:
\[
(U-aD)(U_{+}-bD_{+})(U-bD)
=(U_{+}-bD_{+})(U-bD)(U_{+}-aD_{+}),
\]
where $(a,b)=(1,1)$, $(1,t)$, $(t,1)$ or $(t,t)$.
Expending gives the family of relations:
\begin{IEEEeqnarray*}{rCl}
    r_{(a,b)} &\coloneqq& 
    UU_{+}U
    -U_{+}UU_{+}
    \\
    &&\quad{}-a(
    DU_{+}U
    -U_{+}UD_{+}
    )
    \\
    &&\quad{}-b(
    UD_{+}U
    +UU_{+}D
    -U_{+}DU_{+}
    -D_{+}UU_{+}
    )
    \\
    &&\quad{}+ab(
    DD_{+}U
    +DU_{+}D
    -U_{+}DD_{+}
    -D_{+}UD_{+}
    )
    \\
    &&\quad{}+b^2(UD_{+}D-D_+DU_+)
    \\
    &&\quad{}-ab^2(
    DD_{+}D
    -D_{+}DD_{+}
    )
\end{IEEEeqnarray*}
Here we say ``the relation $r_{(a,b)}$'' to refer to the relation $r_{(a,b)}=0$. 
One checks that the defining relations of $\varLH_n$ imply the above (family of) relation(s). 
It remains to show that this (family of) relation(s), together with the same-label relations, implies the remaining defining relations of $\varLH_n$.

\subsubsection{Case (1)}\leavevmode
\label{subsubsec:proof_equivalence_t_not_zero}
\medbreak
\noindent\textbf{Step 1.}
We compute the following relation, simplifying with the same-label relations:
%
\begin{IEEEeqnarray*}{rC.r.lr.ll.ll.ll.ll.r}
    s_1 &\coloneqq& \frac{1}{(t-1)^2}\Big[\;
        &&
        &(t^3&r_{(1,1)} &{}-{} t^2&r_{(1,t)} &{}-{} t   &r_{(t,1)} &{}+{}   &r_{(t,t)})
        \;U
    \\ &&
        &{}-{}& 
        &(t^3&r_{(1,1)} &{}-{} t^2&r_{(1,t)} &{}-{}{} t^2 &r_{(t,1)} &{}+{} t &r_{(t,t)})
        \;D
    \\ &&
        &{}+{}& U_+
        &(t^3&r_{(1,1)} &{}-{} t&r_{(1,t)}   &{}-{} t^2 &r_{(t,1)} &{}+{}   &r_{(t,t)})
    \\ &&
        &{}-{}& D_+
        &(t^3&r_{(1,1)} &{}-{} t&r_{(1,t)}   &{}-{} t^3 &r_{(t,1)} &{}+{} t &r_{(t,t)})
    \\ &&
        &{}-{}& (t-1) 
        &(&             &{}\;\;\;\;t&r_{(1,t)}   &{}-{}&&                   &r_{(t,t)})  &\;\Big]
    \\
    &=& \IEEEeqnarraymulticol{12}{l}{
        tUD_+U - tDU_+U - t^2D_+UD_+ + t^2D_+DU_+.
    }
\end{IEEEeqnarray*}
This gives a new relation between words of length three.

Recall that by assumption of the case at hand, $t$ is invertible.
We compute the following, again using the same-label relations to simplify:
\begin{IEEEeqnarray*}{rCl}
    s_2 &\coloneqq& 
    \frac{-1}{t^2(t-1)} ( D_+ s_1 + s_1 U - s_1 )
    = D_{+} U D_{+} - D_{+} D U_{+}.
\end{IEEEeqnarray*}
This gives a relation between two words of length three.

\medbreak

\noindent\textbf{Step 2.} We derive further relations from $s_1$ and $s_2$.
Multiplying $s_2$ on the right with $U_+$ gives
\[D_{+} U D_{+} = D_{+} D U_{+}=0.\]
Combining with $s_1$ gives $UD_+U = DU_+U$, using the assumption $t$ invertible.
Multiplying this relation on the left by $D$ gives
\[UD_+U = DU_+U=0.\] 

In turn, we can derive further relations from the relation above by multiplying on the left and on right in such a way that the pattern $UD$ (or $U_+D_+$) appears, and simplifying with the relevant same-label relation.
This gives:
\begin{gather*}
    U_+UD_+ = UD_+,\quad D_+DD_+ = D_+D,\quad U_{+} D U_{+} = D U_{+}\\
    \text{and}\quad U_+DD_+=U_+D+DD_+-D.
    \\[1ex]
    UD_+D = UD_+,\quad UU_+U = U_+U, \quad DU_+D = DU_+\\
    \text{ and }\quad UU_+D = UU_+ + U_+D - U_+.
\end{gather*}

\medbreak

\noindent\textbf{Step 3.}
Using $UD_+U = 0$, $UD_+D = UD_+$, $UU_+D = UU_+ + U_+D - U_+$ and $U_+DU_+ = DU_+$ from Step 2, we compute:
\begin{IEEEeqnarray*}{rCl}
    s_3 &\coloneqq&
    \frac{1}{(t+1)(t-1)^2}U \big[
        t r_{(1,1)} - t r_{(1,t)} - r_{(t,1)} + r_{(t,t)}
    \big]
    =
    U_+D - DU_+.
\end{IEEEeqnarray*}
This gives the relation
\[ U_+D = DU_+.\]

\medbreak

\noindent\textbf{Step 4.}
Using the relations $D_+UD_+ = 0$ and $D_+UU_+ = 0$ from Step 2 and the relation $U_+D = DU_+$ from Step 3, we compute:
\begin{IEEEeqnarray*}{rCl}
    s_4
    &\coloneqq& 
    \frac{1}{(t+1)(t-1)^2}\Big[
    t^2r_{(1,1)} - r_{(1,t)} - t^2r_{(t,1)} + r_{(t,t)}
    \Big]D_+
    =
    - UD_{+}.
\end{IEEEeqnarray*}
This gives the relation
\[UD_+ = 0.\]

\medbreak

\noindent\textbf{Step 5.} Using the relations $UU_+U = U_+U$ from Step 2, the relation $U_+D = DU_+$ from Step 3 and the relation $UD_+ = 0$ from Step 4, we compute:
\begin{IEEEeqnarray*}{rCl}
    s_5
    &\coloneqq& 
    \frac{1}{(t-1)^2}\Big[
    t^2r_{(1,1)} - tr_{(1,t)} - tr_{(t,1)} + r_{(t,t)}
    \Big]D_+
    =
    U_{+}U - U_{+}UU_{+}
\end{IEEEeqnarray*}
This gives the relation
\[U_{+}UU_{+} = U_{+}U.\]
Furthermore, using the relation $D_+DD_+ = D_+D$ from Step 2, the relation $U_+D = DU_+$ from Step 3 and the relation $UD_+ = 0$ from Step 4, we compute:
\begin{IEEEeqnarray*}{rCl}
    s_6
    &\coloneqq& 
    \frac{1}{(t-1)^2}\Big[
    tr_{(1,1)} - r_{(1,t)} - tr_{(t,1)} + r_{(t,t)}
    \Big]D_+
    =
    -t(DD_{+}D - D_{+}D).
\end{IEEEeqnarray*}
Under the assumption $t$ invertible, this gives
\[DD_{+}D = D_{+}D.\]
\medbreak

\noindent\textbf{Step 6.}
Using the relations $DD_{+}D = D_+DD_+$ and $UU_+U = U_+UU_+$ following from Step 2 and Step 5, the relation $U_+D = DU_+$ from Step 3, and the relation $UD_+ = 0$ from Step 4, we see that that the first, second, fifth and sixth summands in $r_{(a,b)}$ are zero.
Moreover, using again $U_+D = DU_+$ and $UD_+ = 0$:
\begin{IEEEeqnarray*}{rCl}
    r_{(t,1)}
    &=&
    tDD_{+}U + D_{+}UU_{+} - tDD_{+} - UU_{+} + tD + U_{+}.
\end{IEEEeqnarray*}
We compute, using the assumptions that $t+1$ and $t$ are invertible and the quadratic relations found in previous steps:
\begin{IEEEeqnarray*}{rCl}
    s_7
    &\coloneqq& 
    -\frac{1}{t(t+1)} 
    \Big[r_{(t,1)}U_+ - tDr_{(t,1)} - U_+r_{(t,1)} + (t+1)Ur_{(t,1)}\Big]
    =
    D_+U - D_+ - U + 1.
\end{IEEEeqnarray*}
This gives the last relation $D_+U = D_+ + U -1$, which concludes the proof in the case $t\neq 0$.

\subsubsection{Case (2)}
\label{subsubsec:proof_equivalence_t_zero}

When $t=0$, the relations $r_{(a,b)}$ reduce as follows:
\begin{IEEEeqnarray*}{LrCl}
    v_1\coloneqq 
    &
    r_{(t,t)}
    &=&
    UU_{+}U - U_{+}UU_{+}
    \\
    v_2\coloneqq 
    &
    r_{(t,t)} - r_{(1,t)}
    &=& 
    DU_{+}U - U_{+}UD_{+}
    \\
    v_3\coloneqq 
    &
    r_{(t,t)} - r_{(t,1)}
    &=& 
    UD_{+}U + UU_{+}D - U_{+}DU_{+} - D_{+}UU_{+} 
    \\
    &&&\quad{}+ UD_{+}D - D_+DU_+
    \\
    v_4\coloneqq 
    &
    r_{(1,1)} + r_{(1,t)} + r_{(t,1)} - r_{(t,t)}
    &=&
    DD_{+}U + DU_{+}D - U_{+}DD_{+} - D_{+}UD_{+}
    \\
    &&&\quad{}+ UD_{+}D - D_+DU_+ - DD_{+}D + D_{+}DD_{+}
\end{IEEEeqnarray*}

\medbreak

\noindent\textbf{Step 1.} We compute:
\begin{IEEEeqnarray*}{rCl}
    s_1 &\coloneqq & U_+v_3 + v_1D + v_2U + U_+v_2D + v_3U + v_3D + D_+v_1D + D_+v_2D - v_1 -2v_3
    \\
    &=&
    DU_{+}U - U_+DU_{+} + DU_+ - UD_{+}.
\end{IEEEeqnarray*}
This gives a relation between words of length at most three.

\medbreak

\noindent\textbf{Step 2.}
We compute:
\begin{IEEEeqnarray*}{rCl}
    s_2 &\coloneqq&
    Ds_1U_+ - s_1U_+ - Ds_1 + s_1
    =
    - UD_{+}  .
\end{IEEEeqnarray*}
This gives the relation
\[UD_+ = 0.\]

\medbreak

\noindent\textbf{Step 3.} With the relation from Step 2, we have that $v_2 = DU_+U$.
Multiplying by $U$ on the left gives $UU_+U - U_+U$. Together with $v_1$, it gives the relations
\[ UU_+U = U_+UU_+ = U_+U.\]

\medbreak

\noindent\textbf{Step 4.}
It follows from Step 3 that $UU_+D = U_+D + UU_+ - U_+$, computing $(UU_+U - U_+U)D$. With this and Step 2, $v_3$ simplifies as:
\begin{IEEEeqnarray*}{rCl}
    v_3 &=& - U_{+}DU_{+} - D_{+}UU_{+} - D_+DU_+ + U_+D + UU_+ - U_+
\end{IEEEeqnarray*}
Computing $U_+v_3 - v_3 - Uv_3$ leads to the relation
\[ DU_+ = U_+D.\]
Note that thanks to \autoref{rem:some-braid-rels-follow-from-quadratic-rels}, it also follows that $D_+DD_+ = D_+D$.

\medbreak

\noindent\textbf{Step 5.}
With the relation $DU_+ = U_+D$ from Step 4, $v_3$ further simplifies as:
\begin{IEEEeqnarray*}{rCl}
    v_3 &=& - D_{+}UU_{+} + UU_+ - U_+
\end{IEEEeqnarray*}
Using $UD_+ = 0$, computing $v_3D_+ - v_3$ leads to the relation
\[D_+U = D_+ + U - 1.\]

\medbreak

\noindent\textbf{Step 6.}
Using the relations found in previous steps, the relation $v_4$ simplifies as follows:
\begin{IEEEeqnarray*}{rCl}
    v_4 
    &=&
    - DD_{+}D + D_{+}D .
\end{IEEEeqnarray*}
This gives the remaining relation $DD_{+}D = D_{+}D$, and concludes.

\subsection{The structure as augmented algebra}\label{section: special quot}

In \cite[Section 6]{DMR_GeneralisationsHeckeAlgebras_2023} it is proposed to study the quotient of $\LH_n\otimes_{\bZ[t]}\Bbbk$, for some appropriate field $\Bbbk$, by the two-sided ideal generated by 
\begin{equation}\label{the conj word}
\chi^{(j+1)} := (\sigma_1 - \rho_1) \cdots  (\sigma_{j}- \rho_{j})
\end{equation}
for $1 \leq j \leq n.$ For a field $\Bbbk$ as in \autoref{cor:equivalence_presentation_over_field}, this is equivalent to describing the quotient of $\varLH_n\otimes_{\bZ}\Bbbk$ by the element $D_1  \cdots  D_{j}.$ 

As a first application of this integral form presentation we describe the proposed quotient:

\begin{prop}\label{description quotient}
The map
\[\pi_n : \varLH_n \rightarrow \bZ : \left\lbrace \begin{array}{l}
U_i \mapsto 1\\
D_i \mapsto 0
\end{array} \right.\]
is a $\bZ$-algebra morphism with $\ker(\pi_n) = (D_1  \cdots  D_{j})$ for any $1 \leq j \leq n$.
\end{prop}
\begin{proof}
Using a direct verification it is easily verified that $\pi_n$ is well-defined, i.e.\ that the defining relations in \autoref{defn:var_loop_hecke} of $\varLH_n$ are satisfied under $\pi_n$. 

Note that $\ker(\pi_n) = (D_1, \ldots, D_n)$ as the quotient map $\varLH_n \twoheadrightarrow \varLH_n / (D_1, \ldots, D_n) $ maps $U_i$ to $1$ by the relation $U_iD_i = U_i+D_i -1$, see \eqref{eq:D_and_U_same_label}. Therefore, this quotient map agrees with $\pi_n$ and it remains to prove that $(D_1, \ldots, D_n) = (D_1. \cdots.D_j)$ for any $1 \leq j \leq n$. Equivalently, we prove that $D_i \equiv 0$ in  $\varLH_n/ (D_1 \cdots D_j)$ for any $1 \leq i \leq n-1$.

We will use $\equiv$ to emphasize that we are working with the quotient $\varLH_n/ (D_1 \cdots D_j)$. As $ 0 \equiv D_1 \cdots D_j $, also $U_1 D_1 \cdots D_m \equiv 0$ for $j \leq m$. Consequently, if $m> 1$ and using the relations for $U_1D_1$ and $U_1D_2$:
\[
\underbrace{U_1 D_2}_{=0}D_3 \cdots D_{m} +  \underbrace{D_1 \cdots D_{m}}_{= 0} - D_2 \cdots D_{m} \equiv 0.
\]
Thus we obtained that $D_2 \cdots  D_m \equiv 0$. Continuing iteratively, we obtain that $D_m \equiv 0$ and hence $D_j \equiv \cdots \equiv D_{n} \equiv 0$.

For the variables $D_i$ with $i < j$ we consider $D_i \cdots D_j 	U_{j-1}\equiv 0$. For that word, the defining relation \eqref{eq:D, U length 2} yields
\[\underbrace{D_i\cdots D_j}_{=0} + D_i \cdots D_{j-2}\underbrace{D_{j-1}U_{j-1}}_{=0} - D_i\cdots D_{j-1} \equiv 0.\]
This time, continuing iteratively with $D_i\cdots D_{j-1} \equiv 0$, we obtain that $D_i \equiv 0$, as desired. 
\end{proof}

\begin{rem}\label{discussion on conj 6.4}
In \cite[Conjecture 6.4]{DMR_GeneralisationsHeckeAlgebras_2023} it was conjectured that the quotient would be non-trivial for $j \neq 1$ and that furthermore the values $\dim e_j \eL \eH_n/(\chi^{(j+1)}) e_i$, with the $e_j$ a system of orthogonal idempotents of $\eH_n/(\chi^{(j+1)})$ adding up to the identity, would be given by the $j \times j$-truncation of the matrix in \autoref{value peirce matrix}. Now, \autoref{description quotient} shows wrong, but this is due to a typo. Indeed, if one replaces in \cite[Conjecture 6.4]{DMR_GeneralisationsHeckeAlgebras_2023} the word from \eqref{the conj word} by its opposite word
\[
\chi^{(j+1)}_{-} := (\sigma_{j} - \rho_j)\cdots(\sigma_1 - \rho_1)
\]
then there is empirical evidence that their conjecture holds. The word $\chi^{(j+1)}_{-}$ was also considered explicitly in \cite[Section 6]{DMR_GeneralisationsHeckeAlgebras_2023}. 
\end{rem}

\subsection{Comparison with the Hecke-Hopf algebra}\label{Hopf-Hecke section}

In \cite{BK_HeckeHopfAlgebras_2019} the Hecke-Hopf algebra ${\bf H}(\Sym_n)$ was introduced. 
This Hopf $\bZ$-algebra \cite[Theorem 1.3]{BK_HeckeHopfAlgebras_2019} has the interesting property that the classical Hecke algebra embeds in it \cite[Theorem 1.9]{BK_HeckeHopfAlgebras_2019}: $\Hecke_n \hookrightarrow {\bf H}(\Sym_n) \otimes_{\bZ} \bZ[q, q^{-1}]$ as a left coideal subalgebra.
The algebra is defined as follows.

\begin{defn}
Let $n \geq 2$. The Hecke-Hopf algebra, denoted $\mathbf{H}(\Sym_n)$, is the $\mathbb{Z}$-algebra generated by $s_i$ and $D_i$ for $i = 1, \dots, n-1$ and subject to the following relations:
\begin{itemize}
    \item $s_i^2 = 1,\quad s_i D_i + D_i s_i = s_i - 1,\quad D_i^2 = D_i$ for $1 \leq i \leq n-1$,
    
    \item $s_j s_i = s_i s_j$,\; $D_j s_i = s_i D_j$,\; $D_j D_i = D_i D_j$ for $|i - j| > 1$,
    
    \item $s_j s_i s_j = s_i s_j s_i$,\; $D_i s_j s_i = s_j s_i D_j$,\; $D_j s_i D_j = s_i D_j D_i + D_i D_j s_i + s_i D_j s_i$ for $|i - j| = 1$.
\end{itemize}
\end{defn}

Note that, when $t - 1$ is invertible, the presentation of the loop Hecke algebra $\LH_n$ in the generators $\{ \rho_i ,D_i := \frac{(\sigma_i - \rho_i)}{1-t}\}$ has quite some similarities with the Hopf-Hecke algebra. Using the presentation obtained via \autoref{thm:equivalence_presentations} we will rather compare with $\varLH_n$ which is also a $\bZ$-algebra.

\begin{prop}\label{Hecke-Hopf on loop Hecke}
Let $n \geq 2$. Then the map
\[\psi_n\colon {\bf H}(\Sym_n) 
\rightarrow \varLH_n \colon
\begin{cases}
s_i &\mapsto U_i - D_i \\
D_i &\mapsto D_i
\end{cases}\]
is an epimorphism of augmented $\bZ$-algebras. However, $\ker(\psi_n)$ is \emph{not} a Hopf ideal of ${\bf H}(\Sym_n)$.
\end{prop}

\begin{rem}\label{loop hecke not hopf}
Unfortunately, as $\ker(\psi_n)$ is not a Hopf ideal, the Hopf structure of the Hecke-Hopf algebra cannot be transported to the loop Hecke algebra. In fact for $n= 2,3,4$ the loop Hecke algebra is {\it not a Hopf algebra}. 

Indeed, in those cases $\dim_{\Bbbk} (\varLH_n\otimes_{\bZ[t]} \Bbbk) =\binom{2n-1}{n}= 3$, $10$ and $35$ when $n= 2$, $3$ and $4$ respectively. For these dimensions it is known that all Hopf algebras over an algebraically closed field $\Bbbk$ of characteristic $0$ are semisimple, see \cite{Ng_HopfAlgebrasDimension_2004, Ng_HopfAlgebrasDimension_2005}. However, it follows from the combination of \autoref{the rep is iso} and \autoref{W-M decompisition End} that the loop Hecke algebra is not semisimple. 
\end{rem}

Thus it seems that the loop Hecke algebra shares with the Hecke algebra the fact of not being a Hopf algebra. Hence it is reasonable to ask the following. 

\begin{quest}\label{qust: loop hecke-hopf}
Does there exist a ``loop Hecke-Hopf algebra'', i.e.\ a Hopf algebra in which the loop Hecke algebra canonically embeds as a coideal subalgebra? 
\end{quest}

We now proceed to the proof.

\begin{proof}[Proof of \autoref{Hecke-Hopf on loop Hecke}]
By construction $\psi_n$ will be an epimorphism of $\bZ$-algebras if it is well-defined. That it is one of augmented algebras means that $\epsilon_{{\bf H}(\Sym_n)} = \pi_n \circ \psi_n$ where $\epsilon_{{\bf H}(\Sym_n)}$ is the counit ${\bf H}(\Sym_n)$ and $\pi_n$ is defined in \autoref{description quotient}. Recall that by definition $\epsilon_{{\bf H}(\Sym_n)}(s_i)=1$ and $\epsilon_{{\bf H}(\Sym_n)}(D_i) =0$ and so we indeed see that $\psi_n$ commutes with the augmentation.

\medbreak

Next, we verify the well-definedness.

The relations $\psi_n(s_i^2) =1$ and $\psi_n(s_iD_i + D_is_i) = \psi_n(s_i -1)$ follow via a direct computation using the same-label relations \eqref{eq:D_and_U_same_label}.

The relations for $|i-j|>1$ follow from the interchange relations \eqref{eq:D_and_U_distant_label}.

The relation $\psi_n(s_j s_i s_j) = \psi_n(s_i s_j s_i)$ follows from direct computation, analogous to how the relation $\rho_i\rho_{i+1}\rho_i = \rho_{i+1}\rho_i\rho_{i+1}$ in $\LH_n$ (see \eqref{eq:sigma_rho_braid_same}) followed from relations in $\varLH_n$ in the proof of \autoref{thm:equivalence_presentations}. 
Using quadratic relations \eqref{eq:D_and_U_same_label} and \eqref{eq:D, U length 2} and considering the cases $j=i\pm 1$ separately, the relation $\psi_n(D_is_js_i) = \psi_n(s_js_iD_j)$ reduces to:
\begin{align*}
    - D_iU_{i+1} - D_iD_{i+1} + D_i + D_iD_{i+1}D_i
    &=
    - D_iU_{i+1}- D_iD_{i+1} + D_i + D_{i+1}D_iD_{i+1}
    \\
    \text{and}\quad
    U_{i}U_{i+1} - U_{i+1} + D_{i+1}D_{i}D_{i+1}
    &=
    U_{i}U_{i+1} + D_{i}U_{i+1} - U_{i+1} - D_{i}U_{i+1} + D_{i}D_{i+1}D_{i}.
\end{align*}
Both hold thanks to $D_{i}D_{i+1}D_{i} = D_{i+1}D_{i}D_{i+1}$.
Finally, we check the relation $\psi_n (D_j s_i D_j) =\psi_n( s_i D_j D_i + D_i D_j s_i + s_i D_j s_i)$. On the one hand $\psi_n (D_j s_i D_j) = D_j U_iD_j - D_j D_i D_j$, and on the other hand a direct computation gives 
\[\psi_n( s_i D_j D_i + D_i D_j s_i + s_i D_j s_i) =U_iD_jU_i - D_iD_jD_i.\]
Considering the cases $j = i \pm 1$ separately, it is a direct computation that $U_iD_jU_i = D_j U_iD_j$. 

\medbreak

Finally, we prove that $\ker(\psi_n)$ is not a Hopf-ideal.
Note that $D_i s_i + D_i \in \ker(\psi_n)$.
To show that $\ker(\psi_n)$ is not a Hopf ideal it suffices to show that\footnote{Alternatively one can verify that it is not a coideal, hence $\ker(\psi_n)$ is also not a bialgebra ideal.} $S((s_i+D_i)D_{i+1}) \notin \ker(\psi_n)$. To do so, recall that the antipode is defined by $S(s_i) = s_i$ and $S(D_i)= - s_i D_i$. Hence
\[S((s_i+D_i)D_{i+1}) = s_i(1-D_i)D_{i+1}.\]
Now, using \eqref{eq:D_and_U_same_label} we find:
\[\Psi_n(s_i(1-D_i)D_{i+1}) = (U_i-D_i)(1-D_i)D_{i+1} =U_iD_{i+1} - U_iD_iD_{i+1} = -D_iD_{i+1} + D_{i+1}.\]
The latter is non-zero as the monomials are $\varLH_n$-reduced words and hence linearly independent by \autoref{thm:reduced_words_basis}. 
\end{proof}

\appendix
\section{Confluence of critical branchings}
\label{app:confluence_critical_branchings}
\input{appendix}


\vspace*{1cm}



   

\printbibliography

\end{document}

%% file: macros-etal/packages-etc.tex
\usepackage{
   a4wide, 
  amsmath,
  amsfonts,
  amssymb,
  graphicx, 
  times,
  color, 
  hyphenat, 
  stmaryrd, datetime, bbm, tikz-cd, thmtools, verbatim
}
\usepackage{ytableau}
\usepackage{mathscinet}
\usepackage[shortlabels]{enumitem}

\usepackage{dynkin-diagrams}
\usepackage{mathdots}
\usepackage[all]{xy}
\usepackage{euscript,mathrsfs}
\usepackage{bbm,xspace}
\usepackage{hyperref}
\usepackage{cleveref}
\hypersetup{colorlinks=true, pdfstartview=FitV, linkcolor=teal, citecolor=teal, urlcolor=teal}

\usepackage[modulo]{lineno}
\usepackage{dirtytalk}

\usepackage{IEEEtrantools}

\usepackage{tikz}
\usetikzlibrary{cd}
\usetikzlibrary{decorations}
\usetikzlibrary{decorations.markings}
\usetikzlibrary{decorations.pathreplacing}
\usetikzlibrary{decorations.pathmorphing}
\usetikzlibrary{arrows.meta,shapes,positioning,matrix,calc}
\usetikzlibrary{shapes.callouts}
\tikzstyle{densely dotted}=[dash pattern=on \pgflinewidth off 0.5pt]
\tikzset{anchorbase/.style={baseline={([yshift=-0.5ex]current bounding box.center)}},
tinynodes/.style={font=\tiny,text height=0.25ex,text depth=0.05ex},
smallnodes/.style={font=\scriptsize,text height=0.75ex,text depth=0.15ex},
usual/.style={line width=0.9,color=black},
dusual/.style={line width=0.9,color=spinach,densely dashed},
pole/.style={line width=3.0,color=specialgray},
crossline/.style={preaction={draw=white,line width=5.0pt,-},preaction={draw=black,line width=0.9pt,-}},
crosspole/.style={preaction={draw=white,line width=6.0pt,-},preaction={draw=specialgray,line width=3.0pt,-}},
mor/.style={line width=0.75,color=black,fill=cream},
blob/.style={circle,fill,minimum size=5.0pt,inner sep=0pt,outer sep=0pt},
blobed/.style n args={3}{decoration={markings,post length=0.5mm,pre length=0.5mm,
mark=at position #1 with {\node[blob,#3,label=left:$#2\!$]at (0,0){};}
},postaction={decorate}},
rblobed/.style n args={3}{decoration={markings,post length=0.5mm,pre length=0.5mm,
mark=at position #1 with {\node[blob,#3,label=right:$\!#2$]at (0,0){};}
},postaction={decorate}},
cutline/.style={line width=2.25,color=purple,densely dotted},
}
\tikzstyle directed=[postaction={decorate,decoration={markings,mark=at position #1 with {\arrow[line width=0.25mm, black]{>}}}}]
\tikzstyle rdirected=[postaction={decorate,decoration={markings,mark=at position #1 with {\arrow[line width=0.25mm, black]{<}}}}]

\tikzstyle{tikzdot}=[fill, circle, inner sep=2pt]
\tikzstyle{smoltikzdot}=[fill=white, draw=black, circle, inner sep=1pt]

\tikzset{
    partial ellipse/.style args={#1:#2:#3}{
        insert path={+ (#1:#3) arc (#1:#2:#3)}
    }
}

\input xy
\xyoption{all}

%% file: macros-etal/defs.tex
\newcommand{\Red}{\mathrm{Red}}
\newcommand{\Dyck}{\mathrm{Dyck}}
\newcommand{\dbDyck}{\widetilde{\Dyck}}

\newcommand{\Sym}{\mathrm{S}} 
\newcommand{\Br}{\mathrm{B}} 
\newcommand{\LBr}{\mathrm{LB}} 
\newcommand{\LBurau}{\mathbf{F}} 
\newcommand{\Hecke}{\eH} 
\newcommand{\LH}{\eL\eH} 
\newcommand{\varLH}{\widetilde{\eL\eH}} 
\newcommand{\glone}{\mathfrak{gl}_{1\vert 1}}
\newcommand{\uglone}{U_q(\glone)}
\newcommand{\neguglone}{U_q^{\leq 0}(\glone)}
\newcommand{\Uq}{U_q}
\newcommand{\Uqle}{U_q^{\leq 0}}

\newcommand{\sX}{\mathsf{X}}
\newcommand{\sR}{\mathsf{R}}
\DeclareMathOperator{\supp}{supp}

\definecolor{myblue}{rgb}{0,.5,1}
\definecolor{myred}{rgb}{0.9,0,0}
\definecolor{mygreen}{rgb}{0,0.7,0}

\newcommand{\brak}[1]{\langle #1\rangle}

\DeclareMathOperator{\Hom}{Hom}

\DeclareMathOperator{\im}{im}
\DeclareMathOperator{\id}{id}

\DeclareMathOperator{\End}{End}
\DeclareMathOperator{\Ext}{Ext}

\newtheorem{thm}{Theorem}[section]
\newtheorem{lem}[thm]{Lemma}
\newtheorem{cor}[thm]{Corollary}
\newtheorem{prop}[thm]{Proposition}
\newtheorem{conj}[thm]{Conjecture}
\theoremstyle{definition}
\newtheorem{defn}[thm]{Definition}

\newtheorem{rem}[thm]{Remark}
\newtheorem{quest}{Question}
\newtheorem{exe}[thm]{Example}

 \newtheorem{maintheorem}{Theorem}

\newtheorem{maincorollary}[maintheorem]{Corollary}
\newtheorem{mainquestion}{Question}

\def\Q{{\sf{Q}}}

\newcommand{\bN}{\mathbb{N}}
\newcommand{\bZ}{\mathbb{Z}}
\newcommand{\bQ}{\mathbb{Q}}
\newcommand{\bR}{\mathbb{R}}
\newcommand{\bC}{\mathbb{C}}

\newcommand{\eH}{\EuScript{H}}

\newcommand{\eL}{\EuScript{L}}

\newcommand{\cC}{\mathcal{C}}

\newcommand{\cS}{\mathcal{S}}

\catcode`\@=11
\long\def\@makecaption#1#2{%
    \vskip 10pt
    \setbox\@tempboxa\hbox{%
\small{#1: }\ignorespaces #2}%
    \ifdim \wd\@tempboxa >\captionwidth {%
        \rightskip=\@captionmargin\leftskip=\@captionmargin
        \unhbox\@tempboxa\par}%
      \else
        \hbox to\hsize{\hfil\box\@tempboxa\hfil}%
    \fi}
\newdimen\@captionmargin\@captionmargin=2\parindent
\newdimen\captionwidth\captionwidth=\hsize
\catcode`\@=12
\def\makeautorefname#1#2{\expandafter\def\csname#1autorefname\endcsname{#2}}
\makeautorefname{lem}{Lemma}%
\makeautorefname{prop}{Proposition}%
\makeautorefname{rem}{Remark}%
\makeautorefname{section}{Section}%
\makeautorefname{subsection}{Section}%
\makeautorefname{subsubsection}{Section}%
\makeautorefname{cor}{Corollary}%

\usepackage{xparse}
\usetikzlibrary {patterns,patterns.meta} 
\newcommand{\diagvshift}{.1}
\newcommand{\diaghshift}{.3}
\newcommand{\diagxscale}{.4*.8}
\newcommand{\diagyscale}{.7*.8}

\tikzset{styleLH/.style={
  xscale=\diagxscale,
  yscale=\diagyscale,
  thick
}}

\NewDocumentCommand{\diagLH}{mO{}}{
  \begin{tikzpicture}[
    styleLH,
    baseline={([yshift=-0.5ex]current bounding box.center)},
    #2
  ]
    #1
  \end{tikzpicture}
}

\NewDocumentCommand{\diagD}{O{0}O{0}}{
    \draw[pattern={Lines[angle=45,distance={3pt/sqrt(2)}]}] ({#1-\diaghshift},{#2+\diagvshift}) rectangle ({#1+1+\diaghshift},{#2+1-\diagvshift});
    \draw ({#1},{#2}) to ({#1},{#2+\diagvshift});
    \draw ({#1+1},{#2}) to ({#1+1},{#2+\diagvshift});
    \draw ({#1},{#2+1-\diagvshift}) to ({#1},{#2+1});
    \draw ({#1+1},{#2+1-\diagvshift}) to ({#1+1},{#2+1});
}
\NewDocumentCommand{\diagU}{O{0}O{0}}{
    \draw ({#1-\diaghshift},{#2+\diagvshift}) rectangle ({#1+1+\diaghshift},{#2+1-\diagvshift});
    \draw ({#1},{#2}) to ({#1},{#2+\diagvshift});
    \draw ({#1+1},{#2}) to ({#1+1},{#2+\diagvshift});
    \draw ({#1},{#2+1-\diagvshift}) to ({#1},{#2+1});
    \draw ({#1+1},{#2+1-\diagvshift}) to ({#1+1},{#2+1});
}

\NewDocumentCommand{\diagSTR}{O{0}O{0}}{
    \draw ({#1},{#2}) to ({#1},{#2+1});
}

\NewDocumentCommand{\diagDm}{O{0}O{0}}{
    \diagD[#1][#2]\diagSTR[#1+2][#2]
}
\NewDocumentCommand{\diagDp}{O{0}O{0}}{
    \diagSTR[#1][#2]\diagD[#1+1][#2]
}
\NewDocumentCommand{\diagUm}{O{0}O{0}}{
    \diagU[#1][#2]\diagSTR[#1+2][#2]
}
\NewDocumentCommand{\diagUp}{O{0}O{0}}{
    \diagSTR[#1][#2]\diagU[#1+1][#2]
}

\NewDocumentCommand{\diagLBOX}{O{0}O{0}}{
    \draw[dashed] ({#1-\diaghshift},{#2+\diagvshift}) rectangle ({#1+1+\diaghshift},{#2+1-\diagvshift});
    \node at (#1+.5,#2+.5) {\scriptsize ?};
    \draw ({#1},{#2}) to ({#1},{#2+\diagvshift});
    \draw ({#1},{#2+1-\diagvshift}) to ({#1},{#2+1});
}
\NewDocumentCommand{\diagRBOX}{O{0}O{0}}{
    \draw[dashed] ({#1-\diaghshift},{#2+\diagvshift}) rectangle ({#1+1+\diaghshift},{#2+1-\diagvshift});
    \node at (#1+.5,#2+.5) {\scriptsize ?};
    \draw ({#1+1},{#2}) to ({#1+1},{#2+\diagvshift});
    \draw ({#1+1},{#2+1-\diagvshift}) to ({#1+1},{#2+1});
}

\NewDocumentCommand{\diagDI}{O{0}O{0}}{
    \diagD[#1][#2]\diagSTR[#1+2][#2]
}
\NewDocumentCommand{\diagID}{O{0}O{0}}{
    \diagSTR[#1][#2]\diagD[#1+1][#2]
}
\NewDocumentCommand{\diagUI}{O{0}O{0}}{
    \diagU[#1][#2]\diagSTR[#1+2][#2]
}
\NewDocumentCommand{\diagIU}{O{0}O{0}}{
    \diagSTR[#1][#2]\diagU[#1+1][#2]
}

\NewDocumentCommand{\diagDII}{O{0}O{0}}{
    \diagD[#1][#2]\diagSTR[#1+2][#2]\diagSTR[#1+3][#2]
}
\NewDocumentCommand{\diagIDI}{O{0}O{0}}{
    \diagSTR[#1][#2]\diagD[#1+1][#2]\diagSTR[#1+3][#2]
}
\NewDocumentCommand{\diagIID}{O{0}O{0}}{
    \diagSTR[#1][#2]\diagSTR[#1+1][#2]\diagD[#1+2][#2]
}
\NewDocumentCommand{\diagUII}{O{0}O{0}}{
    \diagU[#1][#2]\diagSTR[#1+2][#2]\diagSTR[#1+3][#2]
}
\NewDocumentCommand{\diagIUI}{O{0}O{0}}{
    \diagSTR[#1][#2]\diagU[#1+1][#2]\diagSTR[#1+3][#2]
}
\NewDocumentCommand{\diagIIU}{O{0}O{0}}{
    \diagSTR[#1][#2]\diagSTR[#1+1][#2]\diagU[#1+2][#2]
}

\NewDocumentCommand{\diagI}{O{0}O{0}}{
    \diagSTR[#1][#2]
}
\NewDocumentCommand{\diagII}{O{0}O{0}}{
    \diagSTR[#1][#2]\diagSTR[#1+1][#2]
}
\NewDocumentCommand{\diagIII}{O{0}O{0}}{
    \diagSTR[#1][#2]\diagSTR[#1+1][#2]\diagSTR[#1+2][#2]
}

%% file: appendix.tex
In this appendix, we finish the proof of \autoref{thm:reduced_words_basis} by showing that the higher linear rewriting system described in \autoref{fig:higher-Grobner-basis-LH} (see \autoref{fig:higher-Grobner-basis-LH-diagrammatic} for the diagrammatic notation) critically confluate.

We enumerate critical branchings by first considering critical branchings involving a same-label rewriting step (\autoref{subsec:critical_branching_same_label}), and then all the remaining critical branchings (\autoref{subsec:critical_branching_distinct-label}).

\subsection{Same-label rewriting steps and others}\addtocontents{toc}{\protect\setcounter{tocdepth}{1}}
\label{subsec:critical_branching_same_label}

\subsubsection{Same-label rewriting steps and same-label rewriting steps}
We consider branchings whose branches are of type one of the same-label rewriting steps:
\begin{gather*}
\def\intsp{30mu}
    \diagLH{\diagD[0][2]\diagD[0][1]\diagD}
    \mspace{\intsp}
    \diagLH{\diagD[0][2]\diagD[0][1]\diagU}
    \mspace{\intsp}
    \diagLH{\diagD[0][2]\diagU[0][1]\diagU}
    \mspace{\intsp}
    \diagLH{\diagU[0][2]\diagU[0][1]\diagU}
    \mspace{\intsp}
    \diagLH{\diagU[0][2]\diagD[0][1]\diagD}
    \mspace{\intsp}
    \diagLH{\diagU[0][2]\diagD[0][1]\diagU}
    \mspace{\intsp}
    \diagLH{\diagD[0][2]\diagU[0][1]\diagD}
    \mspace{\intsp}
    \diagLH{\diagU[0][2]\diagU[0][1]\diagD}
\end{gather*}
The first four branchings are straightforward. The last four come down to the fact that
\[U\cdot (U+D-1)\overset{*}{\to}_\sR UD
\text{ and }
D\cdot (U+D-1)\overset{*}{\to}_\sR 0,
\]
and vice-versa for the multiplication on the right.

\subsubsection{\texorpdfstring{$DD\to\ldots$}{DD} and others}
\label{subsubsec:critical_branching_DD}
We consider branchings with one branch of type $DD\to\ldots$ and the other branch of type one of the distinct-label or additional rewriting steps:
\begin{gather*}
\def\intsp{30mu}
    \diagLH{\diagIU[0][2]\diagDI[0][1]\diagDI}
    \mspace{\intsp}
    \diagLH{\diagUI[0][2]\diagID[0][1]\diagID}
    \mspace{\intsp}
    \diagLH{\diagID[0][2]\diagID[0][1]\diagUI}
    \mspace{\intsp}
    \diagLH{\diagDI[0][3]\diagDI[0][2]\diagID[0][1]\diagDI}
    \mspace{\intsp}
    \diagLH{\diagDI[0][3]\diagID[0][2]\diagDI[0][1]\diagDI}
    \mspace{\intsp}
    \diagLH{\diagID[0][3]\diagID[0][2]\diagDI[0][1]\diagID}
    \mspace{\intsp}
    \diagLH{\diagID[0][3]\diagDI[0][2]\diagID[0][1]\diagID}
    \mspace{\intsp}
    \diagLH{\diagID[0][3]\diagID[0][2]\diagDI[0][1]\diagIU}
    \mspace{\intsp}
    \diagLH{\diagDI[0][3]\diagDI[0][2]\diagIU[0][1]\diagUI}
    \mspace{\intsp}
    \diagLH{\diagIDI[0][3]\diagIDI[0][2]\diagD[0][1]\diagU[2][1]\diagIUI}
\end{gather*}
Their confluence is relatively straightforward.
The confluence of the third branching comes down to the fact that
\[D_+(D_+ + U - 1) \overset{*}{\to}_\sR D_+ + U - 1.\]

\subsubsection{\texorpdfstring{$DU\to\ldots$}{DU} and others}
We consider branchings with one branch of type $DU\to\ldots$ and the other branch of type one of the distinct-label or additional rewriting steps:
\begin{gather*}
\def\intsp{30mu}
    \diagLH{\diagID[0][2]\diagIU[0][1]\diagDI}
    \mspace{\intsp}
    \diagLH{\diagIU[0][2]\diagDI[0][1]\diagUI}
    \mspace{\intsp}
    \diagLH{\diagDI[0][2]\diagUI[0][1]\diagID}
    \mspace{\intsp}
    \diagLH{\diagUI[0][2]\diagID[0][1]\diagIU}
    \mspace{\intsp}
    \diagLH{\diagDI[0][3]\diagID[0][2]\diagDI[0][1]\diagUI}
    \mspace{\intsp}
    \diagLH{\diagID[0][3]\diagDI[0][2]\diagID[0][1]\diagIU}
    \mspace{\intsp}
    \diagLH{\diagDI[0][3]\diagUI[0][2]\diagIU[0][1]\diagUI}
    \mspace{\intsp}
    \diagLH{\diagID[0][3]\diagIU[0][2]\diagUI[0][1]\diagIU}
\end{gather*}
Since $DU\to\ldots$ rewrites to zero, it suffices to check that the other branch rewrites to zero.

Confluence of the first (resp.\ second) branching uses the first (resp.\ second) additional rewriting step. The same holds for the sixth and seventh branchings.

Confluence of the third and fourth branchings is immediate, as the other branch rewrites to zero.

Confluence of the fifth and eighth branchings is straightforward.

\subsubsection{\texorpdfstring{$UU\to\ldots$}{UU} and others}
We consider branchings with one branch of type $UU\to\ldots$ and the other branch of type one of the distinct-label or additional rewriting steps:
\begin{gather*}
\def\intsp{30mu}
    \diagLH{\diagIU[0][2]\diagIU[0][1]\diagDI}
    \mspace{\intsp}
    \diagLH{\diagUI[0][2]\diagUI[0][1]\diagID}
    \mspace{\intsp}
    \diagLH{\diagID[0][2]\diagUI[0][1]\diagUI}
    \mspace{\intsp}
    \diagLH{\diagUI[0][3]\diagUI[0][2]\diagIU[0][1]\diagUI}
    \mspace{\intsp}
    \diagLH{\diagUI[0][3]\diagIU[0][2]\diagUI[0][1]\diagUI}
    \mspace{\intsp}
    \diagLH{\diagIU[0][3]\diagIU[0][2]\diagUI[0][1]\diagIU}
    \mspace{\intsp}
    \diagLH{\diagIU[0][3]\diagUI[0][2]\diagIU[0][1]\diagIU}
    \mspace{\intsp}
    \diagLH{\diagID[0][3]\diagDI[0][2]\diagIU[0][1]\diagIU}
    \mspace{\intsp}
    \diagLH{\diagDI[0][3]\diagIU[0][2]\diagUI[0][1]\diagUI}
    \mspace{\intsp}
    \diagLH{\diagIDI[0][3]\diagD[0][2]\diagU[2][2]\diagIUI[0][1]\diagIUI}
\end{gather*}
Their confluence is relatively straightforward and similar to \autoref{subsubsec:critical_branching_DD}.
The confluence of the third branching comes down to the fact that
\[(D_+ + U - 1)U \overset{*}{\to}_\sR D_+ + U - 1.\]

\subsubsection{\texorpdfstring{$UD\to\ldots$}{UD} and others}
We consider branchings with one branch of type $UD\to\ldots$ and the other branch of type one of the distinct-label rewriting steps:
\begin{gather*}
\def\intsp{30mu}
    \diagLH{\diagID[0][2]\diagUI[0][1]\diagDI}
    \mspace{\intsp}
    \diagLH{\diagIU[0][2]\diagID[0][1]\diagUI}
    \mspace{\intsp}
    \diagLH{\diagUI[0][3]\diagDI[0][2]\diagID[0][1]\diagDI}
    \mspace{\intsp}
    \diagLH{\diagIU[0][3]\diagID[0][2]\diagDI[0][1]\diagID}
    \mspace{\intsp}
    \diagLH{\diagUI[0][3]\diagIU[0][2]\diagUI[0][1]\diagDI}
    \mspace{\intsp}
    \diagLH{\diagIU[0][3]\diagUI[0][2]\diagIU[0][1]\diagID}
\end{gather*}

The confluence of the first branching is given below:
\[
\begin{tikzcd}[row sep=-2em]
    &
    \diagLH{\diagID[0][1]\diagUI}
    \;+\;
    \diagLH{\diagID[0][1]\diagDI}
    \;-\;
    \diagLH{\diagID}
    \ar[dr,bend left=10pt]
    \\
    \diagLH{\diagID[0][2]\diagUI[0][1]\diagDI}
    \ar[ur,bend left=10pt] \ar[dr,bend right=10pt]
    &&
    \diagLH{\diagID[0][1]\diagDI}
    \;+\;
    \diagLH{\diagUI}
    \;-\;
    \diagLH{\diagIII}
    \\
    &
    \diagLH{\diagID[0][1]\diagDI}
    \;+\;
    \diagLH{\diagUI[0][1]\diagDI}
    \;-\;
    \diagLH{\diagDI}
    \ar[ur,bend right=10pt]
\end{tikzcd}
\]
The confluence of the second branching is similar.

The confluence of the third branching is given below:
\[
\begin{tikzcd}[row sep=-3em]
    &
    \diagLH{\diagUI[0][2]\diagID[0][1]\diagDI}
    \;+\;
    \diagLH{\diagDI[0][2]\diagID[0][1]\diagDI}
    \;-\;
    \diagLH{\diagID[0][1]\diagDI}
    \ar[dr,bend left=15pt]
    \\
    \diagLH{\diagUI[0][3]\diagDI[0][2]\diagID[0][1]\diagDI}
    \ar[ur,bend left=15pt] \ar[dr,bend right=15pt]
    && 0
    \\
    &
    \diagLH{\diagUI[0][2]\diagID[0][1]\diagDI}
    \ar[ur,bend right=15pt]
\end{tikzcd}
\]
The confluence of the fourth branching is given below:
\[
\begin{tikzcd}[row sep=-2em,column sep=.8em]
    &
    \diagLH{\diagIU[0][2]\diagDI[0][1]\diagID}
    \;+\;
    \diagLH{\diagID[0][2]\diagDI[0][1]\diagID}
    \;-\;
    \diagLH{\diagDI[0][1]\diagID}
    \rar
    &
    \diagLH{\diagDI[0][2]\diagIU[0][1]\diagID}
    \;+\;
    \diagLH{\diagID[0][1]\diagDI}
    \;-\;
    \diagLH{\diagDI[0][1]\diagID}
    \ar[dr,bend left=15pt]
    \\
    \diagLH{\diagIU[0][3]\diagID[0][2]\diagDI[0][1]\diagID}
    \ar[ur,bend left=15pt] \ar[dr,bend right=15pt]
    &&&
    \diagLH{\diagDI[0][1]\diagIU}
    \;+\;
    \diagLH{\diagID[0][1]\diagDI}
    \;-\;
    \diagLH{\diagDI}
    \\
    &
    \diagLH{\diagIU[0][2]\diagID[0][1]\diagDI}
    \rar
    &
    \diagLH{\diagIU[0][1]\diagDI}
    \;+\;
    \diagLH{\diagID[0][1]\diagDI}
    \;-\;
    \diagLH{\diagDI}
    \ar[ur,bend right=15pt]
\end{tikzcd}
\]
The confluence of the fifth and sixth branchings is obtained similarly.

\medbreak

Next, we consider branchings with one branch of type $UD\to\ldots$ and the other branch of type one of the additional rewriting steps:
\begin{gather*}
\def\intsp{30mu}
    \diagLH{\diagIU[0][3]\diagID[0][2]\diagDI[0][1]\diagIU}
    \mspace{\intsp}
    \diagLH{\diagID[0][3]\diagDI[0][2]\diagIU[0][1]\diagID}
    \mspace{\intsp}
    \diagLH{\diagUI[0][3]\diagDI[0][2]\diagIU[0][1]\diagUI}
    \mspace{\intsp}
    \diagLH{\diagDI[0][3]\diagIU[0][2]\diagUI[0][1]\diagDI}
    \mspace{\intsp}
    \diagLH{\diagIUI[0][3]\diagIDI[0][2]\diagD[0][1]\diagU[2][1]\diagIUI}
    \mspace{\intsp}
    \diagLH{\diagIDI[0][3]\diagD[0][2]\diagU[2][2]\diagIUI[0][1]\diagIDI}
\end{gather*}
The other branch always rewrite to zero.
We leave it to the reader to check that the branch $UD\to\ldots$ also rewrites to zero.

\subsection{Distinct-label rewriting steps and others}
\label{subsec:critical_branching_distinct-label}

\subsubsection{\texorpdfstring{$U_+D\to\ldots$}{U+D} and others}

This was done in \autoref{subsec:proof_basis_theorem}.

\subsubsection{\texorpdfstring{$UD_+\to\ldots$}{UD+} and others}
We find the following list of critical branchings with the remaining distinct-label rewriting steps:
\begin{gather*}
\def\intsp{30mu}
    \diagLH{\diagUI[0][2]\diagRBOX[-1][1]\diagD[1][1]\diagUI}
    \mspace{\intsp}
    \diagLH{\diagID[0][2]\diagLBOX[2][1]\diagU[0][1]\diagID}
    \mspace{\intsp}
    \diagLH{\diagUII[0][3]\diagIDI[0][2]\diagIID[0][1]\diagIDI}
    \mspace{\intsp}
    \diagLH{\diagUI[0][3]\diagRBOX[-1][2]\diagD[1][2]\diagDI[0][1]\diagID}
    \mspace{\intsp}
    \diagLH{\diagU[0][2]\diagD[2][2]\diagIDI[0][1]\diagIID}
    \mspace{\intsp}    
    \diagLH{\diagUI[0][3]\diagIU[0][2]\diagU[0][1]\diagLBOX[2][1]\diagID}
    \mspace{\intsp}
    \diagLH{\diagUII[0][2]\diagIUI[0][1]\diagU\diagD[2][0]}
    \mspace{\intsp}
    \diagLH{\diagIUI[0][3]\diagUII[0][2]\diagIUI[0][1]\diagIID}
\end{gather*}
The branch $UD_+\to\ldots$ rewrites into zero: to show confluence, we must check that the other branch rewrites to zero.
For the first branching, we have:
\begin{gather*}
    \diagLH{\diagUI[0][2]\diagRBOX[-1][1]\diagD[1][1]\diagUI}
    \quad\to\quad
    \diagLH{\diagUI[0][2]\diagRBOX[-1][1]\diagD[1][1]\diagIII}
    \;+\;
    \diagLH{\diagUI[0][2]\diagRBOX[-1][1]\diagII[1][1]\diagUI}
    \;-\;
    \diagLH{\diagUI[0][2]\diagRBOX[-1][1]\diagII[1][1]\diagIII}
    \quad\to\quad
    \diagLH{\diagUI[0][2]\diagRBOX[-1][1]\diagII[1][1]\diagUI}
    \;-\;
    \diagLH{\diagUI[0][2]\diagRBOX[-1][1]\diagII[1][1]\diagIII}
    \quad\to\quad
    0
\end{gather*}
The last step requires a case by case analysis, depending on whether $\diagLH{\diagLBOX}$ is $\diagLH{\diagI}$, $\diagLH{\diagD}$ or $\diagLH{\diagU}$.
The second branching is analogous.
A similar case by case analysis is necessary for the fourth and sixth branchings.
Note that confluence of the fourth branching in the case $\diagLH{\diagLBOX} = \diagLH{\diagU}$ requires one the additional rewriting step, and similarly for the sixth branching in the case $\diagLH{\diagLBOX} = \diagLH{\diagD}$.
The remaining branchings are straightforward.

\medbreak

Since both $UD_+\to\ldots$ and the additional rewriting steps rewrite to zero, any branching between them is automatically confluent; hence we don't bother classifying these critical branchings.


\subsubsection{\texorpdfstring{$D_+U\to\ldots$}{D+U} and others}

We consider branchings involving a branch of type $D_+U\to\ldots$ and branch of type one of the remaining distinct-label rewriting steps.
\begin{gather*}
\def\intsp{30mu}
    \diagLH{\diagIDI[0][3]\diagIID[0][2]\diagIDI[0][1]\diagUII}
    \mspace{\intsp}
    \diagLH{\diagID[0][3]\diagDI[0][2]\diagRBOX[-1][1]\diagD[1][1]\diagUI}
    \mspace{\intsp}
    \diagLH{\diagIID[0][2]\diagIDI[0][1]\diagD[2][0]\diagU}
    \mspace{\intsp}
    \diagLH{\diagID[0][3]\diagLBOX[2][2]\diagU[0][2]\diagIU[0][1]\diagUI}
    \mspace{\intsp}
    \diagLH{\diagD[2][2]\diagU[0][2]\diagIUI[0][1]\diagUII}
    \mspace{\intsp}
    \diagLH{\diagIID[0][3]\diagIUI[0][2]\diagUII[0][1]\diagIUI}
\end{gather*}
We compute the $D_+U\to\ldots$ branch of the first three branchings:
\begin{align*}
\allowdisplaybreaks
    \diagLH{\diagIDI[0][3]\diagIID[0][2]\diagIDI[0][1]\diagUII}
    \quad&\to\quad
    \diagLH{\diagIDI[0][3]\diagIID[0][2]\diagIDI[0][1]}
    \;+\;
    \diagLH{\diagIDI[0][3]\diagD[2][2]\diagU[0][2]}
    \;-\;
    \diagLH{\diagIDI[0][3]\diagIID[0][2]}
    \quad\to\quad
    \diagLH{\diagIID[0][2]\diagIDI[0][1]}
    \;+\;
    \diagLH{\diagD[2][2]\diagU[0][2]}
    \;-\;
    \diagLH{\diagD[2][2]\diagII[0][2]}
    \\
    \diagLH{\diagID[0][3]\diagDI[0][2]\diagRBOX[-1][1]\diagD[1][1]\diagUI}
    \quad&\to\quad
    \diagLH{\diagID[0][3]\diagDI[0][2]\diagRBOX[-1][1]\diagD[1][1]}
    \;+\;
    \diagLH{\diagID[0][3]\diagDI[0][2]\diagRBOX[-1][1]\diagII[1][1]\diagUI}
    \;-\;
    \diagLH{\diagID[0][3]\diagDI[0][2]\diagRBOX[-1][1]\diagII[1][1]}
    \quad\to\quad
    \diagLH{\diagID[0][3]\diagDI[0][2]\diagRBOX[-1][1]\diagII[1][1]\diagUI}
    \\
    \diagLH{\diagIID[0][2]\diagIDI[0][1]\diagD[2][0]\diagU}
    \quad&\to\quad
    \diagLH{\diagIID[0][2]\diagIDI[0][1]\diagD[2][0]\diagII}
    \;+\;
    \diagLH{\diagIID[0][2]\diagII[0][1]\diagII[2][1]\diagD[2][0]\diagU}
    \;-\;
    \diagLH{\diagIID[0][2]\diagII[0][1]\diagII[2][1]\diagD[2][0]\diagII}
    \quad\to\quad
    \diagLH{\diagIID[0][2]\diagIDI[0][1]}
    \;+\;
    \diagLH{\diagD[2][0]\diagU}
    \;-\;
    \diagLH{\diagD[2][0]\diagII}
\end{align*}
One checks that this confluates with the other branch.
The last three branchings are analogous.

\medbreak

We find the following list of critical branchings with the additional rewriting steps:
\begin{gather*}
\def\intsp{30mu}
    \diagLH{\diagIID[0][2]\diagIDI[0][1]\diagU[2][0]\diagU}
    \mspace{\intsp}
    \diagLH{\diagD[2][2]\diagD[0][2]\diagIUI[0][1]\diagUII}
    \mspace{\intsp}
    \diagLH{\diagD[3][2]\diagID[0][2]\diagUI[2][1]\diagD[0][1]\diagI[4][0]\diagIUI}
    \mspace{\intsp}
    \diagLH{\diagI[4][2]\diagIID[0][2]\diagU[3][1]\diagID[0][1]\diagUI[2][0]\diagU}
\end{gather*}
One checks that both branches rewrite to zero.

\subsubsection{Three-term rewriting steps and others}

We consider branchings involving one the three-term rewriting steps in the set of distinct-label rewriting steps.
First, consider the case where both branches are of this type.
The situation is symmetric in the $D$'s and $U$'s, so we only consider branchings in $D$'s:
\begin{gather*}
\def\intsp{30mu}
    \diagLH{\diagDI[0][4]\diagID[0][3]\diagLBOX[2][2]\diagD[0][2]\diagID[0][1]\diagDI}
    \mspace{\intsp}
    \diagLH{\diagDI[0][3]\diagID[0][2]\diagDI[0][1]\diagID}
    \mspace{\intsp}
    \diagLH{\diagID[0][3]\diagDI[0][2]\diagID[0][1]\diagDI}
    \mspace{\intsp}
    \diagLH{\diagIDI[0][4]\diagIID[0][3]\diagIDI[0][2]\diagDII[0][1]\diagIDI}
    \mspace{\intsp}
    \diagLH{\diagIDI[0][4]\diagDII[0][3]\diagIDI[0][2]\diagIID[0][1]\diagIDI}    \mspace{\intsp}
    \diagLH{\diagID[0][4]\diagDI[0][3]\diagRBOX[-1][2]\diagD[1][2]\diagDI[0][1]\diagID}
\end{gather*}
Consider the first branching.
If $\diagLH{\diagLBOX} = \diagLH{\diagU}$ , we can use the rewriting step $U_+D \to DU_+$ to slide this $U$ out of the diagram, and we recover the case $\diagLH{\diagLBOX} = \diagLH{\diagI}$ .
Moreover, the last branching rewrites into zero when $\diagLH{\diagLBOX} = \diagLH{\diagU}$ , irrespective of the branch.

When $\diagLH{\diagLBOX} = \diagLH{\diagD}$ , the first and last branchings rewrite into $D_{++}D_+D$, irrespective of the branch. The same holds for the fourth branching.

The fifth branching rewrites into $D_+DD_{++}D_+$, irrespective of the branch.

Finally, when $\diagLH{\diagLBOX} = \diagLH{\diagI}$ , the first and last branchings rewrite into $D_+D$, irrespective of the branch. The same holds for the second and third branchings.

\medbreak

Consider then branchings where one branch is a three-term rewriting step and the other branching is one the additional rewriting steps:
\begin{gather*}
\def\intsp{30mu}
    \diagLH{\diagDI[0][3]\diagID[0][2]\diagDI[0][1]\diagIU}
    \mspace{\intsp}
    \diagLH{\diagIDI[0][4]\diagIID[0][3]\diagIDI[0][2]\diagDII[0][1]\diagIUI}
    \mspace{\intsp}
    \diagLH{\diagD[0][2]\diagD[2][2]\diagIDI[0][1]\diagD\diagU[2][0]}
    \mspace{\intsp}
    \diagLH{\diagDI[0][4]\diagID[0][3]\diagLBOX[2][2]\diagD[0][2]\diagIU[0][1]\diagUI}
    \mspace{\intsp}
    \diagLH{\diagDII[0][3]\diagIDI[0][2]\diagD[0][1]\diagU[2][1]\diagIUI}
    \mspace{\intsp}
    \diagLH{\diagIDI[0][4]\diagIID[0][3]\diagLBOX[3][2]\diagID[0][2]\diagD[0][1]\diagU[2][1]\diagIUI}
    \mspace{\intsp}
    \diagLH{\diagD[0][2]\diagDI[2][2]\diagU[3][1]\diagID[0][1]\diagD\diagUI[2][0]}
    \\[1ex]
\def\intsp{30mu}
    \diagLH{\diagID[0][4]\diagDI[0][3]\diagRBOX[-1][2]\diagD[1][2]\diagDI[0][1]\diagIU}
    \mspace{\intsp}
    \diagLH{\diagIDI[0][4]\diagDII[0][3]\diagIDI[0][2]\diagIIU[0][1]\diagIUI}
    \mspace{\intsp}
    \diagLH{\diagIDI[0][4]\diagDII[0][3]\diagRBOX[-1][2]\diagDI[1][2]\diagD[0][1]\diagU[2][1]\diagIUI}
    \mspace{80mu}
    \diagLH{\diagIDI[0][4]\diagDII[0][3]\diagIUI[0][2]\diagIIU[0][1]\diagIUI}
    \mspace{\intsp}
    \diagLH{\diagDI[0][4]\diagIU[0][3]\diagLBOX[2][2]\diagU[0][2]\diagIU[0][1]\diagUI}
    \mspace{\intsp}
    \diagLH{\diagIDI[0][4]\diagU[2][3]\diagD[0][3]\diagLBOX[3][2]\diagIU[0][2]\diagIIU[0][1]\diagIUI}
    \\[1ex]
\def\intsp{30mu}
    \diagLH{\diagID[0][4]\diagDI[0][3]\diagRBOX[-1][2]\diagU[1][2]\diagUI[0][1]\diagIU}
    \mspace{\intsp}
    \diagLH{\diagDI[0][3]\diagIU[0][2]\diagUI[0][1]\diagIU}
    \mspace{\intsp}
    \diagLH{\diagIDI[0][4]\diagIIU[0][3]\diagIUI[0][2]\diagUII[0][1]\diagIUI}
    \mspace{\intsp}
    \diagLH{\diagD[0][2]\diagU[2][2]\diagIUI[0][1]\diagU\diagU[2][0]}
    \mspace{\intsp}
    \diagLH{\diagIDI[0][3]\diagU[2][2]\diagD[0][2]\diagIUI[0][1]\diagIIU}
    \mspace{\intsp}
    \diagLH{\diagIDI[0][4]\diagU[2][3]\diagD[0][3]\diagRBOX[-1][2]\diagUI[1][2]\diagUII[0][1]\diagIUI}
    \mspace{\intsp}
    \diagLH{\diagID[0][2]\diagU[3][2]\diagD[0][1]\diagUI[2][1]\diagIU\diagU[3][0]}
\end{gather*}
Since additional rewriting steps rewrite to zero, it suffices to check that the other branch rewrites to zero.
This is straightforward for branchings in the first row, as the additional rewriting step can still be applied after applying the three-term rewriting step.
The same is true for branchings in the third row.

Consider the first branching of the second row.
The case $\diagLH{\diagLBOX} = \diagLH{\diagU}$ rewrites to zero,
and when $\diagLH{\diagLBOX} = \diagLH{\diagD}$ or $\diagLH{\diagLBOX} = \diagLH{\diagI}$, we can rewrite until an additional rewriting step can be applied.
Similar arguments apply to the remaining branchings in the second row.

\subsubsection{additional rewriting steps and additional rewriting steps}

Both branches rewrite to zero, so any such branching is trivially confluent.

%% file: loopHecke.bib
@article{Bergman_DiamondLemmaRing_1978,
  title = {The Diamond Lemma for Ring Theory},
  author = {Bergman, George M.},
  date = {1978},
  journaltitle = {Advances in Mathematics},
  shortjournal = {Adv. in Math.},
  volume = {29},
  number = {2},
  pages = {178--218},
  issn = {0001-8708},
  doi = {10.1016/0001-8708(78)90010-5},
  mrnumber = {506890}
}

@article{BFM_RepresentationsLoopBraid_2019,
  title = {Representations of the Loop Braid Group and {{Aharonov-Bohm}} like Effects in Discrete (3+1)-Dimensional Higher Gauge Theory},
  author = {Bullivant, Alex and Faria Martins, João and Martin, Paul},
  date = {2019},
  journaltitle = {Advances in Theoretical and Mathematical Physics},
  shortjournal = {Adv. Theor. Math. Phys.},
  volume = {23},
  number = {7},
  pages = {1685--1769},
  issn = {1095-0761,1095-0753},
  doi = {20200529124931},
  mrnumber = {4101642}
}

@article{BH_ConfigurationSpacesRings_2013,
  title = {Configuration Spaces of Rings and Wickets},
  author = {Brendle, Tara E. and Hatcher, Allen},
  date = {2013},
  journaltitle = {Commentarii Mathematici Helvetici. A Journal of the Swiss Mathematical Society},
  shortjournal = {Comment. Math. Helv.},
  volume = {88},
  number = {1},
  pages = {131--162},
  issn = {0010-2571,1420-8946},
  doi = {10.4171/CMH/280},
  mrnumber = {3008915}
}

@article{BK_HeckeHopfAlgebras_2019,
  title = {Hecke-{{Hopf}} Algebras},
  author = {Berenstein, Arkady and Kazhdan, David},
  date = {2019-09-07},
  journaltitle = {Advances in Mathematics},
  shortjournal = {Advances in Mathematics},
  volume = {353},
  pages = {312--395},
  issn = {0001-8708},
  doi = {10.1016/j.aim.2019.06.018}
}

@article{BPT_VirtualArtinGroups_2023,
  title = {Virtual {{Artin}} Groups},
  author = {Bellingeri, Paolo and Paris, Luis and Thiel, Anne-Laure},
  date = {2023},
  journaltitle = {Proceedings of the London Mathematical Society. Third Series},
  shortjournal = {Proc. Lond. Math. Soc. (3)},
  volume = {126},
  number = {1},
  pages = {192--215},
  issn = {0024-6115,1460-244X},
  doi = {10.1112/plms.12491},
  mrnumber = {4753773}
}

@article{Buchberger_AlgorithmFindingBasis_2006,
  title = {An Algorithm for Finding the Basis Elements of the Residue Class Ring of a Zero Dimensional Polynomial Ideal},
  author = {Buchberger, Bruno},
  date = {2006},
  journaltitle = {Journal of Symbolic Computation},
  shortjournal = {J. Symbolic Comput.},
  volume = {41},
  number = {3-4},
  pages = {475--511},
  issn = {0747-7171,1095-855X},
  doi = {10.1016/j.jsc.2005.09.007},
  mrnumber = {2202562}
}

@article{Burau_UberZopfgruppenUnd_1935,
  title = {Über {{Zopfgruppen}} Und Gleichsinnig Verdrillte {{Verkettungen}}},
  author = {Burau, Werner},
  date = {1935},
  journaltitle = {Abhandlungen aus dem Mathematischen Seminar der Universität Hamburg},
  shortjournal = {Abh. Math. Sem. Univ. Hamburg},
  volume = {11},
  number = {1},
  pages = {179--186},
  issn = {0025-5858,1865-8784},
  doi = {10.1007/BF02940722},
  mrnumber = {3069652}
}

@article{BWC_ExoticStatisticsStrings_2007,
  title = {Exotic Statistics for Strings in {{4D BF}} Theory},
  author = {Baez, John C. and Wise, Derek K. and Crans, Alissa S.},
  date = {2007},
  journaltitle = {Advances in Theoretical and Mathematical Physics},
  shortjournal = {Adv. Theor. Math. Phys.},
  volume = {11},
  number = {5},
  pages = {707--749},
  issn = {1095-0761,1095-0753},
  doi = {201308301648},
  mrnumber = {2362007}
}

@unpublished{Callan_BijectionsDyckPaths_2007,
  title = {Bijections from {{Dyck}} Paths to 321-Avoiding Permutations Revisited},
  author = {Callan, David},
  date = {2007-11-16},
  eprint = {0711.2684},
  eprinttype = {arXiv},
  eprintclass = {math},
  doi = {10.48550/arXiv.0711.2684},
  pubstate = {prepublished}
}

@article{Chang_RepresentationsLoopBraid_2020,
  title = {Representations of the Loop Braid Groups from Braided Tensor Categories},
  author = {Chang, Liang},
  date = {2020},
  journaltitle = {Journal of Mathematical Physics},
  shortjournal = {J. Math. Phys.},
  volume = {61},
  number = {5},
  pages = {051702, 8},
  issn = {0022-2488,1089-7658},
  doi = {10.1063/5.0005266},
  mrnumber = {4097803}
}

@thesis{Dahm_GeneralizationBraidTheory_1962,
  type = {phdthesis},
  title = {A Generalization of Braid Theory},
  author = {Dahm, David Michael},
  date = {1962},
  institution = {ProQuest LLC, Ann Arbor, MI},
  mrnumber = {2613609},
  pagetotal = {59}
}

@article{Damiani_JourneyLoopBraid_2017,
  title = {A Journey through Loop Braid Groups},
  author = {Damiani, Celeste},
  date = {2017},
  journaltitle = {Expositiones Mathematicae},
  shortjournal = {Expo. Math.},
  volume = {35},
  number = {3},
  pages = {252--285},
  issn = {0723-0869,1878-0792},
  doi = {10.1016/j.exmath.2016.12.003},
  mrnumber = {3689901}
}

@article{DMR_GeneralisationsHeckeAlgebras_2023,
  title = {Generalisations of {{Hecke}} Algebras from Loop Braid Groups},
  author = {Damiani, Celeste and Martin, Paul and Rowell, Eric C.},
  date = {2023},
  journaltitle = {Pacific Journal of Mathematics},
  shortjournal = {Pacific J. Math.},
  volume = {323},
  number = {1},
  pages = {31--65},
  issn = {0030-8730,1945-5844},
  doi = {10.2140/pjm.2023.323.31},
  mrnumber = {4594776}
}

@article{FRR_BraidpermutationGroup_1997,
  title = {The Braid-Permutation Group},
  author = {Fenn, Roger and Rimányi, Richárd and Rourke, Colin},
  date = {1997},
  journaltitle = {Topology. An International Journal of Mathematics},
  shortjournal = {Topology},
  volume = {36},
  number = {1},
  pages = {123--135},
  issn = {0040-9383},
  doi = {10.1016/0040-9383(95)00072-0},
  mrnumber = {1410467}
}

@article{GHM_ConvergentPresentationsPolygraphic_2019,
  title = {Convergent Presentations and Polygraphic Resolutions of Associative Algebras},
  author = {Guiraud, Yves and Hoffbeck, Eric and Malbos, Philippe},
  date = {2019},
  journaltitle = {Mathematische Zeitschrift},
  shortjournal = {Math. Z.},
  volume = {293},
  number = {1-2},
  pages = {113--179},
  issn = {0025-5874,1432-1823},
  doi = {10.1007/s00209-018-2185-z},
  mrnumber = {4002273}
}

@article{GM_HigherdimensionalCategoriesFinite_2009,
  title = {Higher-Dimensional Categories with Finite Derivation Type},
  author = {Guiraud, Yves and Malbos, Philippe},
  date = {2009},
  journaltitle = {Theory and Applications of Categories},
  shortjournal = {Theory Appl. Categ.},
  volume = {22},
  pages = {No. 18, 420--478},
  issn = {1201-561X},
  mrnumber = {2559651}
}

@article{Goldsmith_TheoryMotionGroups_1981,
  title = {The Theory of Motion Groups},
  author = {Goldsmith, Deborah L.},
  date = {1981},
  journaltitle = {Michigan Mathematical Journal},
  shortjournal = {Michigan Math. J.},
  volume = {28},
  number = {1},
  pages = {3--17},
  issn = {0026-2285,1945-2365},
  doi = {10.1307/mmj/1029002454},
  mrnumber = {600411}
}

@article{Kauffman_VirtualKnotTheory_1999,
  title = {Virtual Knot Theory},
  author = {Kauffman, Louis H.},
  date = {1999},
  journaltitle = {European Journal of Combinatorics},
  shortjournal = {European J. Combin.},
  volume = {20},
  number = {7},
  pages = {663--690},
  issn = {0195-6698,1095-9971},
  doi = {10.1006/eujc.1999.0314},
  mrnumber = {1721925}
}

@article{KL_VirtualBraidsLmove_2006,
  title = {Virtual Braids and the {{L-move}}},
  author = {Kauffman, Louis H. and Lambropoulou, Sofia},
  date = {2006},
  journaltitle = {Journal of Knot Theory and its Ramifications},
  shortjournal = {J. Knot Theory Ramifications},
  volume = {15},
  number = {6},
  pages = {773--811},
  issn = {0218-2165,1793-6527},
  doi = {10.1142/S0218216506004750},
  mrnumber = {2253835}
}

@article{McCool_BasisconjugatingAutomorphismsFree_1986,
  title = {On Basis-Conjugating Automorphisms of Free Groups},
  author = {McCool, J.},
  date = {1986},
  journaltitle = {Canadian Journal of Mathematics. Journal Canadien de Mathématiques},
  shortjournal = {Canad. J. Math.},
  volume = {38},
  number = {6},
  pages = {1525--1529},
  issn = {0008-414X,1496-4279},
  doi = {10.4153/CJM-1986-073-3},
  mrnumber = {873421}
}

@article{MDD_DyckPathsRestricted_2006,
  title = {Dyck Paths and Restricted Permutations},
  author = {Mansour, Toufik and Deng, Eva Y.P. and Du, Rosena R.X.},
  date = {2006-07},
  journaltitle = {Discrete Applied Mathematics},
  shortjournal = {Discrete Applied Mathematics},
  volume = {154},
  number = {11},
  pages = {1593--1605},
  issn = {0166218X},
  doi = {10.1016/j.dam.2006.02.004},
  langid = {english}
}

@article{Mitsuhashi_SchurWeylReciprocityQuantum_2006,
  title = {Schur-{{Weyl}} Reciprocity between the Quantum Superalgebra and the {{Iwahori-Hecke}} Algebra},
  author = {Mitsuhashi, Hideo},
  date = {2006},
  journaltitle = {Algebras and Representation Theory},
  shortjournal = {Algebr. Represent. Theory},
  volume = {9},
  number = {3},
  pages = {309--322},
  issn = {1386-923X,1572-9079},
  doi = {10.1007/s10468-006-9014-5},
  mrnumber = {2251378}
}

@article{Moon_HighestWeightVectors_2003,
  author = {Moon, Dongho},
  date = {2003},
  journaltitle = {Journal of the Korean Mathematical Society},
  shortjournal = {J. Korean Math. Soc.},
  volume = {40},
  number = {1},
  pages = {1--28},
  issn = {0304-9914,2234-3008},
  doi = {10.4134/JKMS.2003.40.1.001},
  mrnumber = {1945710},
  title = {Highest Weight Vectors of Irreducible Representations of the Quantum Superalgebra \{\vphantom\}$\germ U_q({\rm gl}(m,n))$\vphantom\{\}}
}

@article{MRT_ClassificationChargeconservingLoop_2025,
  title = {Classification of Charge-Conserving Loop Braid Representations},
  author = {Martin, Paul and Rowell, Eric C. and Torzewska, Fiona},
  date = {2025},
  journaltitle = {Journal of Algebra},
  shortjournal = {J. Algebra},
  volume = {666},
  pages = {878--931},
  issn = {0021-8693,1090-266X},
  doi = {10.1016/j.jalgebra.2024.12.011},
  mrnumber = {4844775}
}

@article{Ng_HopfAlgebrasDimension_2004,
  title = {Hopf Algebras of Dimension Pq},
  author = {Ng, Siu-Hung},
  date = {2004},
  journaltitle = {Journal of Algebra},
  shortjournal = {J. Algebra},
  volume = {276},
  number = {1},
  pages = {399--406},
  issn = {0021-8693,1090-266X},
  doi = {10.1016/j.jalgebra.2003.11.008},
  mrnumber = {2054403}
}

@article{Ng_HopfAlgebrasDimension_2005,
  title = {Hopf Algebras of Dimension 2p},
  author = {Ng, Siu-Hung},
  date = {2005},
  journaltitle = {Proceedings of the American Mathematical Society},
  shortjournal = {Proc. Amer. Math. Soc.},
  volume = {133},
  number = {8},
  pages = {2237--2242},
  issn = {0002-9939,1088-6826},
  doi = {10.1090/S0002-9939-05-07804-4},
  mrnumber = {2138865}
}

@article{Reshetikhin_MultiparameterQuantumGroups_1990,
  title = {Multiparameter Quantum Groups and Twisted Quasitriangular {{Hopf}} Algebras},
  author = {Reshetikhin, N.},
  date = {1990},
  journaltitle = {Letters in Mathematical Physics},
  shortjournal = {Lett. Math. Phys.},
  volume = {20},
  number = {4},
  pages = {331--335},
  issn = {0377-9017,1573-0530},
  doi = {10.1007/BF00626530},
  mrnumber = {1077966}
}

@article{Sartori_AlexanderPolynomialQuantum_2015,
  title = {The {{Alexander}} Polynomial as Quantum Invariant of Links},
  author = {Sartori, Antonio},
  date = {2015-04},
  journaltitle = {Arkiv för Matematik},
  volume = {53},
  number = {1},
  pages = {177--202},
  publisher = {Institut Mittag-Leffler},
  issn = {0004-2080, 1871-2487},
  doi = {10.1007/s11512-014-0196-5}
}

@article{Sartori_CategorificationTensorPowers_2016,
  author = {Sartori, Antonio},
  date = {2016},
  journaltitle = {Selecta Mathematica. New Series},
  shortjournal = {Selecta Math. (N.S.)},
  volume = {22},
  number = {2},
  pages = {669--734},
  issn = {1022-1824,1420-9020},
  doi = {10.1007/s00029-015-0202-1},
  mrnumber = {3477333},
  title = {Categorification of Tensor Powers of the Vector Representation of {$U_q(\mathfrak{gl}(1|1))$}}
}

@article{Satoh_VirtualKnotPresentation_2000,
  title = {Virtual Knot Presentation of Ribbon Torus-Knots},
  author = {Satoh, Shin},
  date = {2000},
  journaltitle = {Journal of Knot Theory and its Ramifications},
  shortjournal = {J. Knot Theory Ramifications},
  volume = {9},
  number = {4},
  pages = {531--542},
  issn = {0218-2165,1793-6527},
  doi = {10.1142/S0218216500000293},
  mrnumber = {1758871}
}

@article{Savushkina_GroupConjugatingAutomorphisms_1996,
  title = {On a Group of Conjugating Automorphisms of a Free Group},
  author = {Savushkina, A. G.},
  date = {1996},
  journaltitle = {Matematicheskie Zametki},
  shortjournal = {Mat. Zametki},
  volume = {60},
  number = {1},
  pages = {92--108, 159},
  issn = {0025-567X,2305-2880},
  doi = {10.1007/BF02308881},
  mrnumber = {1431462}
}

@unpublished{Schelstraete_RewritingModuloDiagrammatic_2025,
  title = {Rewriting modulo in Diagrammatic Algebras and Application to Categorification},
  author = {Schelstraete, Léo},
  date = {2025-02-05},
  eprint = {2502.03028},
  eprinttype = {arXiv},
  eprintclass = {math},
  doi = {10.48550/arXiv.2502.03028},
  pubstate = {prepublished}
}

@incollection{Shirshov_AlgorithmicProblemsLie_2009,
  title = {Some Algorithmic Problems for {{Lie}} Algebras},
  booktitle = {Selected {{Works}} of {{A}}.{{I}}. {{Shirshov}}},
  author = {Shirshov, A. I.},
  editor = {Bokut, Leonid and Shestakov, Ivan and Latyshev, Victor and Zelmanov, Efim},
  date = {2009},
  pages = {125--130},
  publisher = {Birkhäuser},
  location = {Basel},
  doi = {10.1007/978-3-7643-8858-4_13},
  isbn = {978-3-7643-8858-4},
  langid = {english}
}

@book{Stanley_CatalanNumbers_2015,
  title = {Catalan Numbers},
  author = {Stanley, Richard P.},
  date = {2015},
  publisher = {Cambridge University Press, New York},
  doi = {10.1017/CBO9781139871495},
  isbn = {978-1-107-42774-7},
  mrnumber = {3467982},
  pagetotal = {viii+215}
}

@incollection{Vershinin_HomologyVirtualBraids_2001,
  title = {On Homology of Virtual Braids and {{Burau}} Representation},
  booktitle = {Journal of {{Knot Theory}} and Its {{Ramifications}}},
  author = {Vershinin, Vladimir V.},
  date = {2001},
  volume = {10},
  number = {5},
  pages = {795--812},
  issn = {0218-2165,1793-6527},
  doi = {10.1142/S0218216501001165},
  mrnumber = {1839703}
}

@article{Zhang_StructureRepresentationsQuantum_1998,
  title = {Structure and Representations of the Quantum General Linear Supergroup},
  author = {Zhang, R. B.},
  date = {1998},
  journaltitle = {Communications in Mathematical Physics},
  shortjournal = {Comm. Math. Phys.},
  volume = {195},
  number = {3},
  pages = {525--547},
  issn = {0010-3616,1432-0916},
  doi = {10.1007/s002200050401},
  mrnumber = {1640999}
}

@phdthesis{samson_phd,
    AUTHOR = {Black, Samson},
     TITLE = {Representations of {H}ecke algebras and the {A}lexander
              polynomial},
      school = {University of Oregon},
 PUBLISHER = {ProQuest LLC, Ann Arbor, MI},
      YEAR = {2010},
     PAGES = {50},
      ISBN = {978-1124-16381-9},
   MRCLASS = {99-05},
  MRNUMBER = {2782311},
       URL = {https://www.proquest.com/docview/748837289},
}

@software{BSM+_Magma_,
  title = {Magma},
  author = {Bosma, Wieb and Steel, Allan and Matthews, Graham and Fisher, Damien and Cannon, John and Contini, S. and Smith, Ben},
  url = {http://magma.maths.usyd.edu.au/magma/}
}

@unpublished{BQ_RemarksFaithfulnessBurau_2024,
  title = {Some Remarks about the Faithfulness of the {{Burau}} Representation of {{Artin--Tits}} Groups},
  author = {Bapat, Asilata and Queffelec, Hoel},
  date = {2024-08-30},
  eprint = {2409.00144},
  eprinttype = {arXiv},
  eprintclass = {math},
  doi = {10.48550/arXiv.2409.00144},
  pubstate = {prepublished}
}

@article{KS_QuiversFloerCohomology_2002,
  title = {Quivers, {{Floer}} Cohomology, and Braid Group Actions},
  author = {Khovanov, Mikhail and Seidel, Paul},
  date = {2002},
  journaltitle = {Journal of the American Mathematical Society},
  shortjournal = {J. Amer. Math. Soc.},
  volume = {15},
  number = {1},
  pages = {203--271},
  issn = {0894-0347,1088-6834},
  doi = {10.1090/S0894-0347-01-00374-5},
  mrnumber = {1862802}
}

@unpublished{ltv-verma-howe,
  title = {Verma {H}owe duality and {LKB} representations},
  author = {Lacabanne, Abel and Tubbenhauer, Daniel and Vaz, Pedro},
  date = {2022-07-19},
  eprint = {2207.09124},
  eprinttype = {arXiv},
  eprintclass = {math},
  doi = {10.48550/arXiv.2207.09124},
  pubstate = {prepublished}
}

@article {jimbo-q-analogue,
    AUTHOR = {Jimbo, Michio},
     TITLE = {A {$q$}-analogue of {$U(\mathfrak{gl}(N+1))$}, {H}ecke
              algebra, and the {Y}ang-{B}axter equation},
   JOURNAL = {Lett. Math. Phys.},
  FJOURNAL = {Letters in Mathematical Physics. A Journal for the Rapid
              Dissemination of Short Contributions in the Field of
              Mathematical Physics},
    VOLUME = {11},
      YEAR = {1986},
    NUMBER = {3},
     PAGES = {247--252},
      ISSN = {0377-9017},
   MRCLASS = {17B25 (58F07 82A15)},
  MRNUMBER = {841713},
MRREVIEWER = {Nick\ Yu.\ Reshetikhin},
       DOI = {10.1007/BF00400222},
       URL = {https://doi.org/10.1007/BF00400222},
}

@article {jackson-kerler,
    AUTHOR = {Jackson, Craig and Kerler, Thomas},
     TITLE = {The {L}awrence-{K}rammer-{B}igelow representations of the
              braid groups via {$U_q(\mathfrak{sl}_2)$}},
   JOURNAL = {Adv. Math.},
  FJOURNAL = {Advances in Mathematics},
    VOLUME = {228},
      YEAR = {2011},
    NUMBER = {3},
     PAGES = {1689--1717},
      ISSN = {0001-8708,1090-2082},
   MRCLASS = {20F36 (17B10 17B37 20C15 57R56)},
  MRNUMBER = {2824566},
MRREVIEWER = {Eric\ C.\ Rowell},
       DOI = {10.1016/j.aim.2011.06.027},
       URL = {https://doi.org/10.1016/j.aim.2011.06.027},
}
